\documentclass[10pt]{article}
\usepackage[ansinew]{inputenc}
\usepackage{amsmath}
\usepackage{amsfonts}
\usepackage{amssymb}
\usepackage{amsthm}
\usepackage{newlfont}
\usepackage{amsbsy}
\usepackage{bm}
\usepackage{color}
\usepackage{epsfig}
\usepackage{dsfont}
\usepackage{latexsym}
\usepackage{array}
\usepackage{longtable}
\usepackage{enumerate}
\usepackage{multicol}
\usepackage{mathrsfs}
\usepackage{wasysym}
\usepackage{graphics}
\usepackage{graphicx}

\usepackage{hyperref}
\hypersetup{
    colorlinks,
    citecolor=blue,
    filecolor=black,
    linkcolor=blue,
    urlcolor=black
}
\usepackage{bbm}
\DeclareMathOperator{\sign}{sign}

\numberwithin{equation}{section}

\newtheorem{theorem}{Theorem}[section]
\newtheorem{lemma}[theorem]{Lemma}
\newtheorem{corol}[theorem]{Corollary}

\newtheorem{prop}[theorem]{Proposition}
\newtheorem{rem}[theorem]{Remark}
\newtheorem{claim}[theorem]{Claim}

\date{}

\providecommand{\keywords}[1]
{
  \small	
  \textbf{\textit{Keywords--}} #1
}

\makeatletter
\@namedef{subjclassname@1991}{Subject}
\makeatother

\title{\textbf{Spatial decay properties for a model in shear flows posed on the cylinder}}
\author{Ricardo A. Pastr\'an\thanks{Universidad Nacional de Colombia, Bogot\'a,
E-mail: {\tt rapastranr@unal.edu.co}}\,  and Oscar Ria\~no \thanks{Universidad Nacional de Colombia, Bogot\'a,
E-mail: {\tt ogrianoc@unal.edu.co}}
}

\begin{document}

\maketitle 

\begin{abstract}  
We study spatial decay properties for solutions of the Pelinovski-Stepanyants equation posed on the cylinder. We establish the maximum polynomial decay admissible for solutions of such a model. It is verified that the equation on the cylinder propagates polynomial weights with different restrictions than the model set in $\mathbb{R}^2$.  For example, a local well-posedness theory is deduced which contains the line solitary wave provided by solutions of the Benjamin-Ono equation extended to the cylinder. Key results for our analysis are obtained from a detailed study of the dispersive effects of the linear equation in weighted Sobolev spaces. The results in this paper appear to be one of the first studies of polynomial decay for nonlocal models in the cylinder.  
\end{abstract}


{\small \textbf{\textit{Subject--}} 35Q35, 35B05, 35B60 }

\tableofcontents

%%%%%%%%%%%%%%%%%%%%%%%%%%%%%%%%%%%%%%

\section{Introduction}

This work concerns the Cauchy problem associated to the following generalization of the Pelinovsky-Stepanyants equation
\begin{equation}\label{SHeq}
\begin{cases}
  \partial_t u -\mathcal{H}_x\partial_x^2u-\mathcal{H}_x\partial_y^2u+\sum\limits_{k=1}^K \nu_k u^k\partial_x u=0,\hskip 15pt (x,y)\in \mathbb{R}\times \mathbb{T}, \,  t\in \mathbb{R}, \\
  u(x,y,0)= u_0(x,y),
  \end{cases}
\end{equation}
with $K\geq 1$ being an integer number, constants $\nu_k$ not all null, where $\mathcal{H}_x$ denotes the Hilbert transform operator acting only on the $x$-variable, which for a sufficiently regular function $\phi(x,y)$, $x\in \mathbb{R}$, $y\in \mathbb{T}$ is defined as
\begin{equation*}
\mathcal{H}_x\phi(x,y)=\frac{1}{\pi}\textit{p.v.}\int_{\mathbb{R}}\frac{\phi(z,y)}{x-z}\, dz,  
\end{equation*}
or equivalently, by using the Fourier transform, $\widehat{\mathcal{H}_x\phi}(\xi,\eta)=-i\sign(\xi)\widehat{\phi}(\xi,\eta)$, $\xi \in \mathbb{R}$, $\eta \in \mathbb{Z}$.

When $K=1$, $\nu_1\neq 0$, i.e., \eqref{SHeq} with nonlinearity $\nu_1 (u \partial_x u)$, the equation in \eqref{SHeq} appears to describe two-dimensional weakly nonlinear long-wave perturbation on the background of a boundary-layer type plane-parallel shear flow, we refer to  \cite{PelinovskyShrira1995,PelinovskyStepanyants1994,AbramyanStepanyantsShrira1992,DyachenkoKuznetsov1994}. Originally, the model in \eqref{SHeq} was introduced by Pelinovsky and Stepanyants \cite{PelinovskyStepanyants1994} as a simplification of the \emph{Shrira equation}:
\begin{align}\label{E:HBO}
\partial_t u- \mathcal{R}_1 \partial_x^2 u- \mathcal{R}_1 \partial_y^2 u+  u\partial_x u=0, \qquad (x,y)\in \mathbb{R}^2, \, \, t\in \mathbb{R},
\end{align}
where $\mathcal{R}_1$ denotes the Riesz transform operator, which is defined via the singular integral by 
\begin{equation*}
\mathcal{R}_1 f%\phi
(x,y)=\frac{1}{2\pi} \, \textit{p.v.} \int \frac{(x-z_1) \, %\phi
f(z_1,z_2)}{\big((x-z_1)^2+(y-z_2)^2 \big)^{3/2}}\, dz_1 \, dz_2.
\end{equation*}
Shrira \cite{Shrira1989} was one of the first to mention \eqref{E:HBO} in the analysis of 2D shear flows. For recent investigation of \eqref{E:HBO}, we refer to \cite{HickmanLinaresRiano2019,RianoRoudenko2022,OscSvetKai2021,Riano2020}. In the context of stability of nonlinear waves of \eqref{E:HBO}, in \cite{PelinovskyStepanyants1994}, equation in \eqref{SHeq} is proposed as a model to study the spectrum and eigenfunctions of a simplified eigenvalue problem associated with \eqref{E:HBO}. For more details, see Sections 2 and 3 of \cite{PelinovskyStepanyants1994}. On a similar note, in \cite{PelinovskyShrira1995}, equation in \eqref{SHeq} appears when considering almost one-dimension waves in \eqref{E:HBO} (formally, it is assumed $|\partial_y^2 u|\ll |\partial_x^2 u|$), then in its dimensionless form, one obtains that \eqref{E:HBO} reduces to the equation in \eqref{SHeq}, which is the object of study of the present work. 

The model in \eqref{SHeq} also serves as a two-dimensional extension of the Benjamin-Ono equation (see, \cite{IfrimTata2019,KenigIonescu2007,KillipLaurensVisan2023,MolinetPilodBO2012,Ponce2016,RoudenkoWangYang2021,Saut2018,Tao2004} and reference therein):
\begin{equation}\label{BO}
    \partial_tu -\mathcal{H}_x\partial_x^2u +\sum\limits_{k=1}^K \nu_k u^k\partial_x u=0, \qquad x\in \mathbb{R}, \, t\in \mathbb{R}.
\end{equation}
Concerning invariants of the equation, real solutions of the IVP \eqref{SHeq} formally conserve at least the following quantities 
\begin{align} 
    M[u]&=\int_{\mathbb{R}\times \mathbb{T}} u^2(x,y,t)\, dx dy, \label{mass}\\
    E[u]&=\frac{1}{2} \int_{\mathbb{R}\times \mathbb{T}}\Big( |D_x^{\frac{1}{2}}u(x,y,t)|^2+ |D_x^{-\frac{1}{2}}\partial_y u(x,y,t)|^2-\sum_{k=1}^K\frac{2\nu_k}{(k+1)(k+2)}u^{k+2}(x,y,t)\Big) \, dx dy, \label{Energy}
\end{align}
 where $D_x^{\pm \frac{1}{2}}$ is the fractional derivative operator in the $x$-variable, which is defined through the Fourier transform as $\mathcal{F}(D_x^{\pm \frac{1}{2}}u)(\xi,\eta)=|\xi|^{\pm \frac{1}{2}}\widehat{u}(\xi,\eta)$.

In this paper, we seek to study the maximum spatial decay measured by a fractional scale $\theta>0$ in the space $L^2(|x|^{2\theta}\, dx dy)$, such that initial conditions in such space generate solutions of the Cauchy problem \eqref{SHeq} that persist in the same space.  Such questions are of interest to measure the effects of periodic transversal perturbations on the existence and behavior of solutions of dispersive partial differential equations. In addition to deducing that solutions of \eqref{SHeq} do not propagate weights of arbitrary size and that such restrictions are strongly influenced by the presence of the periodic variable, we will see that the propagation of weights for the equation \eqref{SHeq} set on the cylinder differs from the behavior of the equation posed in $\mathbb{R}^2$ (cf., \cite{Riano2021}). Note that in the periodic case, where the domain is restricted to an interval, the propagation of polynomial weights in the spatial variable $y$ does not lead to interesting results.

Studies of persistence in weighted spaces such as the ones presented here have been investigated for different models, having their origins in the Benjamin-Ono equation, see \cite{FonsecaPonce2011,FonsecaLinaresPonce2012,Iorio1986,Iorio1991,FonsecaLinaresPonce2013}, and for higher dimensional models, see \cite{CunhaPastor2021,CunhaPastor2014,Riano2020,Riano2021} and the references therein. However, this work seems to be the first one addressing the propagation of fractional polynomial weights for a nonlocal model in the cylinder.

Concerning studies of the Cauchy problem \eqref{SHeq} defined on spatial variables $\mathbb{R}^2$ and $\mathbb{T}^2$,  we refer to \cite{Schippa2020,BustamanteJimenezMejia2019II,Nascimento2023,EsfahaniPastor2018}. By well-posedness of a given initial value problem, we mean the existence, uniqueness, and continuous dependence of the map data-solution. In this regard, in \cite{Schippa2020,Riano2021,BustamanteJimenezMejia2019} the local well-posedness for \eqref{SHeq} with nonlinearity $u\partial_x u$ has been established in $H^s(\mathbb{K}^2)$, $s>\frac{3}{2}$, in the cases $\mathbb{K}=\mathbb{R}$ and $\mathbb{K}=\mathbb{T}$. We remark that in \cite{Riano2021}, it was established local well-posedness in a space $X^s(\mathbb{K}^2)$, $s>\frac{3}{2}$, $\mathbb{K}=\mathbb{R}$ and $\mathbb{K}=\mathbb{T}$, which is adapted to the energy of the equation \eqref{Energy}. However, there are few studies on the Cauchy problem \eqref{SHeq} posed on the cylinder $\mathbb{R}\times \mathbb{T}$. In general, for the combined nonlinearity cases $K\geq 2$ in \eqref{SHeq}, it seems that many questions including minimum regularity in $H^s$ or in anisotropic spaces $H^{s_1,s_2}$ seem to be open for investigation for any of the settings $\mathbb{R}^2$, $\mathbb{T}^2$, and $\mathbb{R}\times \mathbb{T}$.

 Furthermore, our motivation also comes from the study of \emph{solitary waves} and \emph{line solitary waves} solutions of \eqref{SHeq}.  Since the most studied case is $K=1$ and $\nu_1=1$, i.e., $u\partial_x u$, let us focus our discussion on this nonlinearity.  By a solitary wave, we mean a non-trivial solution of \eqref{SHeq} of the form $u(x,y,t)=Q_c(x-ct,y)$, where $c>0$, $x\in \mathbb{R}$, $y\in \mathbb{T}$, we have that $Q_c(x,y)=cQ(cx,y)$ is a solution of
\begin{equation}\label{Soleq}
    \mathcal{H}_x\partial_x Q+\partial_x^{-1}\mathcal{H}_x\partial_y^2Q+Q-\frac{1}{2}Q^2=0.
\end{equation}
For a study of existence of solutions for \eqref{Soleq} with variables $x,y\in \mathbb{R}$, we refer to \cite{EsfahaniPastor2018}. Concerning line solitary waves of \eqref{SHeq}, we recall that the Benjamin-Ono equation \eqref{BO} with $u\partial_x u$ has solitary wave solutions $u(x,t)=Q_{c,\text{BO}}(x-ct)$, where $c>0$, and $Q_{c,\text{BO}}(x)=cQ_{\text{BO}}(cx)$ solves
\begin{equation}\label{grStateBO}
    \mathcal{H}_x\partial_x Q_{\text{BO}}+Q_{\text{BO}}-\frac{1}{2}Q^2_{\text{BO}}=0.
\end{equation}
It is known, see \cite{Albert1994,AmickToland1991,Benjamin1967}, that there is a unique (up to translations) solution of \eqref{grStateBO}, which is
\begin{equation*}
    Q_{\text{BO}}(x)=\frac{4}{1+x^2}.
\end{equation*}
Consequently, setting $\widetilde{Q}_c(x,y)=Q_{c,\text{BO}}(x)$, $x\in \mathbb{R}$, $y\in \mathbb{T}$, one can regard soliton solutions of \eqref{grStateBO} as a line solitary wave of \eqref{Soleq}. A natural question then is the study of the stability/instability of line solitary waves $\widetilde{Q}_c(x,y)$ with respect to periodic perturbations in the transversal direction. As far as we know,  there are no rigorous proofs of such results. In this paper, we do not pursue the question of stability and instability. Here we seek to analyze the maximum polynomial decay propagated by solutions of \eqref{SHeq} in the cylinder. Such a question is motivated by the spatial decay of the solitary wave and the line solitary wave introduced above, which shows that it is natural to analyze the Cauchy problem \eqref{SHeq} in Sobolev spaces with polynomial weights. In a way, one can expect some difference between the initial value problem \eqref{SHeq} in the cylinder case and the real case as the line soliton generates a solution of \eqref{SHeq} in $\mathbb{R}\times \mathbb{T}$, but this is not an initial datum admissible for the polynomial decay theory in $L^2(\mathbb{R}^2)$ (see, \cite{Riano2021}). 

\subsection{Main results} 

We will show that dispersive effects are primarily responsible for determining spatial decay properties of solutions of \eqref{SHeq}. For this, we first detail all the admissible polynomial decay propagated by solutions of the linear equation associated with \eqref{SHeq}. We denote by $\{U(t)\}$ the unitary group of solutions determined by the linear equation $\partial_t v-\mathcal{H}_x\partial_x^2v-\mathcal{H}_x\partial_y^2v=0$ (for an explicit definition, see \eqref{Group} below). Our first result establishes a complete study of the fractional polynomial decay for the group $\{U(t)\}$ in the cylinder.  Above, given $s_1,s_2\in \mathbb{R}$, $H^{s_1,s_2}(\mathbb{R}\times \mathbb{T})$ denotes the usual anisotropic Sobolev space, for a precise definition see Section \ref{Sectiprel} below.

\begin{theorem}\label{LinearEst}

(i) Let $0<\theta <\frac{1}{2}$. Assume $f\in H^{\theta,0}(\mathbb{R}\times \mathbb{T})\cap L^2(|x|^{2\theta}\, dx dy)$, then
\begin{equation}\label{weightineq}
    \|\langle x\rangle^{\theta} U(t)f\|_{L^2}\lesssim \langle t \rangle^{\theta}\big(\|\langle x\rangle^{\theta} f\|_{L^2}+\|f\|_{H^{\theta,0}}\big).
\end{equation}
(ii) If $\theta=\frac{1}{2}$, assume $f\in H^{\frac{1}{2},0}(\mathbb{R}\times \mathbb{T})\cap L^2(|x|\, dx dy)$ and  $\mathcal{H}_xf \in L^2(|x|\, dx dy)$, then it follows
\begin{equation}\label{weightineq2}
    \|\langle x\rangle^{\frac{1}{2}} U(t)f\|_{L^2}\lesssim \langle t \rangle^{\frac{1}{2}}\big(\|\langle x\rangle^{\frac{1}{2}} f\|_{L^2}+\|\langle x\rangle^{\frac{1}{2}}\mathcal{H}_x f\|_{L^2}+\|f\|_{H^{\frac{1}{2},0}}\big).
\end{equation}
(iii) Let $\frac{1}{2}<\theta<\frac{3}{2}$, $f\in H^{\theta,0}(\mathbb{R}\times \mathbb{T})\cap L^2(|x|^{2\theta}\, dx dy)$. Assume that 
\begin{equation}\label{hiphote1}
\int_{\mathbb{T}}\int_{\mathbb{R}}e^{-i\eta y}f(x,y)\, dxdy =0\, \, \text{ for all } \eta\in \mathbb{Z}\setminus\{0\}.
\end{equation}
Then \eqref{weightineq} is also valid in this case.
\\ \\
(iv) Let $k+\frac{1}{2}\leq \theta<k+\frac{3}{2}$, $k\in \mathbb{Z}^{+}$, $f\in H^{\theta,0}(\mathbb{R}\times \mathbb{T})\cap L^2(|x|^{2\theta}\, dx dy)$.
Assume that 
\begin{equation}\label{hiphote2}
 \int_{\mathbb{T}}\int_{\mathbb{R}}e^{-i\eta y}x^{l} f(x,y)\, dxdy =0\, \text{ for all } \eta\in\mathbb{Z}\setminus\{0\}, \text{ and } l=0,1,\dots, k-1. 
\end{equation}
Additionally, if $k\geq 2$,  assume that
\begin{equation}\label{hiphote2.1}
 \int_{\mathbb{T}}\int_{\mathbb{R}}x^{l} f(x,y)\, dxdy =0\, \text{ for all } \, \, l=0,1,\dots, k-2. 
\end{equation}
If $\theta=k+\frac{1}{2}$, in addition to \eqref{hiphote2} and \eqref{hiphote2.1}  assume further that 
\begin{equation}\label{hiphote3}
\mathcal{H}_x (x^{k} f)\in L^2(|x|\, dx).
\end{equation}
Then
 \begin{equation}
    \|\langle x\rangle^{k+\frac{1}{2}} U(t)f\|_{L^2}\lesssim \langle t \rangle^{k+\frac{1}{2}}\big(\|\langle x\rangle^{k+\frac{1}{2}} f\|_{L^2}+\|\langle x\rangle^{\frac{1}{2}}\mathcal{H}_x(x^k f)\|_{L^2}+\|f\|_{H^{k+\frac{1}{2},0}}\big).
\end{equation}   
If $k+\frac{1}{2}<\theta<k+\frac{3}{2}$ in addition to \eqref{hiphote2} and \eqref{hiphote2.1}  assume that 
\begin{equation}\label{hiphote4}
 \int_{\mathbb{T}}\int_{\mathbb{R}}e^{-i\eta y}x^{k} f(x,y)\, dxdy =0\, \text{ for all } \eta\in \mathbb{Z}\setminus\{0\}. 
\end{equation}
Then \eqref{weightineq} is also valid in this case.

\end{theorem}

\begin{rem}
(i) Let us explain the relationship between the dispersions in \eqref{SHeq} and the hypothesis \eqref{hiphote1}, \eqref{hiphote2},  \eqref{hiphote2.1},  \eqref{hiphote3} and  \eqref{hiphote4}. The conditions \eqref{hiphote1}, \eqref{hiphote2},  \eqref{hiphote3} and  \eqref{hiphote4} are due to the dispersive term $\mathcal{H}_x\partial_y^2$. This can be seen from the cases where  $\mathcal{H}_x$  establishes a bounded operator in $L^2(|x|^{2\theta}\,dx)$. For example, it is known that the Hilbert transform (see \cite{HuntMuckenhouptWheeden1973,Petermichl2007}) is a bounded operator in such a space whenever $0<\theta<\frac{1}{2}$. This partially justifies why we do not need conditions in (i) Theorem \ref{LinearEst}. When $\theta=\frac{1}{2}$, we cannot deduce the boundedness of the Hilbert transform, so we need to assume $\mathcal{H}_xf\in L^2(|x|\, dx)$ (see \eqref{hiphote3}). Nevertheless, \cite[Theorem 4.3]{Yafaev1999} shows that the hypothesis $\mathcal{H}_x f\in L^2(|x|\, dx dy)$ can be replaced by the assumption that for all $\eta$ the map $\xi \mapsto \widehat{f}(\xi,\eta)$ belongs to the $L^2(\mathbb{R})$-closure of the space of square-integrable continuous odd functions. Finally, when $\frac{1}{2}<\theta<1$, to have that  $\mathcal{H}_x f\in L^2(|x|^{2\theta}\, dx)$, we impose the zero mean condition on $f$ (see Proposition \ref{decaywight} below, and \cite[Proposition 3.1]{Yafaev1999}), which in the periodic setting are \eqref{hiphote1}, \eqref{hiphote2} and \eqref{hiphote4}.
\\ \\ 
On the other hand, the condition \eqref{hiphote2.1} comes from the dispersive term $\mathcal{H}_x\partial_x^2$, which is the symbol of the Benjamin-Ono equation \eqref{BO}. To see this, we compare with the results in \cite{FonsecaPonce2011},  which show that when the weight is greater or equal than $\frac{5}{2}$, the propagation of such weights for BO holds assuming the one-dimensional version of \eqref{hiphote2.1} (i.e., without integration in the periodic variable). We give a more detailed description of the previous discussion in Table \ref{table1} below.
\\ \\
(ii) The theory in weighted spaces for solutions of \eqref{SHeq} posed on the cylinder presents different results from those in the case of the equation set in $\mathbb{R}^2$. For example, we see that the line solitary wave provided from the Benjamin-Ono, $\widetilde{Q}(x,y)=Q_{BO}(x)\in L^2(|x|^{2\theta}\, dx dy)$ for all $0<\theta<\frac{3}{2}$, and it satisfies \eqref{hiphote1}. Thus, Theorem \ref{LinearEst} shows that $U(t)(\widetilde{Q}(x,y))\in L^2(|x|^{2\theta}\, dx dy)$. However, the results for the equation posed on the Euclidean plane in \cite{Riano2021} show that $U(t)(Q_{BO}(x)\phi(y))\notin L^2(|x|^{2\theta}\, dx dy)$ for all $\frac{1}{2}<\theta<\frac{3}{2}$, and all $\phi(y)$ be a smooth function compactly supported with $\int_{\mathbb{R}} \phi(y)\, dy \neq 0$. Therefore, the theory in the cylinder propagates weights that are not admissible in the $\mathbb{R}^2$ case.
\end{rem}

\begin{table}[!ht]
    \centering
\begin{tabular}{ |c|c|c| } 
 \hline
{\bf Size of the weight $\theta$ } & {\bf Condition due}  & {\bf  Condition due}  \\
 {\bf in space $L^2(|x|^{2\theta}\, dxdy)$} & {\bf to dispersive term $\mathcal{H}_x\partial_y^2$} & {\bf to dispersive term $\mathcal{H}_x\partial_x^2$} \\ \hline
 $0<\theta<\frac{1}{2}$ & No condition & No condition \\ 
 $\theta=\frac{1}{2}$ & Extra condition $\mathcal{H}_xf \in L^2(|x|\, dx dy)$  & No condition \\  
$\frac{1}{2}<\theta<\frac{3}{2}$ &  \eqref{hiphote1}  & No condition \\ 
$\theta=\frac{3}{2}$ &  \eqref{hiphote3} with $k=1$  & No condition \\ 
$\frac{3}{2}<\theta<\frac{5}{2}$ &  \eqref{hiphote2} and \eqref{hiphote4}  & No condition \\ 
$\theta=\frac{5}{2}$ &  \eqref{hiphote3} with $k=2$  & \eqref{hiphote2.1} with $k=2$ \\ 
$\frac{5}{2}<\theta<\frac{7}{2}$ &  \eqref{hiphote2} and \eqref{hiphote4} with $k=2$  & \eqref{hiphote2.1} with $k=2$ \\ 
$\theta=k+\frac{1}{2}$, $k\geq 2$ &  \eqref{hiphote2} and \eqref{hiphote3}  & \eqref{hiphote2.1}  \\ 
$k+\frac{1}{2}<\theta<k+1$, $k\geq 2$ &  \eqref{hiphote2} and \eqref{hiphote4}   & \eqref{hiphote2.1} \\ 
 \hline 
\end{tabular}
     \caption{Dispersive effects connection with hypothesis in Theorem \ref{LinearEst}.}
\label{table1}
\end{table}

A natural question now is whether conditions \eqref{hiphote1}, \eqref{hiphote2}, \eqref{hiphote2.1} and \eqref{hiphote4} are proper to the dispersive effects in \eqref{SHeq}. In the case where the weight $\theta\neq k+\frac{1}{2}$, for each $k\in \mathbb{Z}^{+}\cup\{0\}$, the following theorem establishes that the conditions in Theorem \ref{LinearEst} are optimal in order to have that at two different times $t_1\neq t_2$, $U(t_1)f, U(t_2)f\in H^{\theta,0}(\mathbb{R}\times \mathbb{T})\cap L^2(|x|^{2\theta}\, dx dy)$ (for simplicity in the presentation of the theorem below, we have set $t_1=0$).

\begin{theorem}\label{lineaerequniquecont1}
Let $k\geq 0$ be an integer, $\frac{1}{2}+k<\theta<k+\frac{3}{2}$, and $f\in H^{\theta,0}(\mathbb{R}\times \mathbb{Z})\cap L^2(|x|^{2\theta}\, dxdy)$. Assume that for some time $t\neq 0$, $U(t)f\in L^2(|x|^{2\theta}\, dxdy)$. 
\begin{itemize}
    \item[(i)]  If $k=0,1$, it follows that $f$ must satisfy
\begin{equation}\label{hiphoteunique1}
 \int_{\mathbb{T}}\int_{\mathbb{R}}e^{-i\eta y}x^{l} f(x,y)\, dxdy =0, 
\end{equation}
for all $l=0,1,\dots,k$, and all $\eta\in \mathbb{Z}$ whose square times $|t|$ is not a nonnull multiple of $2\pi$, i.e., $|t|\eta^2\neq 2\pi l'$ for each $l'\in \mathbb{Z}$ with $l'\geq 0$.

\item[(ii)] If $k\geq 2$, it follows that $f$ must satisfy
\begin{equation}\label{hiphoteunique2}
 \int_{\mathbb{T}}\int_{\mathbb{R}}e^{-i\eta y}x^{l} f(x,y)\, dxdy =0, 
\end{equation}
 for all $l=0,\dots,k-2$ and all $\eta\in \mathbb{Z}$, and   
\begin{equation}\label{hiphoteunique3}
 \int_{\mathbb{T}}\int_{\mathbb{R}}e^{-i\eta y}x^{l} f(x,y)\, dxdy =0, 
\end{equation}
for all $l=k-1,k$, and all $\eta\in \mathbb{Z}$ whose square times $|t|$ is not a nonnull multiple of $2\pi$, i.e., $|t|\eta^2\neq 2\pi l'$ for each $l'\in \mathbb{Z}$ with $l'\geq 0$.
\end{itemize}
\end{theorem}

When we study the operator $U(t)$ in the space $L^{2}(|x|^{2(k+\frac{1}{2})}\, dx dy)$ for some integer $k\in \mathbb{Z}^{+}\cup\{0\}$, it is not clear whether the condition \eqref{hiphote3} in Theorem \ref{LinearEst} is optimal. This is an inherent difficulty from the fact that $\mathcal{H}_x$ is not a bounded operator in $L^{2}(|x|^{2(k+\frac{1}{2})}\, dx dy)$. However, using the results of Theorem \ref{lineaerequniquecont1} and further estimates of the linear part (see Lemma \ref{uniquecontcritcase}), we deduce the following result.

\begin{corol}\label{lineaerequniquecont2}
Let $k\geq 0$ be an integer, and $f\in H^{k+\frac{1}{2},0}(\mathbb{R}\times \mathbb{Z})\cap L^2(|x|^{2(k+\frac{1}{2})}\, dx dy)$. Assume that for some fixed time $t\neq 0$, $U(t)f\in L^2(|x|^{2(k+\frac{1}{2})}\,dx dy)$. 
\begin{itemize}
    \item[(i)]  If $k=0,1$, it follows that
\begin{equation}\label{xkuniquecont0}
 U(t)\mathcal{H}_x(x^k f)\in L^2(|x|\, dx dy)\qquad \text{ if and only if } \qquad  \mathcal{H}_x(x^k f)\in L^2(|x|\, dx dy).
\end{equation}

\item[(ii)] If $k\geq 2$, it follows that
\begin{equation}\label{xkuniquecont}
\begin{aligned}
   U(t)(x^kf)\in L^2(|x|\, dx dy) \qquad \text{ if and only if } \qquad \int_{\mathbb{T}} \int_{\mathbb{R}} e^{-i\eta y}x^{k-2} f(x,y)\, dx dy=0,
\end{aligned}    
\end{equation}
for all $\eta\in \mathbb{Z}$. Moreover, assuming either of the conditions in \eqref{xkuniquecont}, it follows that 
\begin{equation}\label{xkuniquecont1}
 U(t)\mathcal{H}_x(x^k f)\in L^2(|x|\, dx dy)\qquad \text{ if and only if } \qquad  \mathcal{H}_x(x^{k} f)\in L^2(|x|\, dx dy).
\end{equation}
\end{itemize}
\end{corol}

\begin{rem}
We note that the conditions given in Theorem \ref{lineaerequniquecont1} are less restrictive than those in Theorem \ref{LinearEst} as they only need $\eta\in \mathbb{Z}$ to be such  $|t|\eta^2\neq 2\pi l'$ for each $l'\in \mathbb{Z}$, $l'\geq 0$. Such a conclusion holds since the time $t$ in this theorem is fixed. In contrast, in Theorem \ref{LinearEst}, we have stated the conditions in such a way that they do not depend on any particular time $t\in \mathbb{R}$. Nevertheless, for a fixed time $t\neq 0$,  the results of Theorem \ref{LinearEst} are valid assuming that \eqref{hiphote1}, \eqref{hiphote2} and \eqref{hiphote4} hold for all $\eta\in \mathbb{Z}$ such that $|t|\eta^2\neq 2\pi l'$ for each $l'\in \mathbb{Z}$, $l'\geq 0$. Note that \eqref{hiphote2.1} and \eqref{hiphote3} remain the same.
\end{rem}

Next, we present our local existence result for solutions of \eqref{SHeq} in weighted spaces.

\begin{theorem}\label{LWPresultinWS}

Let $K\geq 1$ be integer, and $\nu_k\in \mathbb{R}$, $k=1,\dots,K$ be constants not all null. Let $s_1>\frac{3}{2}$, $s_2>\frac{1}{2}$ be fixed.
\begin{itemize}
    \item[(i)] Let $0<\theta<\frac{1}{2}$. Then the Cauchy problem \eqref{SHeq} is locally well-posed in $H^{s_1,s_2}(\mathbb{R}\times \mathbb{T})\cap L^2(|x|^{2\theta}\, dx dy)$.

\item[(ii)] The Cauchy problem \eqref{SHeq} is locally well-posed in the space
$$L^2(|x|\, dx dy)\cap \{f\in H^{s_1,s_2}(\mathbb{R}\times \mathbb{T}): \mathcal{H}_x f \in L^2(|x|\, dx dy)\}.$$

\item[(iii)] Let $\frac{1}{2}<\theta<\frac{3}{2}$. Then the Cauchy problem \eqref{SHeq} is locally well-posed in 
$$L^2(|x|^{2\theta}\, dx dy)\cap \{f\in H^{s_1,s_2}(\mathbb{R}\times \mathbb{T}): \int_{\mathbb{T}}\int_{\mathbb{R}}e^{-iy \eta}f(x,y)\, dy=0, \, \, \text{for all } \eta \in \mathbb{Z}\}.$$
\end{itemize}
\end{theorem}

The proof of Theorem \ref{LWPresultinWS} is based on energy estimates, where we use the symmetry of the equation and properties of the Hilbert transform to simplify the effects of the periodic variable $y$. A key idea in our arguments is to prove the persistence of polynomial decay for smooth solutions of the equation \eqref{SHeq}, see Lemma \ref{smoothsolutdecay}. Such a result depends on our linear estimates deduced on Theorem \ref{LinearEst}. 

\begin{rem}
  Motivated by the requirements between decay and regularity for solutions of the linear equation, see Theorem \ref{LinearEst}, one expects that all the results in Theorem \ref{LWPresultinWS} could be extended to regularities $s_1\geq \theta$, and $s_2\geq 0$. In this paper, we do not pursue the minimal regularity measured in $s_1$ and $s_2$ to propagate polynomial weights. What we seek is to develop a detailed study of the linear equation, which ultimately determines the maximum polynomial decay admitted by the solutions of \eqref{SHeq}. The restrictions $s_1>\frac{3}{2}$, and $s_2> \frac{1}{2}$ in Theorem \ref{LWPresultinWS} come from the standard well-posedness results in $H^{s_1,s_2}(\mathbb{R}\times \mathbb{T})$. However, our proof of Theorem \ref{LWPresultinWS} is valid for regularities $s_1\geq \max\{\theta,\frac{1}{2}\}$, and $s_2\geq 0$ provided that there exists a local theory of solutions of \eqref{SHeq} in  $H^{s_1,s_2}(\mathbb{R}\times \mathbb{T})$ for which the solutions also satisfy
  \begin{equation*}
      u, \partial_x u\in L^{k}([0,T];L^{\infty}(\mathbb{R}\times \mathbb{T})),
  \end{equation*}
  for all $k=1,\dots,K$. 
\end{rem}

Next, we present our asymptotic at infinity unique continuation principle for solutions of \eqref{SHeq}.  For more details on these types of questions, see \cite{LinaresPonce2021}.

\begin{theorem}\label{Uniquecontprinc}
Let $K\geq 1$ be integer, and $\nu_k\in \mathbb{R}$, $k=1,\dots,K$ be constants not all null.  
\begin{itemize}
    \item[(i)] Let $\frac{1}{4}<\theta<\frac{1}{2}$, $s_1\geq \frac{8\theta+2}{4\theta-1}$,  $s_2> \frac{6\theta}{4\theta-1}$. Let $u\in C([0,T];H^{s_1,s_2}(\mathbb{R}\times \mathbb{T}))\cap C([0,T];L^2(|x|^{2\theta}\, dxdy))$ be a solution of \eqref{SHeq} such that there exist two different times $t_2>t_1$ such that
    \begin{equation*}
        u(t_1), u(t_2)\in L^2(|x|^{1^{+}}\, dx dy). 
    \end{equation*}
Then 
\begin{equation}\label{Uniqueident1}
\int_{\mathbb{T}}\int_{\mathbb{R}} e^{-i\eta y}u(x,y,t)\, dx dy=0,    
\end{equation}
for all $t\in[t_1,T]$, $\eta\in \mathbb{Z}$ be such that $|t_2-t_1|\eta^{2}\neq 2\pi l'$ for each $l'\in \mathbb{Z}$ with $l'\geq 0$.   

  \item[(ii)] Let $1<\theta<\frac{3}{2}$, $s_1\geq \frac{8\theta+6}{4\theta-3}$,  $s_2> \frac{6\theta}{4\theta-3}$. Let $u\in C([0,T];H^{s_1,s_2}(\mathbb{R}\times \mathbb{T}))\cap C([0,T];L^2(|x|^{2\theta}\, dxdy))$ be a solution of \eqref{SHeq} such that
  \begin{equation*}
    \int_{\mathbb{T}}\int_{\mathbb{R}}e^{-iy \eta}u(x,y,t)\,dx dy=0, \, \, \text{for all } \eta \in \mathbb{Z}, \, t\in[0,T],  
  \end{equation*}
  and such that  there exist two different times $t_2>t_1$ for which
    \begin{equation*}
        u(t_1), u(t_2)\in L^2(|x|^{3^{+}}\, dx dy). 
    \end{equation*}
Then 
\begin{equation}\label{Uniqueident2}
\begin{aligned}
    \sin((t_2-t_1)\eta^2)\frac{\partial}{\partial \xi}\widehat{u}(0,\eta,t_1)=\int_0^{t_2-t_1}\sin((t_2-t_1-\tau)\eta^2)\Big(\sum_{k=1}^K 
    \frac{i \nu_k}{k+1}\widehat{u^{k+1}}(0,\eta,\tau)\Big)\, d\tau,
\end{aligned}    
\end{equation}
for all $\eta\in \mathbb{Z}$. In particular, if for some $0<T_1\leq T$, $u\in L^{\infty}([0,T_1];L^2(|x|^{3^{+}}\, dx dy))$ with $u(0)\in L^2(|x|^{3+}\,dxdy)$, then
\begin{equation}\label{contraconlcus}
\int_{\mathbb{T}}\int_{\mathbb{R}} e^{-i\eta y} x u(x,y,0)\, dx dy=0,    
\end{equation}
for all $\eta\in \mathbb{Z}\setminus\{0\}$.
\end{itemize}
\end{theorem}

The above theorem is a consequence of our studies for the linear equation in Theorems \ref{LinearEst} and \ref{lineaerequniquecont2}. In a way, it is observed that dispersion has a strong influence in establishing restrictions to propagate certain spatial decay for solutions of the nonlinear equation \eqref{SHeq}. Furthermore, we see that due to the influence of the Hilbert transform and its iteration with the periodic variable, the results in Theorem \ref{Uniquecontprinc} give us identities \eqref{Uniqueident1} and \eqref{Uniqueident2}, which are more restrictive than in the case of the Benjamin-Ono equation, cf. \cite{FonsecaPonce2011}. However, we can draw some conclusions from our unique continuation principles.

\begin{rem}
\begin{itemize}
    \item[(i)] Theorem \ref{Uniquecontprinc} (i) establishes that for arbitrary initial data in $u_0\in H^{s_1,s_2}(\mathbb{R}\times \mathbb{Z})\cap L^2(|x|^{1^{+}}\, dx dy)$ (with $s_1$, $s_2$ be given in Theorem \ref{Uniquecontprinc} (i)), the polynomial decay $|x|^{\frac{1}{2}^{-}}$ is the largest possible decay for solutions of \eqref{SHeq} in $L^2$-spaces. To see this, let $u_0\in H^{s_1,s_2}(\mathbb{R}\times \mathbb{Z})\cap L^2(|x|^{1^{+}}\, dx dy)$ be such that $\widehat{u_0}(0,\eta)\neq 0$ for some $\eta$. By Theorem \ref{LWPresultinWS} there exist a time $T>0$, and a unique solution of \eqref{SHeq} with initial condition $u_0$ such that $u\in C([0,T];L^2(|x|^{1^{-}}\, dx dy))$, but it can not happen that for some $0<T_1\leq T$, $u\in L^{\infty}([0,T_1];L^2(|x|^{1^{+}}\, dx dy))$. Otherwise, setting $t_1=0$ in \eqref{Uniqueident1} and using a continuity argument with $t_2\to 0$, it would follow that $\widehat{u_0}(0,\eta)=0$. Consequently, Theorem \ref{Uniquecontprinc} (i) shows that the conditions in Theorem \ref{LWPresultinWS} (iii) are required to propagate weights of magnitude $\theta>\frac{1}{2}$. However, we do not know if the condition in the space in Theorem \ref{LWPresultinWS} (ii) is optimal.
    \item[(ii)] Similar as before, Theorem \ref{Uniquecontprinc} (ii) shows that for arbitrary initial data in $u_0\in H^{s_1,s_2}(\mathbb{R}\times \mathbb{Z})\cap L^2(|x|^{3^{+}}\, dx dy)$ such that $\widehat{u_0}(0,\eta)=0$ for all $\eta \in \mathbb{Z}$, the polynomial decay $|x|^{\frac{3}{2}^{-}}$ is the largest possible decay for solutions of \eqref{SHeq} in $L^2$-spaces. For example, one can take a sufficiently regular initial condition $u_0\in L^2(|x|^{3^{+}}\, dx dy)$ be such that $\widehat{u_0}(0,\eta)= 0$ for all $\eta$, and such that  $\widehat{x u_0}(0,\eta)\neq 0$ for some $\eta \neq 0$. It follows from Theorem \ref{LWPresultinWS} that there exist a time $T>0$ and a unique solution of \eqref{SHeq} with initial condition $u_0$ such that $u\in C([0,T];L^2(|x|^{3^{-}}\, dx dy))$, but according to \eqref{contraconlcus}, it can not happen that for some $0<T_1\leq T$, $u\in L^{\infty}([0,T_1];L^2(|x|^{3^{+}}\, dx dy))$. 
    \item[(iii)] In the case of the Benjamin-Ono equation (\eqref{BO} with $K=k=1$), it is known from \cite{FonsecaPonce2011} that the only solution $u$ of such equation that satisfies $|x|^{\frac{7}{2}}u(t_j)\in L^{2}(\mathbb{R})$ for three different times $t_1,t_2,t_3$ is the null solution. However, in the cylinder case of the Cauchy problem \eqref{SHeq}, it is not known if there exists a maximum polynomial decay $\beta$ and a number of times $j=1,\dots M$ such that $|x|^{\beta}u(t_j)\in L^{2}(\mathbb{R}\times \mathbb{T})$ forces the solution $u$ of \eqref{SHeq} to be null. In this sense, our results in Theorem \ref{Uniquecontprinc} (ii) are not sufficient to conclude such a result.
    \end{itemize}
\end{rem}

This paper is organized as follows: Section \ref{Sectiprel} introduces the preliminaries and estimates necessary for the derivation of our main results. Section \ref{Lineareqstudysec} focuses on the study of the linear equation determined by \eqref{SHeq}. In this section, we derive the persistence results in weights spaces in Theorem \ref{LinearEst} and the unique continuation principles in Theorem \ref{lineaerequniquecont1} and Corollary \ref{lineaerequniquecont2}. We then conclude with the results for the nonlinear model \eqref{SHeq} in Section \ref{SectionNonlinear}, in which we present local well-posedness results in weighted spaces and asymptotic at infinity unique continuation principles discussed in Theorems \ref{LWPresultinWS} and \ref{Uniquecontprinc}.

%%%%%%%%%%%%%%%%%%%%%%%%%%%%%%%%%%%%%%%%%%%%%%%%%%%%%%%%%%%%%%%%%%%%%%%%%%%%%%%%%%%%%%%%%%%%%%%%%%%%%%%%%%%%%%%%%%%%%%%%%%%%%%%%%%%%%%%%%%%%%%%%%%%%%%%%%%%%%%%%%%%%%%%%%%%%%%%%%%%%%%%%%%%%%%%%%%%%%%%%%%%%%%%%%%%%%%%%%%%%%%%%%%%%%%%%%%%%%%%%%%%%%%%%%%%%%%%%%%%%%%%%%%%%%%%%%%%%%%%%%%%

\section{Notation and preliminaries}\label{Sectiprel}

Given two positive numbers $a$ and $b$, $a\lesssim b$ means that there exists a positive constant $c>0$ such that $a\leq c b$. We write $a\sim b$ whenever $a\lesssim b$ and $b\lesssim a$. $[A,B]$ stands for the commutator between two operators $A,B$, i.e.,
\begin{equation*}
[A,B]=AB-BA.
\end{equation*}  
Given $p\in [1,\infty]$ and $d\geq 1$ integer, the Lebesgue spaces $L^p(\mathbb{K})$ are defined as usual  by its norm as $\|f\|_{L^p(\mathbb{K})}=\left\|f\right\|_{L^p}=\left(\int_{\mathbb{K}} |f(x)|^p\,dx\right)^{\frac{1}{p}},$ with the usual modification when $p=\infty$. To emphasize the dependence on the spatial variables, we will denote by $\|f\|_{L^p(\mathbb{R}\times \mathbb{T})}=\|f\|_{L^p_{xy}(\mathbb{R}\times \mathbb{T})}$. $S(\mathbb{R}\times \mathbb{T})$ is the space of test functions that consists of smooth functions $f\in C^{\infty}(\mathbb{R}\times \mathbb{T})$, where for all $y\in \mathbb{T}$, the function $x\mapsto f(x,y)$ defines a Schwartz function or a rapidly decreasing function. The space $C^{\infty}_{c}(\mathbb{R})$ denotes the set of all functions $f:\mathbb{R}\rightarrow \mathbb{R}$ which are both smooth and compactly supported.

The  Fourier transform is defined by $$\widehat{f}(\xi,\eta)=\mathcal{F}f(\xi,\eta)=\int_{\mathbb{T}} \int_{\mathbb{R}} f(x,y) e^{-i x\cdot \xi-iy\cdot \eta}\, dx dy,$$
where $\xi\in \mathbb{R}$, $\eta \in \mathbb{Z}$. We denote the inverse Fourier transform by $f^{\vee}=\mathcal{F}^{-1}(f)$. $D^s_x$ denotes the homogeneous derivative of order $s\in \mathbb{R}$ in the $x$-direction, which is defined as $D_x^s f=(\mathcal{H}_x\partial_x)^s f=(|\xi|^s \widehat{f}(\xi,\eta))^{\vee}$. Similarly, we define $D_y^s$, $s\in \mathbb{R}$, as the homogeneous derivative of order $s$ acting on the $y$-variable. The anisotropic Sobolev spaces $H^{s_1,s_2}(\mathbb{R}\times \mathbb{T})$ consist of all tempered distributions on $\mathcal{S}(\mathbb{R}\times \mathbb{T})$ such that $\left\|f\right\|_{H^{s_1,s_2}}=\left\|J^{s_1}_x f \right\|_{L^2}+\left\|J_y^{s_2} f \right\|_{L^2}<\infty$, where  the Bessel operator $J^{s_1}_x$ is defined via the Fourier transform according to $\widehat{J^{s_1}_x f}(\xi,\eta)=\langle \xi \rangle^{s_1} \widehat{f}(\xi,\eta)$, and  $J^{s_2}_y$ is given by $\widehat{J^{s_2}_y f}(\xi,\eta)=\langle \eta \rangle^{s_2} \widehat{f}(\xi,\eta)$. When $s_1=s_2$, we write $H^s(\mathbb{R}\times \mathbb{T})=H^{s,s}(\mathbb{R}\times \mathbb{T})$, and we denote by $J^{s}_{x,y}$ the operator defined by the Fourier multiplier $\langle(\xi,\eta)\rangle^s$. We will also work with the space $H^{s,0}(\mathbb{R}\times \mathbb{Z})$, $s\geq 0$, which is defined by the norm $\|J_{x}^s f(x,y)\|_{L^2(\mathbb{R}\times \mathbb{Z})}$.

The unitary group $\{U(t)\}$ associated with the linear problem determined by \eqref{SHeq} is defined as 
\begin{equation}\label{Group}
    U(t)f(x,y)=\frac{1}{(2\pi)^{2}}\sum_{\eta=-\infty}^{\infty}\int_{\mathbb{R}} e^{it\omega(\xi, \eta)+ix\xi+iy\eta} \widehat{f}(\xi,\eta)\, d\xi,
\end{equation}
where
\begin{equation}\label{dispersivesymbol}
    \omega(\xi,\eta)=\sign(\xi)\xi^2+\sign(\xi)\eta^2.
\end{equation}

We also consider the unitary group $\{U_1(t)\}$ associated with the linear equation $\partial_tv-\mathcal{H}_x\partial_y^2v=0$, which satisfies
\begin{equation}\label{Grouponly}
    U_1(t)f(x,y)=\frac{1}{(2\pi)^{2}}\sum_{\eta=-\infty}^{\infty}\int_{\mathbb{R}} e^{it\sign(\xi)\eta^2+ix\xi+iy\eta} \widehat{f}(\xi,\eta)\, d\xi.
\end{equation}

We will use the following approximation of the function $\langle x \rangle$ as in \cite{FonsecaLinaresPonce2013} (see also \cite{FonsecaLinaresPonce2012,FonsecaPonce2011}). For any $N\in \mathbb{Z}^{+}$ and $0< \theta \leq 1$, we define the truncated weights $\langle \cdot \rangle_{N}^{\theta} : \mathbb{R} \rightarrow \mathbb{R}$ according to
 \begin{equation}\label{trunweightdef}
\langle x \rangle_{N}^{\theta} =\left\{\begin{aligned} 
 & (1+ |x|^2)^{\frac{\theta}{2}}, \text{ if } |x|\leq N, \\
 &(2N)^{\theta}, \hspace{1,1cm}\text{if } |x|\geq 3N
 \end{aligned}\right.
 \end{equation}
in such a way that $\langle x \rangle_{N}$ is smooth and non-decreasing in $|x|$ with $|\partial_x\langle x \rangle_{N}| \leq 1$,  and there exist constants $c_{l}$ independent of $N$ such that $|\partial_x^l\langle x \rangle_{N}| \leq c\partial_x^l\langle x \rangle$, for each $l=2,3$. 

We require some continuity properties for the Hilbert transform operator in the weighted spaces determined by $\langle x \rangle_{N}^{\theta}$ 

\begin{prop}\label{a2condH}
For any $-\frac{1}{2}<\theta<\frac{1}{2}$ and $N\in \mathbb{Z}^{+}$, the Hilbert transform $\mathcal{H}_x$ is bounded in $L^2(\mathbb{R},\langle x \rangle^{2\theta}_N dx)$ with a constant depending on $\theta$ but independent of $N\in \mathbb{Z}^{+}$.
\end{prop}

The deduction of the above proposition can be consulted in \cite[Proposition 1]{FonsecaPonce2011}. We remark that such a result is based on the continuity of the Hilbert transform in weighted $L^p$-spaces (see \cite{HuntMuckenhouptWheeden1973,Petermichl2007}).

%%%%%%%%%%%%%%%%%%%%%%%%%%%%%%%%%%%%%%%%%%%%%%%%%%%%%%%%%%%%%%%%%%%%%%%%%%%%%%%%%%%%%%%%%%%%%%%%%%%%%%%%%%%%%%%%%%%%%%%%%%%%%%%%%%%%%%%%%%%%%%%%%%%%%%%%%%%%%%%%%%%%%%%%%%%%%%%%%%%%

\subsection{Commutator and fractional derivative estimates}

We recall the following commutator estimate for the Hilbert transform.
\begin{prop}\label{CalderonComGU}  
Let $1<p<\infty$. Assume $l,m \in \mathbb{Z}^{+}\cup \{0\}$, $l+m\geq 1$ then
\begin{equation}\label{Comwellprel1}
    \|\partial_x^l[\mathcal{H}_x,g]\partial_x^{m}f\|_{L^p(\mathbb{R})} \lesssim_{p,l,m} \|\partial_x^{l+m} g\|_{L^{\infty}(\mathbb{R})}\|f\|_{L^p(\mathbb{R})}.
\end{equation}
\end{prop}

\begin{proof}
We refer to \cite[Lemma 3.1]{DawsonMcGahaganPonce2008}. For a fractional version, see also \cite[Proposition 1.1]{Riano2021}.
\end{proof}
We also require the following commutator estimate for fractional derivatives, whose deduction can be consulted in \cite[Proposition 3.10]{Li2019} and \cite{DawsonMcGahaganPonce2008}.
\begin{prop}\label{propcomm2} 
For any $0\leq \beta<1$, $0<\gamma \leq 1-\beta$, $1<p<\infty$, we have
\begin{equation}
    \|D_x^{\beta}[D_x^{\gamma},f]D^{1-(\beta+\gamma)}g\|_{L^p(\mathbb{R})}\lesssim \|\partial_x f\|_{L^{\infty}(\mathbb{R})}\|g\|_{L^p(\mathbb{R})}.
\end{equation}
\end{prop}

In the context of studies of persistence of fractional weights, the following characterization of fractional Sobolev space  $W^{b,p}(\mathbb{R})$ is frequently used, see \cite{FonsecaPonce2011,FonsecaLinaresPonce2012,Iorio1986,Iorio1991,CunhaPastor2021,CunhaPastor2014,Riano2020,Riano2021}.  Given $b\in(0,1)$, $\frac{2}{(1+2b)}<p<\infty$, we have that $f\in W^{b,p}(\mathbb{R})$ if and only if $f\in L^p(\mathbb{R})$, and 
\begin{equation}\label{fracderiv1}
\mathcal{D}^bf(x)=\left(\int\limits_{\mathbb{R}}\frac{|f(x)-f(z)|^2}{|x-z|^{1+2b}}\, dz\right)^{\frac{1}{2}}\in L^{p}(\mathbb{R}),    
\end{equation}
with 
\begin{equation}\label{equi1}
\left\|J^b f\right\|_{L^p}=\left\|(1-\partial_x^2)^{\frac{b}{2}} f\right\|_{L^p} \sim \big(\left\|f\right\|_{L^p}+\left\|\mathcal{D}^b f\right\|_{L^p}\big) \sim \big( \left\|f\right\|_{L^p}+\left\|D^b f\right\|_{L^p}\big).
\end{equation} 
Moreover, if $f\in H^b(\mathbb{R})$,
\begin{equation}\label{equi2}
 \|\mathcal{D}^b(f)\|_{L^2}\sim \|D^b f\|_{L^2}.    
\end{equation}
For proof of the previous characterization, we refer to \cite{Stein1961}. As a consequence of the definition of the fractional derivative in \eqref{fracderiv1}, we have    
\begin{equation}\label{prelimneq0}
\mathcal{D}^b(fg)(x) \lesssim \mathcal{D}^b(f)(x)|g(x)|+\|f\|_{L^{\infty}}\mathcal{D}^b(g)(x).
\end{equation}
For all $b\in (0,1)$, it follows
\begin{equation} \label{prelimneq}  
\left\|\mathcal{D}^b(fg)\right\|_{L^2} \lesssim \left\|f\mathcal{D}^b g\right\|_{L^2}+\left\|g\mathcal{D}^bf \right\|_{L^2},
\end{equation}
and it holds
\begin{equation} \label{prelimneq1} 
\left\|\mathcal{D}^{b} h\right\|_{L^{\infty}} \lesssim \big(\left\|h\right\|_{L^{\infty}}+\left\|\partial_x h\right\|_{L^{\infty}} \big).
\end{equation} 

\begin{rem}
We will use the fractional derivative introduced above to exchange fractional polynomial decay in space into fractional differentiation in the frequency domain. We will use the notation $\mathcal{D}_{\xi}^b$ to emphasize that we only differentiate in the $\xi$-direction. 
\end{rem}

\begin{lemma}\label{derivexp}   
Let $b\in (0,1)$, $t\in \mathbb{R}$, then
\begin{equation}\label{steinderiv1} 
\begin{aligned}
\mathcal{D}^{b}_{\xi}\big(e^{i t \xi|\xi| t}\big)(x) \lesssim \langle t \rangle^{b}\langle x \rangle^{b}.
\end{aligned}
\end{equation}
Moreover, let $\eta\in \mathbb{Z},$ $t\in \mathbb{R}$, then 
\begin{equation}\label{steinderiv2}
  \mathcal{D}^b_{\xi} (e^{i t \sign(\xi)\eta^2})(x)\lesssim \min\{1, |t|\eta^2\}|x|^{-b}, \hspace{0.5cm} x\neq 0. 
\end{equation}
\end{lemma}

\begin{proof}
The inequality \eqref{steinderiv1} was deduced in \cite[Proposition 2]{FonsecaPonce2011} (see also \cite{NahasPonce2009}). Next, we combine the definition of the fractional derivative in \eqref{fracderiv1}, the mean value theorem, and the change of variable $w=x-z$ to deduce
\begin{equation*}
 \begin{aligned}
\Big( \mathcal{D}^b_{\xi} (e^{i t \sign(\xi)\eta^2})(x)\Big)^2 =&\int_{\sign(x)\neq \sign(y)} \frac{|e^{i t \sign(x) \eta^2}-e^{it\sign(z) \eta^2}|^2}{|x-z|^{1+2b}}\, dz  \\
 \lesssim & \min\{1, (t\eta^2)^2\}\int_{|w|>|x|}\frac{1}{|w|^{1+2b}}\, dw\\
 \lesssim & \min\{1, (t\eta^2)^2\}\frac{1}{|x|^{2b}}.
 \end{aligned}   
\end{equation*}
Taking the square root of the above inequality yields \eqref{steinderiv2}.
\end{proof}

The same arguments in the deduction of \eqref{steinderiv2} yield:

\begin{corol}\label{signcoro}
Let $b\in (0,1)$, then
\begin{equation*}
\begin{aligned}
\mathcal{D}^{b}_{\xi}\big(\sign(\xi)\big)(x) \lesssim |x|^{-b}, \hspace{0.5cm} x\neq 0.
\end{aligned}
\end{equation*}
\end{corol}

Next, to control the negative weight $|x|^{-b}$ in \eqref{steinderiv2}, we require the following proposition.
\begin{prop}\label{decaywight}
 Let $\frac{1}{2}<\ell\leq 1$ and $f\in H^{\ell}(\mathbb{R})$ such that $f(0)=0$, then
 \begin{equation*}
 \begin{aligned}
 \||\cdot|^{-\ell}f\|_{L^2(\mathbb{R})}\lesssim \|f\|_{H^{\ell}(\mathbb{R})}.    
 \end{aligned}    
 \end{equation*}
\end{prop}

\begin{proof}
We refer to \cite[Proposition 3.1]{Yafaev1999}. 
\end{proof}

We now present some interpolation results.
\begin{prop}\label{interpo} For any $a,b>0$, $\gamma\in (0,1)$,
\begin{equation}\label{interpineq1}
\|\langle x \rangle^{\gamma a}\big(J^{(1-\gamma)b}f\big)\|_{L^2(\mathbb{R})}\lesssim \|J^{b}f\|^{1-\gamma}_{L^2(\mathbb{R})}\|\langle x \rangle^{a} f\|_{L^2(\mathbb{R})}^{\gamma},
\end{equation}
\begin{equation}\label{interpineq2}
\|J^{\gamma a}\big(\langle x \rangle^{(1-\gamma)b}f\big)\|_{L^2(\mathbb{R})}\lesssim \|\langle x \rangle^{b}f\|^{1-\gamma}_{L^2(\mathbb{R})}\|J^{a} f\|_{L^2(\mathbb{R})}^{\gamma}.
\end{equation}
\end{prop}

\begin{proof}
We refer to \cite[Lemma 4]{NahasPonce2009}, see also \cite[Lemma 2.7]{LinaresPastorSilva2020}.
\end{proof}

Next, we introduce an interpolation inequality for the operator $\langle x \rangle_N$ in \eqref{trunweightdef}, whose deduction can be consulted in \cite[Lemma 1]{FonsecaPonce2011}. 

\begin{prop}\label{interpo2} Let $N\in \mathbb{Z}^{+}$. For any $a,b>0$, $\gamma\in (0,1)$,
\begin{equation}\label{interpineq3}
\|J^{\gamma a}\big(\langle x \rangle_N^{(1-\gamma)b}f\big)\|_{L^2(\mathbb{R})}\lesssim \|\langle x \rangle_N^{b}f\|^{1-\gamma}_{L^2(\mathbb{R})}\|J^{a} f\|_{L^2(\mathbb{R})}^{\gamma},
\end{equation}
where the implicit constant is independent of $N$. 
\end{prop}

We conclude this section by deducing the following interpolation result in $L^2(\mathbb{R}\times \mathbb{T})$.
\begin{lemma}\label{InterpLemma}
Let $a, b>0$. Assume that $f \in H^a(\mathbb{R}\times \mathbb{T})\cap L^2(|x|^b\, dx dy)$. Then for any $\gamma \in (0,1)$
\begin{equation}\label{eqsuppsep3}
\|J^{\gamma a}_{x,y}\big(\langle x\rangle^{(1-\gamma) b}f)\|_{L^2_{x,y}} \lesssim \|\langle x \rangle^{b}f\|_{L^{2}_{x,y}}^{1-\gamma}\|J^a_{x,y} f\|_{L^2_{x,y}}^{\gamma}.
\end{equation}
\end{lemma}
\begin{proof}
Let $g\in \mathcal{S}(\mathbb{R}\times \mathbb{T})$ with $\|g\|_{L^{2}_{x,y}}=1$. We define
\begin{equation*}
F(z)=e^{z^2-1}\int_{\mathbb{T}}\int_{\mathbb{R}} J_{x,y}^{z a}(\langle x \rangle^{(1-z)b}f(x,y))\overline{g(x,y)} \, dxdy.
\end{equation*}
Using that $g\in \mathcal{S}(\mathbb{R}\times \mathbb{T})$, and Plancherel's identity, we have that $F$ is continuous in $\{z=\beta+i\kappa\, :\, 0\leq \beta \leq 1\}$ and analytic in its interior. Moreover,
\begin{equation*}
|F(0+i\kappa)|\lesssim e^{-(\kappa^2+1)}\|\langle x\rangle^{b} f\|_{L^2_{x,y}}.
\end{equation*}
On the other hand, we have
\begin{equation}\label{interpcyl1}
\begin{aligned}
|F(1+i\kappa)|\lesssim & e^{-\kappa^2}\|J^a_{x,y}(\langle  x \rangle^{-i\kappa b} f)\|_{L^2_{x,y}} \\
\lesssim & e^{-\kappa^2}\big(\|J^a_{x}(\langle  x \rangle^{-i\kappa b} f)\|_{L^2_{x,y}}+\|J^a_{y} f\|_{L^2_{x,y}}\big),
\end{aligned}
\end{equation}
where we have used that $\langle  x \rangle^{-iyb}$ is independent of the $y$-variable. Now, writing $a=m+m_1$ with $m\in \mathbb{Z}^{+}\cup\{0\}$, and $m_1\in [0,1)$, we distribute the local and fractional derivatives, and using \eqref{equi1}, we get
\begin{equation*}
\begin{aligned}
    \|J^a_{x}(\langle  x \rangle^{-i\kappa b} f)\|_{L^2_{x,y}} 
%    \lesssim & \|f\|_{L^2_{x,y}}+\|D^{m_1}_x(\partial_x^m(\langle  x \rangle^{-i\kappa b} f))\|_{L^2_{x,y}}\\
\lesssim & \|f\|_{L^2_{x,y}}+\sum_{k=0}^{m}\|\|D^{m_1}_x(\partial_x^{k}(\langle  x \rangle^{-i\kappa b})\partial^{m-k} f)\|_{L^2_{x}}\|_{L^2_{y}}\\
        \lesssim & \|f\|_{L^2_{x,y}}+\sum_{k=0}^{m}\big(\|\mathcal{D}_x^{m_1}(\partial_x^{k}(\langle  x \rangle^{-i\kappa b}))\|_{L^{\infty}_x}+\|\partial_x^{k}(\langle  x \rangle^{-i\kappa b})\|_{L^{\infty}_x}\big)\|J^{a}_x f\|_{L^2_{x,y}}\\
        \lesssim & \|J^{a}_x f\|_{L^2_{x,y}},
\end{aligned}    
\end{equation*}
where we also used the properties \eqref{prelimneq0}-\eqref{prelimneq1}. Gathering the previous estimate, we deduce
\begin{equation}\label{interpcyl2}
\begin{aligned}
|F(1+i\kappa)|\lesssim & e^{-\kappa^2}\|J^a_{x,y} f\|_{L^2_{x,y}}.
\end{aligned}
\end{equation}
We remark that the implicit constants in \eqref{interpcyl1} and \eqref{interpcyl2} depend polynomially on $\kappa$ (i.e., sums of powers of $\kappa$ with fixed degree), but thanks to the exponential factor $e^{-\kappa^2}$, such constants can be controlled uniformly for all $\kappa$. At this point, \eqref{eqsuppsep3} is a consequence of Hadamard's three lines Theorem.

\end{proof}

%%%%%%%%%%%%%%%%%%%%%%%%%%%%%%%%%%%%%%%%%%%%%%%%%%%%%%%%%%%%%%%%%%%%%%%%%%%%%%%%%%%%%%%%%%%%%%%%%%%%%%%%%%%%%%%%%%%%%%%%%%%%%%%%%%%%%%%%%%%%%%%%%%%%%%%%%%%%%%%%%%%%%%%%%%%%%%%%%%%%%%%%%%%%%%%%%%%%%%%%%%%%
\section{Spatial decay properties linear equation}\label{Lineareqstudysec}

To obtain estimates of the group $U(t)$ in weighted spaces, transferring polynomial decay in the spatial domain into differentiation in the frequency domain, we need estimates of the derivatives of the function $e^{i t\omega(\xi,\eta)}$, which is the multiplier associated with the family of operators $U(t)$. We perform such an estimate in the following proposition. But first, we introduce some notation, we will denote by $\delta_{\xi}$ the Dirac delta distribution at the origin acting only on the $\xi$-variable, i.e., $\delta_{\xi}(\phi)=\phi(0,\eta)$, $\eta\in \mathbb{Z}$, for any $\phi=\widehat{\psi}$ with $\psi\in S(\mathbb{R}\times \mathbb{T})$. We denote its $k$th distributional derivative with respect to $\xi$ as $\delta_{\xi}^{(k)}=\partial_{\xi}^{k}\big(\delta_{\xi}\big)$.

\begin{prop}\label{generalderiv}
Let $k\geq 1$ be a fixed integer and $f(\xi,\eta)\in H^{k,0}(\mathbb{R}\times\mathbb{Z})$. Consider $\omega=\omega(\xi,\eta)$ be given as in \eqref{dispersivesymbol}. Then, the following identity holds true
\begin{equation}\label{derivative}
\begin{aligned}
 \frac{\partial^k}{\partial \xi^k}\big(e^{it\omega}f\Big)=&\sum_{l=0}^{k-1}\sum_{m=0}^{k-1-l}\nu_{l,m}(t)\sin(\eta^2 t)\frac{\partial^{l}}{\partial_{\xi}^{l}}f(0,\eta)\delta_{\xi}^{(m)} +\sum_{l=0}^{k-3}\sum_{m=0}^{k-3-l} d_{l,m}(t)\cos(\eta^2 t)\frac{\partial^{l}}{\partial_{\xi}^{l}}f(0,\eta)\delta_{\xi}^{(m)} \\
 &+e^{it \omega }\sum_{m=1}^k \sum_{\substack{l_1+\dots+l_{m}+l_{m+1}=k-m\\ 0\leq l_j\leq 1\\ j=1,\dots,m}}c_{l_1,\dots,l_{m+1}}\partial_{\xi}^{l_1}(2it|\xi|)\dots \partial_{\xi}^{l_m}(2it|\xi|)\frac{\partial^{l_{m+1}}}{\partial_{\xi}^{l_{m+1}}}f\\
&+ e^{it \omega} \frac{\partial^{k}}{\partial_{\xi}^{k}}f,
\end{aligned}    
\end{equation}
where the differentiation is taken in the distributional sense, and the above holds for polynomials $\nu_{l,m}(t)$, $d_{l,m}(t)$ of degree less than or equal to $\lfloor\frac{k-1}{2}\rfloor$\footnote{ Here $\lfloor \cdot \rfloor$ denotes the floor function, that is the function with input $x\in \mathbb{R}$ and output the greatest integer less than or equal to $x$.} with $\nu_{k-1,0}(t)=\text{constant}\neq 0$, $d_{k-3,0}(t)=4it$, and some constants $c_{l_1,\dots,l_{m+1}}$, where $l_1=1$, $l_2=k-2$, and $c_{l_1,l_2}=1$, when $m=1$.  Above, we also assume the zero convention for the empty summation such as $\sum_{m=0}^{k-3}(\cdots)=0$, whenever $k=1,2$.
\end{prop}

\begin{proof}
Using the definition of the distributional derivative, it is not hard to see
\begin{equation}\label{distribderiv1}
 \frac{\partial}{\partial_{\xi}}(e^{it \omega }f)=2i\sin(\eta^2 t)f(0,\eta)\delta_{\xi}+e^{it\omega}(2it|\xi|)f+e^{it\omega }\frac{\partial}{\partial_{\xi}}f, 
\end{equation}
it holds
\begin{equation}\label{distribderiv2}
 \frac{\partial}{\partial_{\xi}}( e^{it \omega }(2it \sign(\xi))f)=4it\cos(\eta^2 t)f(0,\eta)\delta_{\xi}+e^{it\omega}(2it\sign(\xi))(2it|\xi|)f+e^{it\omega } (2it\sign(\xi)) \frac{\partial}{\partial \xi}f  
\end{equation}
and
\begin{equation}\label{distribderiv3}
 \frac{\partial}{\partial_{\xi}}(e^{it \omega }(2it|\xi|)f)=e^{it\omega}(2i|\xi|t)^2f+e^{it\omega }(2it\sign(\xi))f+e^{it\omega }(2it|\xi|)\frac{\partial}{\partial_{\xi}}f.   
\end{equation}

Thus, combining the fact that $\partial_{\xi}(|\xi|)=\sign(\xi)$, and the identities \eqref{distribderiv1}, \eqref{distribderiv2} and \eqref{distribderiv3}, we can use an inductive argument on the order of the derivative $k\geq 1$ in $\frac{\partial^{k}}{\partial \xi^k}(e^{i\omega t}f)$ to get \eqref{derivative}.
\end{proof}

In the following proposition, we deduce some estimates in weighted spaces for the group $\{U_1(t)\}$ defined in \eqref{Grouponly}.
\begin{prop}\label{decaysimplyeq}
Let $0<\theta<\frac{1}{2}$. If $f\in L^2(\mathbb{R}\times\mathbb{T})\cap L^2(|x|^{2\theta}\, dx)$, then
\begin{equation}\label{group1deca}
  \|\langle x\rangle^{\theta}U_1(t)f\|_{L^2}\lesssim \|\langle x \rangle^{\theta}f\|_{L^2}.
\end{equation}
If $\theta=\frac{1}{2}$, assume $f\in L^2(\mathbb{R}\times\mathbb{T})\cap L^2(|x|\, dx)$ and $\mathcal{H}_xf\in L^2(|x|\, dx)$, then
\begin{equation}\label{group1deca1}
  \|\langle x\rangle^{\frac{1}{2}}U_1(t)f\|_{L^2}+\|\langle x\rangle^{\frac{1}{2}}U_1(t)\mathcal{H}_xf\|_{L^2}\lesssim \|\langle x \rangle^{\theta}f\|_{L^2}+\|\langle x \rangle^{\theta}\mathcal{H}_xf\|_{L^2}.
\end{equation}
\end{prop}

\begin{proof}
We will assume that $f\in S(\mathbb{R}\times \mathbb{T})$ as the general case follows from an approximation argument applied to our estimates. Consequently, setting $u=U_1(t)f$, it follows that $u\in C^k(\mathbb{R},H^{\infty}(\mathbb{R}\times \mathbb{T}))$ for all integer $k\geq 0$, where $H^{\infty}(\mathbb{R}\times \mathbb{T}))=\bigcap_{s\geq 0} H^s(\mathbb{R}\times \mathbb{T})$, and it also holds that $u$ solves the equation
\begin{equation}\label{lineareq1}
    \partial_t u-\mathcal{H}_x \partial_y^2 u=0.
\end{equation}
Since $u(t)\in H^{\infty}(\mathbb{R}\times \mathbb{T})$, we can apply $\mathcal{H}_x$ to the previous equation, using that $\mathcal{H}_x^2=-1$ to get
\begin{equation}\label{lineareq2}
    \partial_t \mathcal{H}_xu+ \partial_y^2 u=0.
\end{equation}
Multiplying \eqref{lineareq1} by $\langle x\rangle_N^{2\theta} u$ and integrating over $\mathbb{R}\times \mathbb{T}$ yields
\begin{equation*}
  \frac{1}{2}\frac{d}{dt}\int \big(\langle x\rangle_N^{\theta} u\big)^2\, dx dy =\int  \mathcal{H}_x \partial_y^2 u \langle x\rangle_N^{2\theta} u\, dx dy.
\end{equation*}
Note that the above identity is justified by the conditions on regularity and integrability of $u$, and the fact that the weight $\langle x \rangle_N^{\theta}\in L^{\infty}(\mathbb{R})$. Next, we multiply \eqref{lineareq2} by $\langle x\rangle_N^{2\theta} \mathcal{H}_xu$, integrating the resulting expression over $\mathbb{R}\times \mathbb{T}$, we deduce
\begin{equation*}
\begin{aligned}
     \frac{1}{2}\frac{d}{dt}\int \big(\langle x\rangle_N^{\theta}\mathcal{H}_x u\big)^2\, dx dy=&-\int   \partial_y^2 u \langle x\rangle_N^{2\theta}\mathcal{H}_x u\, dx dy\\
  =&-\int   \langle x\rangle_N^{2\theta} u\mathcal{H}_x \partial_y^2u\, dx dy, 
\end{aligned}
\end{equation*}
where we have also integrated by parts in the $y$-variable, and we used that $\langle x\rangle_N^{\theta}\in L^{\infty}(\mathbb{R})$ does not depend on $y$. Consequently, we combine the previous equations to get
\begin{equation}\label{diffident}
    \begin{aligned}
     \frac{1}{2}\frac{d}{dt}\Big(\int \big(\langle x\rangle_N^{\theta} u\big)^2\, dx dy+\int \big(\langle x\rangle_N^{\theta}\mathcal{H}_x u\big)^2\, dx dy \Big)=0.
    \end{aligned}
\end{equation}
Integrating over $[0,t]$ if $t>0$, or over $[t,0]$ when $t<0$, the identity above yields
\begin{equation}\label{Afterintegra}
\begin{aligned}
   \|\langle x\rangle_N^{\theta}u\|_{L^2}^2+\|\langle x\rangle_N^{\theta}\mathcal{H}_x u\|_{L^2}^2=&\|\langle x\rangle_N^{\theta}f\|_{L^2}^2+\|\langle x\rangle_N^{\theta}\mathcal{H}_x f\|_{L^2}^2.
\end{aligned}
\end{equation}
We notice that when $0<\theta<\frac{1}{2}$, Proposition \ref{a2condH} implies
\begin{equation*}
\begin{aligned}
\|\langle x\rangle_N^{\theta}\mathcal{H}_x u\|_{L^2}+\|\langle x\rangle_N^{\theta}\mathcal{H}_x f\|_{L^2}=&\|\, \|\langle x\rangle_N^{\theta}\mathcal{H}_x u\|_{L_x^2}\|_{L^2_y}+\|\, \|\langle x\rangle_N^{\theta}\mathcal{H}_x f\|_{L^2_x}\|_{L^2_y}\\
\lesssim & \|\langle x\rangle_N^{\theta} u\|_{L^2}+ \|\langle x\rangle_N^{\theta} f\|_{L^2},    
\end{aligned}
\end{equation*}
with implicit constant independent of $N$. Then, using that $\langle x \rangle^{\theta}_N\leq \langle x\rangle^{\theta}$, we infer 
\begin{equation*}
\begin{aligned}
   \|\langle x\rangle_N^{\theta}u\|_{L^2}^2+\|\langle x\rangle_N^{\theta}\mathcal{H}_x u\|_{L^2}^2\lesssim \|\langle x\rangle^{\theta}f\|_{L^2}^2.  
\end{aligned}
\end{equation*}
Taking $N\to \infty$ in the previous inequality gives \eqref{group1deca}. On the other hand, when $\theta=\frac{1}{2}$, we obtain the exact same estimate \eqref{Afterintegra}, which is justified by the assumption $\mathcal{H}_xf \in L^2(|x|\, dx dy)$. Thus, taking $N\to \infty$, we get \eqref{group1deca1}. 
\end{proof}

\subsection{Proof of Theorem \ref{LinearEst}}

We are in the condition to obtain spatial decay properties for the group $\{U(t)\}$ of solutions of the linear equation associated with \eqref{SHeq}.

\begin{proof}[Proof of Theorem \ref{LinearEst}]
We will assume that $f$ is sufficiently regular and has enough decay to justify the argument developed below. The general case follows from approximation to our estimates. Depending on the size of the weight $\theta>0$, we consider the following cases:
\begin{itemize}
    \item[(a)] $0<\theta\leq \frac{1}{2}$.
    \item[(b)] $\frac{1}{2}<\theta< 1$.
    \item[(c)] $1 \leq \theta<\frac{3}{2} $.
    \item[(d)] $k+\frac{1}{2}\leq \theta< k+1$, $k\in \mathbb{Z}^{+}$.
     \item[(e)] $k+1\leq \theta< k+\frac{3}{2}$, $k\in \mathbb{Z}^{+}$.
\end{itemize}

Cases (a), (b), and (c) show parts $(i)$, $(ii)$, and $(iii)$ in the statement of Theorem \ref{LinearEst}. To make our arguments more understandable in the more general cases (d) and (e), we have decided to show (a), (b), and (c) first. We remark that cases (d) and (e) concern part $(iv)$. Also, notice that $k=0$ in (d) and (e) gives us (b) and (c), respectively. 
\\ \\
\underline{\bf (a) Assume $0<\theta\leq \frac{1}{2}$}. We first recall the notation $\omega(\xi,\eta)=\xi|\xi|+\sign(\xi)\eta^2$, $\xi\in \mathbb{R}$, $\eta\in \mathbb{Z}$ introduced in \eqref{dispersivesymbol}. Writing $e^{it\omega}=e^{it\xi|\xi|}e^{it\sign(\xi)\eta^2}$, by using Plancherel's identity, the equivalence \eqref{equi1}, and the product rule type \eqref{prelimneq}, we get 
\begin{equation}\label{linearestimeq0}
\begin{aligned}
 \||x|^{\theta}U(t)f\|_{L^2}\sim & \|D_{\xi}^{\theta}(e^{it\omega}\widehat{f}\,\big)\|_{L^2}\\
 \lesssim & \|\widehat{f}\,\|_{L^2}+\|\|\mathcal{D}_{\xi}^{\theta}\big(e^{it\omega}\widehat{f}\,\big)\|_{L^2_{\xi}}\|_{L^2_{\eta}}\\
 \lesssim & \|\widehat{f}\,\|_{L^2}+\|\|\mathcal{D}_{\xi}^{\theta}\big(e^{it\xi|\xi|}\big)\widehat{f}\,\|_{L^2_{\xi}}\|_{L^2}+\|\mathcal{D}_{\xi}^{\theta}\big(e^{it\sign(\xi)\eta^2}\widehat{f}\,\big)\|_{L^2}.
\end{aligned}
\end{equation}
Next, we estimate the right-hand side of the above inequality. By Lemma \ref{derivexp}, inequality \eqref{steinderiv1}, and Plancherel's identity, we find
\begin{equation}\label{linearestimeq0.1}
\begin{aligned}
\|\|\mathcal{D}_{\xi}^{\theta}\big(e^{it\xi|\xi|}\big)\widehat{f}\,\|_{L^2_{\xi}}\|_{L^2}\lesssim \langle t \rangle^{\theta}\|\|\langle \xi\rangle^{\theta}\widehat{f}\,\|_{L^2_{\xi}}\|_{L^2_{\eta}}\sim\langle t \rangle^{\theta}\|f\|_{H^{\theta,0}}.    
\end{aligned}   
\end{equation}
On the other hand, if $0<\theta<\frac{1}{2}$, \eqref{equi1}, Plancherel's identity and Proposition \ref{decaysimplyeq} yield
\begin{equation}\label{linearestimeq0.2}
 \begin{aligned}
\|\mathcal{D}_{\xi}^{\theta}\big(e^{it\sign(\xi)\eta^2}\widehat{f}\,\big)\|_{L^2}\lesssim \|\langle x \rangle^{\theta}U_1(t)f\|_{L^2} \lesssim \|\langle x \rangle^{\theta}f\|_{L^2}.    
 \end{aligned}   
\end{equation}
If $\theta=\frac{1}{2}$, we further assume $\mathcal{H}_xf\in L^2(|x|\, dx)$. Then it follows from Proposition \ref{decaysimplyeq} that
\begin{equation}\label{linearestimeq0.3}
 \begin{aligned}
\|\mathcal{D}_{\xi}^{\theta}\big(e^{it\sign(\xi)\eta^2}\widehat{f}\,\big)\|_{L^2}\lesssim \|\langle x \rangle^{\theta}U_1(t)f\|_{L^2} \lesssim \|\langle x \rangle^{\theta}f\|_{L^2}+\|\langle x \rangle^{\theta}\mathcal{H}_xf\|_{L^2}.    
 \end{aligned}   
\end{equation}
Plugging \eqref{linearestimeq0.1}, \eqref{linearestimeq0.2} and \eqref{linearestimeq0.3} into \eqref{linearestimeq0}, we complete the deduction of \eqref{weightineq} and \eqref{weightineq2}.
\\ \\
\underline{\bf (b) Assume $\frac{1}{2}<\theta <1$}. By Plancherel's identity, \eqref{equi1}, and \eqref{prelimneq} applied to the product $e^{it\xi|\xi|}e^{it\sign(\xi)\eta^2}\widehat{f}$, we get 
\begin{equation}\label{linearestimeq0.4}
\begin{aligned}
 \||x|^{\theta}U(t)f\|_{L^2}\lesssim & \|e^{it\omega}\widehat{f}\,\|_{L^2}+\|\|\mathcal{D}_{\xi}^{\theta}\big(e^{it\omega}\widehat{f}\,\big)\|_{L^{2}_{\xi}}\|_{L^2_{\eta}}\\
\lesssim & \|f\|_{L^2}+\|\mathcal{D}_{\xi}^{\theta}\big(\widehat{f}\,\big)\|_{L^2}+\|\|\mathcal{D}_{\xi}^{\theta}\big(e^{it\xi|\xi|}\big)\widehat{f}\,\|_{L^{2}_{\xi}}\|_{L^2_{\eta}}+\|\|\mathcal{D}_{\xi}^{\theta}\big(e^{it\sign(\xi)\eta^2}\big)\widehat{f}\,\|_{L^{2}_{\xi}}\|_{L^2_{\eta}}  \\
\lesssim & \|\langle x \rangle^{\theta}f\|_{L^2}+\langle t\rangle^{\theta}\|\langle \xi \rangle^{\theta}\widehat{f}\,\|_{L^2}+\|\min\{1,|t|\eta^2\}|\xi|^{-\theta}\widehat{f}\,\|_{L^2}.  
\end{aligned}    
\end{equation}
Clearly, when $\eta=0$, we have $\min\{1,|t|\eta^2\}|\xi|^{-\theta}\widehat{f}(\xi,\eta)=0$ for all $\xi\neq 0$. Moreover, by hypothesis \eqref{hiphote1}, it follows $\widehat{f}(0,\eta)=0$, for all $\eta\in \mathbb{Z}\setminus\{0\}$. Thus, Proposition \ref{decaywight} shows
\begin{equation*}
\begin{aligned}
\|\min\{1,|t|\eta^2\}|\xi|^{-\theta}\widehat{f}\,\|_{L^2}=&\|\|\min\{1,|t|\eta^2\}|\xi|^{-\theta}\widehat{f}(\xi,\eta)\|_{L^2_{\xi}}\|_{L^2_{\eta}(\mathbb{Z}\setminus\{0\})}\\
\leq & \|\,\||\xi|^{-\theta}\widehat{f}(\xi,\eta)\|_{L^2_{\xi}}\|_{L^2_{\eta}(\mathbb{Z}\setminus\{0\})}\\
\lesssim & \|\|\widehat{f}(\xi,\eta)\|_{H^{\theta}_{\xi}}\|_{L^2_{\eta}(\mathbb{Z}\setminus \{0\})}\\
\lesssim & \|\widehat{f}\,\|_{H^{\theta,0}}\sim \|\langle x \rangle^{\theta}f\|_{L^2}.
\end{aligned}    
\end{equation*}
Gathering the previous estimates, we deduce \eqref{weightineq} for $\frac{1}{2}<\theta<1$.
\begin{rem}\label{remarkHilbertL2}
The arguments in the deduction of the case (b) above and Corollary \ref{signcoro} show that for a function $f\in L^2(\mathbb{R}\times \mathbb{T})\cap L^2(|x|^{2\theta}\, dxdy)$, $\frac{1}{2}<\theta<1$, such that $\widehat{f}(0,\eta)=0$ for all $\eta\in \mathbb{Z}$, it follows
\begin{equation*}
\|\langle x \rangle^{\theta}\mathcal{H}_xf\|_{L^2}\lesssim \|\langle x \rangle^{\theta}f\|_{L^2}.   
\end{equation*}
To verify this result, we apply Plancherel's identity together with \eqref{equi1} and \eqref{prelimneq} to get
\begin{equation*}
\begin{aligned}
\|\langle x \rangle^{\theta}\mathcal{H}_xf\|_{L^2}\lesssim & \|\sign(\xi)\widehat{f}(\xi,\eta)\|_{L^2}+\|\|\mathcal{D}_{\xi}^{\theta}(\sign(\xi)\widehat{f}(\xi,\eta))\|_{L^2_{\xi}}\|_{L^2_{\eta}} \\
    \lesssim & \|f\|_{L^2}+\|\|\sign(\xi)\mathcal{D}_{\xi}^{\theta}(\widehat{f}(\xi,\eta))\|_{L^2_{\xi}}\|_{L^2_{\eta}}+\|\|\widehat{f}(\xi,\eta)\mathcal{D}_{\xi}^{\theta}(\sign(\xi))\|_{L^2_{\xi}}\|_{L^2_{\eta}}\\
\lesssim & \|\langle x \rangle^{\theta} f\|_{L^2}+\|\|\, |\xi|^{-\theta}\widehat{f}(\xi,\eta)\|_{L^2_{\xi}}\|_{L^2_{\eta}},
\end{aligned}    
\end{equation*}
where we have also used  Corollary \ref{signcoro} to estimate $\mathcal{D}_{\xi}^{\theta}(\sign(\xi))$. Now, since $\widehat{f}(0,\eta)=0$ for all $\eta\in \mathbb{Z}$, we can apply Proposition \ref{decaywight} to deduce
\begin{equation*}
\begin{aligned}
\|\|\, |\xi|^{-\theta} \widehat{f}(\xi,\eta)\|_{L^2_{\xi}}\|_{L^2_{\eta}}\lesssim \|\|\widehat{f}(\xi,\eta)\|_{H^{\theta}_{\xi}}\|_{L^2_{\eta}} \lesssim \|\langle x\rangle^{\theta} f\|_{L^2}.   
\end{aligned}    
\end{equation*}
Collecting the above estimates, we obtain the desired inequality.
\end{rem}
The previous remark will be useful to deduce part (d) below. Next, we deal with (c).
\\ \\
\underline{\bf (c) Assume $1\leq \theta <\frac{3}{2}$}. We recall identity  \eqref{distribderiv1}, i.e.,
\begin{equation*}
 \frac{\partial}{\partial_{\xi}}(e^{it \omega }\widehat{f}\,)=2i\sin(\eta^2 t)\widehat{f}(0,\eta)\delta_{\xi}+e^{it\omega}(2i t|\xi|)\widehat{f}(\xi,\eta)+e^{it\omega }\frac{\partial}{\partial \xi}\widehat{f}(\xi,\eta). 
\end{equation*}
By hypothesis \eqref{hiphote1}, we have that the term with $\delta_{\xi}$ in the identity above equals zero whenever $\eta\in \mathbb{Z}\setminus\{0\}$. When $\eta=0$, this follows from the cancellation of the function $\sin(\eta^2 t)$. Consequently, writing $\theta=1+\theta_1$ with $0\leq \theta_1<\frac{1}{2}$, Plancherel's identity and the previous discussion show
\begin{equation}\label{linearestimeq1}
 \begin{aligned}
 \||x|^{\theta}U(t)f\|_{L^2}\leq \|D_{\xi}^{\theta_1}\partial_{\xi}\big(e^{it\omega}\widehat{f}\,\big)\|_{L^2}\lesssim \|D^{\theta_1}_{\xi}\big(e^{it\omega}(2i|\xi|t)\widehat{f}(\xi,\eta)\big)\|_{L^2}+\|D_{\xi}^{\theta_1}\big(e^{it\omega }\partial_\xi\widehat{f}(\xi,\eta)\big)\|_{L^2}.
 \end{aligned}   
\end{equation}
In the case $\theta_1=0$, we assume $D^{\theta_1}_{\xi}$ and $\mathcal{D}^{\theta_1}_{\xi}$ are the identity operators. Let us estimate the right-hand side of the previous inequality. We use the equivalence \eqref{equi1}, the type of product rule \eqref{prelimneq}, and Lemma \ref{derivexp} to deduce
\begin{equation}\label{linearestimeq1.1}
\begin{aligned}
\|D^{\theta_1}_{\xi}\big(e^{it\omega}(2i|\xi|t)\widehat{f}\,\big)\|_{L^2}\lesssim & \|(2it|\xi|)\widehat{f}\,\big\|_{L^2}+\|\,\|\mathcal{D}_{\xi}^{\theta_1}\big(e^{it\omega}(2it|\xi|)\widehat{f}\,\big)\|_{L^2_{\xi}}\|_{L^2_{\eta}}\\
\lesssim & \|(2it|\xi|)\widehat{f}\,\|_{L^2}+\|\mathcal{D}_{\xi}^{\theta_1}\big(e^{it\sign(\xi)\xi^2}\big)(2it|\xi|)\widehat{f}\,\|_{L^2}\\
&+\|\mathcal{D}_{\xi}^{\theta_1}\big(e^{it\sign(\xi)\eta^2}\big)(2it|\xi|)\widehat{f}\,\|_{L^2}+\|\mathcal{D}_{\xi}^{\theta_1}\big((2it|\xi|)\widehat{f}\,\big)\|_{L^2}\\
\lesssim & \langle t\rangle\|f\|_{H^{1,0}}+\langle t\rangle^{1+\theta_1}\|\langle\xi \rangle^{\theta_1}|\xi|\widehat{f}\,\|_{L^2}+\langle t \rangle\|\min\{1,\eta^2 |t|\}\frac{1}{|\xi|^{\theta_1}}|\xi|\widehat{f}\,\|_{L^2}\\
&+\langle t \rangle\|\mathcal{D}_{\xi}^{\theta_1}\big(|\xi|\widehat{f}\,\big)\|_{L^2},
\end{aligned}    
\end{equation}
where we have also used Plancherel's identity and the fact $|t|\leq \langle t\rangle$. Since $\theta=1+\theta_1$, we get
\begin{equation*}
\begin{aligned}
\langle t\rangle^{1+\theta_1}\|\langle\xi \rangle^{\theta_1}|\xi|\widehat{f}\,\|_{L^2}+\langle t \rangle\|\min\{1,\eta^2 |t|\}\frac{1}{|\xi|^{\theta_1}}|\xi|\widehat{f}\,\|_{L^2}\lesssim  \langle t\rangle^{\theta}\|f\|_{H^{\theta,0}}.  
\end{aligned}    
\end{equation*}
Next, we use \eqref{prelimneq}, \eqref{prelimneq1}, and interpolation Proposition \ref{interpo} to get
\begin{equation}\label{linearestimeq3.1}
\begin{aligned}
\|\mathcal{D}_{\xi}^{\theta_1}\big(|\xi|\widehat{f}\,\big)\|_{L^2}=\|\mathcal{D}_{\xi}^{\theta_1}\Big(\frac{|\xi|}{\langle\xi\rangle}\langle\xi\rangle\widehat{f}\,\Big)\|_{L^2} \lesssim & \|\mathcal{D}_{\xi}^{\theta_1}\Big(\frac{|\xi|}{\langle\xi\rangle}\Big)\|_{L^{\infty}}\|\langle\xi\rangle\widehat{f}\,\|_{L^2}+\|J^{\theta_1}\big(\langle\xi\rangle\widehat{f}\,\big)\|_{L^2} \\
 \lesssim & \|f\|_{H^{1,0}}+\|\|\langle\xi\rangle^{1+\theta_1}\widehat{f}\,\|_{L^2_{\xi}}^{\frac{1}{1+\theta_1}}\|J_{\xi}^{1+\theta_1}\widehat{f}\,\|_{L^2_{\xi}}^{\frac{\theta_1}{1+\theta_1}}\|_{L^2_{\eta}}\\  
\lesssim & \|f\|_{H^{\theta,0}}+\|\langle x\rangle^{\theta}f\|_{L^2},
\end{aligned}    
\end{equation}
where we have used Young's inequality $ab\lesssim a^{p_1}+b^{p_1}$, $\frac{1}{p_1}+\frac{1}{p_2}=1$, $a,b\geq 0$ and Minkowski inequality in $L^2_{\xi}(\mathbb{R})$ together with Plancherel's identity. Summarizing the previous estimates, we obtain
\begin{equation}\label{linearestimeq2}
\|D^{\theta_1}_{\xi}\big(e^{it\omega}(2i|\xi|t)\widehat{f}\,\big)\|_{L^2}\lesssim \langle t \rangle^{\theta}\big(\|f\|_{H^{\theta,0}}+\|\langle x\rangle^{\theta}f\|_{L^2}\big).   
\end{equation}
It remains to control the second factor on the right-hand side of \eqref{linearestimeq1}. If $\theta_1=0$, at once we deduce
\begin{equation}\label{linearestimeq3}
 \begin{aligned}
 \|e^{it\omega }\frac{d}{d\xi}\widehat{f}\,\|_{L^2}\leq \|\langle x\rangle f\|_{L^2}\leq \|\langle x\rangle^{\theta} f\|_{L^2}.   
 \end{aligned}   
\end{equation}
If $0<\theta_1<\frac{1}{2}$, by Plancherel's identity and \eqref{weightineq} (which we deduced in (a) above), it is seen that
\begin{equation}\label{linearestimeq4}
 \begin{aligned}
 \|D_{\xi}^{\theta_1}\big(e^{it\omega }\frac{d}{d\xi}\widehat{f}\,\big)\|_{L^2}\sim\||x|^{\theta_1}U(t)(xf)\|_{L^2}  \lesssim & \langle t \rangle^{\theta_1}\big(\|\langle x\rangle^{1+\theta_1} f\|_{L^2}+\|\|J^{\theta_1}_x\big(\langle x\rangle f\big)\|_{L^2_x}\|_{L^2_y}\big)\\
\lesssim & \langle t \rangle^{\theta_1}\big(\|\langle x\rangle^{1+\theta_1} f\|_{L^2}+\|\|\langle x \rangle^{1+\theta_1}f\|_{L^2_x}^{\frac{1}{1+\theta_1}}\|J^{1+\theta_1}_x f\|_{L^2_x}^{\frac{\theta_1}{1+\theta_1}}\|_{L^2_y}\big)\\
\lesssim & \langle t \rangle^{\theta}\big(\|\langle x\rangle^{\theta} f\|_{L^2}+\| f\|_{H^{\theta,0}}\big),
 \end{aligned}   
\end{equation}
where we used interpolation Proposition \ref{interpo} and Young's inequality as before.  Finally, inserting \eqref{linearestimeq2}, \eqref{linearestimeq3} and \eqref{linearestimeq4} in \eqref{linearestimeq1}, we obtain \eqref{weightineq} for the present restrictions on $\theta$ and $f$.
\\ \\
\underline{\bf (d) $k+\frac{1}{2}\leq \theta<k+1$, $k\in \mathbb{Z}^{+}$}. By Plancherel's identity, the estimate for $|x|^{\theta}U(t)f\in L^2(\mathbb{R}\times\mathbb{T})$ is equivalent to the estimate of $D^{\theta_1}_{\xi}\partial_{\xi}^k(e^{it\omega}\widehat{f}\,)\in L^2(\mathbb{R}\times\mathbb{Z})$, where $\frac{1}{2}\leq \theta_1:=\theta-k<1$, or equivalently, we must estimate $\partial_{\xi}^k(e^{it\omega}\widehat{f}\,)\in H^{\theta_1,0}_{\xi}(\mathbb{R}\times \mathbb{Z})$. We have from hypothesis \eqref{hiphote2} that $\partial_{\xi}^l\widehat{f}(0,\eta)=0$, for all $\eta\in \mathbb{Z}\setminus\{0\}$, and each $l=1,\dots,k-1$, and if $k\geq 2$, it follows from \eqref{hiphote2.1} that $\partial_{\xi}^{l}\widehat f(0,0)=0$, for each $l=0,1,\dots,k-2$. Thus, identity in Proposition \ref{generalderiv} reduces to
\begin{equation}\label{derivativesimpl}
\begin{aligned}
 \frac{\partial^k}{\partial \xi^k}\big(e^{it\omega}\widehat{f}\,\big)=&e^{it \omega }\sum_{m=1}^k \sum_{\substack{l_1+\dots+l_{m}+l_{m+1}=k-m\\ 0\leq l_j\leq 1\\ j=1,\dots,m}}c_{l_1,\dots,l_{m+1}}\partial_{\xi}^{l_1}(2it|\xi|)\dots \partial_{\xi}^{l_m}(2it|\xi|)\partial_{\xi}^{l_{m+1}}\widehat{f}\\
&+ e^{it \omega} \partial_{\xi}^{k}\widehat{f},
\end{aligned}    
\end{equation}
where we have also used that $\sin(\eta^2 t)=0$, when $\eta=0$. Note that up to this point, condition $\partial_{\xi}^{l}\widehat f(0,0)=0$ is only necessary for $l=0,1,\dots,k-3$. To deduce \eqref{weightineq} for the present case, we will estimate the $H^{\theta_1,0}_{\xi}(\mathbb{R}\times \mathbb{Z})$-norm of \eqref{derivativesimpl}. We first study each factor composing the sum of the first term on the right-hand side of \eqref{derivativesimpl}. We consider two extra cases.
\\ \\
\underline{\bf Assume that $l_1=l_2=\dots=l_m=1$}. It follows, $l_{m+1}=k-2m$ and
\begin{equation*}
\begin{aligned}
\|D^{\theta_1}_{\xi}\big(e^{it\omega}\partial_{\xi}^{l_1}(2it|\xi|)\dots \partial_{\xi}^{l_m}(2it|\xi|)\partial_{\xi}^{l_{m+1}}\widehat{f}\,\big)\|_{L^2}=\|D^{\theta_1}_{\xi}\big(e^{it\omega}(2it\sign(\xi))^m\partial^{k-2m}_{\xi}\widehat{f}\,\big)\|_{L^2}.   
\end{aligned}    
\end{equation*}
Hence, when $m$ is even, we get
\begin{equation}\label{firstdecomp1}
\begin{aligned}
\|D^{\theta_1}_{\xi}\big(e^{it\omega}(2it\sign(\xi))^m\partial_{\xi}^{k-2m}\widehat{f}\,\big)\|_{L^2} \lesssim &\langle t\rangle^m \|D^{\theta_1}_{\xi}\big(e^{it\omega}\partial^{k-2m}_{\xi}\widehat{f}\,\big)\|_{L^2}\\
\lesssim &\langle t\rangle^m \|\langle x \rangle^{\theta_1}U(t)(x^{k-2m}f)\|_{L^2},
\end{aligned}    
\end{equation}
and when $m$ is odd, 
\begin{equation}\label{firstdecomp2}
\begin{aligned}
\|D^{\theta_1}_{\xi}\big(e^{it\omega}(2it\sign(\xi))^m\partial_{\xi}^{k-2m}\widehat{f}\,\big)\|_{L^2} \lesssim &\langle t\rangle^m \|D^{\theta_1}_{\xi}\big(e^{it\omega}\sign(\xi)\partial_{\xi}^{k-2m}\widehat{f}\,\big)\|_{L^2}\\
\lesssim &\langle t\rangle^m \|\langle x \rangle^{\theta_1}U(t)\mathcal{H}_x(x^{k-2m}f)\|_{L^2}.
\end{aligned}    
\end{equation}
We proceed with the estimates of \eqref{firstdecomp1} and \eqref{firstdecomp2} in the cases: $\frac{1}{2}<\theta_1<1$, and $\theta_1=\frac{1}{2}$. When $\frac{1}{2}<\theta_1<1$, since $m\geq 1$, \eqref{hiphote2} shows that $x^{k-2m}f$ satisfies the hypothesis of $(iii)$ in Theorem \ref{LinearEst} (i.e., it satisfies the results proven in (b) and (c) above). Thus, we deduce
\begin{equation}\label{firstdecomp3}
 \begin{aligned}
 \|\langle x \rangle^{\theta_1}U(t)(x^{k-2m}f)\|_{L^2} \lesssim & \langle t \rangle^{\theta_1}\big(\|\langle x \rangle^{\theta_1}x^{k-2m}f\|_{L^2}+\|\|J^{\theta_1}_{x}(x^{k-2m}f)\|_{L^2_x}\|_{L^2_y}\big) \\
 \lesssim & \langle t \rangle^{\theta_1}\big(\|\langle x \rangle^{\theta_1}x^{k-2m}f\|_{L^2}+\|\|J^{\theta_1}_{x}(\langle x\rangle^{k-2m}f)\|_{L^2_x}\|_{L^2_y}\big) \\
  \lesssim & \langle t \rangle^{\theta_1}\big(\|\langle x \rangle^{\theta_1+k-2m}f\|_{L^2}+\|f\|_{H^{\theta,0}}\big)  , 
 \end{aligned}   
\end{equation}
where to control $J^{\theta_1}_{x}(x^{k-2m}f)\in L^{2}(\mathbb{R}\times \mathbb{T})$, we have used interpolation and Young's inequality (this argument is rather similar to \eqref{linearestimeq3.1}). Next, we recall  that for a function $g$ sufficiently regular with enough decay such that $\widehat{g}(0,\eta)=0$ for all $\eta\in \mathbb{Z}$, it follows 
\begin{equation}\label{conmm}
[\mathcal{H}_x,x]g=0.    
\end{equation}
The assumptions \eqref{hiphote2} and \eqref{hiphote2.1} imply that $\widehat{(x^{l}f)}(0,\eta)=0$ for all $\eta\in \mathbb{Z}$, and each $l=1,\dots,k-2$, thus \eqref{conmm} yields
\begin{equation}\label{conmuthilber}
  \mathcal{H}_x(x^l f)=x\mathcal{H}_x(x^{l-1}f)=\dots=x^l\mathcal{H}_x f\in L^{2}(\mathbb{R}\times \mathbb{T}),  
\end{equation}
for all $l=1,\dots,k-2$. Note that since $k-2m\leq k-2$, in particular $\widehat{(x^{k-2m}f)}(0,\eta)=0$ for all $\eta\in \mathbb{Z}$. Hence, this fact and $\frac{1}{2}<\theta_1<1$ allow us to apply Theorem \ref{LinearEst} (iii) (which is part (b) above), the result in Remark \ref{remarkHilbertL2} and complex interpolation Proposition \ref{interpo} to infer
\begin{equation}\label{firstdecomp4}
 \begin{aligned}
 \|\langle x \rangle^{\theta_1}U(t)\mathcal{H}_x(x^{k-2m}f)\|_{L^2} \lesssim & \langle t \rangle^{\theta_1}\big(\|\langle x \rangle^{\theta_1}\mathcal{H}_x(x^{k-2m}f)\|_{L^2}+\|\|J^{\theta_1}_{x}\mathcal{H}_x(x^{k-2m}f)\|_{L^2_x}\|_{L^2_y}\big) \\
 \lesssim & \langle t \rangle^{\theta_1}\big(\|\langle x \rangle^{\theta_1}(x^{k-2m}f)\|_{L^2}+\|\|J^{\theta_1}_{x}(\langle x\rangle^{k-2m}f)\|_{L^2_x}\|_{L^2_y}\big) \\
  \lesssim & \langle t \rangle^{\theta_1}\big(\|\langle x \rangle^{\theta_1+k-2m}f\|_{L^2}+\|f\|_{H^{\theta,0}}\big).
 \end{aligned}   
\end{equation}
Next, we consider $\theta_1=\frac{1}{2}$ in \eqref{firstdecomp1} and \eqref{firstdecomp2}. In this case, we will apply (ii) in Theorem \ref{LinearEst}. For this, let us first show that $|x|^{\frac{1}{2}+k-2m}\mathcal{H}_xf\in L^2(\mathbb{R}\times \mathbb{T})$, i.e., by the equivalence in \eqref{equi1} (see also \eqref{equi2}), and \eqref{conmuthilber}, it is enough to show $\mathcal{D}_{\xi}^{\frac{1}{2}}(\sign(\xi)\partial_{\xi}^{k-2m}\widehat{f}\,)\in L^2(\mathbb{R}\times \mathbb{Z})$. We use \eqref{prelimneq}, Corollary \ref{signcoro} and Plancherel's identity to get 
\begin{equation*}
 \begin{aligned}
 \|\mathcal{D}_{\xi}^{\theta_1}(\sign(\xi)\partial_{\xi}^{k-2m}\widehat{f}\,)\|_{L^2}\lesssim & \|\|\mathcal{D}_{\xi}^{\theta_1}(\sign(\xi))\partial_{\xi}^{k-2m}\widehat{f}\,\|_{L^2_{\xi}}\|_{L^2_{\eta}}+\|\|J_{\xi}^{\theta_1}\partial_{\xi}^{k-2m}\widehat{f}\,\|_{L^2_{\xi}}\|_{L^2_{\eta}}\\
 \lesssim & \|\||\xi|^{-\frac{1}{2}}\partial_{\xi}^{k-2m}\widehat{f}\,\|_{L^2_{\xi}}\|_{L^2_{\eta}}+\|\langle x \rangle^{\theta}f\|_{L^2}.
 \end{aligned}   
\end{equation*}
To estimate the first term on the right-hand side of the inequality above, we use again that $\partial_{\xi}^{k-2m}\widehat{f}(0,\eta)=0$ for all $\eta \in \mathbb{Z}$ (which is a consequence of \eqref{hiphote2} and \eqref{hiphote2.1}), and we split the integration on $L^2_{\xi}(\mathbb{R})$ to use the mean value theorem to find
\begin{equation*}
\begin{aligned}
\|\||\xi|^{-\frac{1}{2}}\partial_{\xi}^{k-2m}\widehat{f}\,\|_{L^2_{\xi}}\|_{L^2_{\eta}}\leq & \|\||\xi|^{-\frac{1}{2}}\big(\partial_{\xi}^{k-2m}\widehat{f}(\xi,\eta)-\partial_{\xi}^{k-2m}\widehat{f}(0,\eta)\big)\|_{L^2_{\xi}(|\xi|\leq 1)}\|_{L^2_{\eta}}+\|\|\partial_{\xi}^{k-2m}\widehat{f}\,\|_{L^2_{\xi}(|\xi|> 1)}\|_{L^2_{\eta}} \\
\lesssim & \|\||\xi|^{\frac{1}{2}}\|_{L^2_{\xi}(|\xi|\leq 1)}\|\partial_{\xi}^{k+1-2m}\widehat{f}(\cdot,\eta)\|_{L^{\infty}_{\xi}}\|_{L^2_{\eta}}+\|\partial_{\xi}^{k-2m}\widehat{f}\,\|_{L^2}\\
\lesssim & \|\|\partial_{\xi}^{k+1-2m}\widehat{f}(\cdot,\eta)\|_{H^{1}_{\xi}}\|_{L^2_{\eta}}+\|\partial_{\xi}^{k-2m}\widehat{f}\,\|_{L^2}\\
\lesssim & \|\langle x \rangle^{\theta}f\|_{L^2}.
\end{aligned}
\end{equation*}
Above we used Sobolev embedding $H^{1}_{\xi}(\mathbb{R})\hookrightarrow L^{\infty}_{\xi}(\mathbb{R})$ and the fact that $k+2-2m\leq \theta$ provided that $m\geq 1$. Gathering the previous result, we conclude
\begin{equation}\label{Hdecay}
 \|  \langle x\rangle^{\frac{1}{2}+k-2m}\mathcal{H}_x f\|_{L^2}\lesssim \|\langle x \rangle^{\theta}f\|_{L^2}.
\end{equation}
Consequently, \eqref{Hdecay} and \eqref{conmuthilber} show that $x^{k-2m}f$ satisfies the hypothesis in Theorem \ref{LinearEst} $(ii)$, which we deduce in step (a) above. Hence, we can use similar ideas as in \eqref{firstdecomp3} and \eqref{firstdecomp4} to conclude
\begin{equation*}
 \begin{aligned}
\|\langle x\rangle^{\frac{1}{2}}U(t)(x^{k-2m}f)\|_{L^2}+\|\langle x\rangle^{\frac{1}{2}}U(t)\mathcal{H}_x(x^{k-2m}f)\|_{L^2} \lesssim \langle t \rangle^{\frac{1}{2}}\big(\|\langle x \rangle^{\theta}f\|_{L^2}+\|f\|_{H^{\theta,0}}\big).
 \end{aligned}   
\end{equation*}
This completes the case $\theta_1=\frac{1}{2}$. Going back to \eqref{firstdecomp1} and \eqref{firstdecomp2}, all of the above estimates can be summarized as follows 
\begin{equation*}
 \|D^{\theta_1}_{\xi}\big(e^{it\omega}\partial_{\xi}^{l_1}(2it|\xi|)\dots \partial_{\xi}^{l_m}(2it|\xi|)\partial_{\xi}^{l_{m+1}}\widehat{f}\,\big)\|_{L^2}\lesssim \langle t \rangle^{\theta}\big(\|\langle x \rangle^{\theta}f\|_{L^2}+\|f\|_{H^{\theta,0}}\big),   
\end{equation*}
whenever $k+\frac{1}{2}\leq \theta< k+1$, $\theta_1=\theta-k$, $m=1,\dots,k$, $l_1+\dots+l_{m+1}=k-m$, and $l_1=\dots=l_m=1$. 
\\ \\
\underline{\bf Assume that $l_j=0$ for some $j=1,\dots,m$}. To simplify notation, let us denote $$g_m(t,\xi):=\partial_{\xi}^{l_1}(2it|\xi|)\dots \partial_{\xi}^{l_m}(2it|\xi|).$$
Since $0\leq l_k\leq 1$, and there exists some $l_j=0$, it follows $m-\sum_{r=1}^m l_r\geq 1$, and we have that
\begin{equation}\label{decaycond}
 |g_m(t,\xi)|\lesssim \langle t \rangle^{m}|\xi|^{m-\sum_{r=1}^m l_r} .
\end{equation}
Now, by the arguments in \eqref{linearestimeq0.4} and \eqref{decaycond}, we have 
\begin{equation}\label{linearestimeq4.1}
\begin{aligned}
\|D^{\theta_1}_{\xi}\big(e^{it\omega}g_m(t,\xi)\partial_{\xi}^{l_{m+1}}\widehat{f}\,\big)\|_{L^2}\lesssim & \langle t\rangle^{m+\theta_1}\big( \|J^{\theta_1}_{\xi}\big(g_m(1,\xi)\partial_{\xi}^{l_{m+1}}\widehat{f}\big)\,\|_{L^2}+\|\langle \xi\rangle^{\theta_1} g_m(1,\xi)\partial_{\xi}^{l_{m+1}}\widehat{f}\,\|_{L^2}\\
&+\|\min\{1,|t|\eta^2\}|\xi|^{-\theta_1} g_m(1,\xi)\partial_{\xi}^{l_{m+1}}\widehat{f}\,\|\big)\\
\lesssim& \langle t\rangle^{m+\theta_1}\big( \|J^{\theta_1}_{\xi}\big(g_m(1,\xi)\partial_{\xi}^{l_{m+1}}\widehat{f}\big)\,\|_{L^2}+\|\langle \xi\rangle^{\theta_1+m-\sum_{r=1}^m l_r} \partial_{\xi}^{l_{m+1}}\widehat{f}\,\|_{L^2}\\
&+\|\langle \xi\rangle^{m-\sum_{r=1}^m l_r-\theta_1} \partial_{\xi}^{l_{m+1}}\widehat{f}\,\|_{L^2}\big).
\end{aligned}    
\end{equation}
To complete the study of the previous estimate, it is not hard to see that several applications of the identity $\langle \xi\rangle^{\nu}\partial_{\xi}h=\partial_{\xi}(\langle \xi\rangle^{\nu} h)-\partial_{\xi}(\langle \xi\rangle^{\nu})h$, the fact that $|\partial_{\xi}(\langle \xi\rangle^{\nu})|\lesssim \langle \xi\rangle^{\nu-1}$, and interpolation Proposition \ref{interpo} establish 
\begin{equation}\label{linearestimeq4.2}
\begin{aligned}
    \|\langle \xi\rangle^{\theta_1+m-\sum_{r=1}^m l_r} \partial_{\xi}^{l_{m+1}}\widehat{f}\,\|_{L^2}+\|\langle \xi\rangle^{m-\sum_{r=1}^m l_r-\theta_1} \partial_{\xi}^{l_{m+1}}\widehat{f}\,\|_{L^2}\lesssim &\|\|J^{l_{m+1}}_{\xi}\big(\langle \xi\rangle^{\theta_1+m-\sum_{r=1}^m l_r} \widehat{f}\,\big)\|_{L^2_{\xi}}\|_{L^2_{\eta}} \\
\lesssim &\|\|\langle \xi\rangle^{\theta} \widehat{f}\,\|_{L^2_{\xi}}^{\frac{\theta-l_{m+1}}{\theta}}\|J^{\theta}_{\xi}\widehat{f}\,\|_{L^2_{\xi}}^{\frac{l_{m+1}}{\theta}}\|_{L^2_{\eta}}\\
\lesssim& \|\langle x \rangle^{\theta}f\|_{L^2}+\|f\|_{H^{\theta,0}}.
\end{aligned}    
\end{equation}
On the other hand, following a similar argument as in \eqref{linearestimeq3.1}, using \eqref{prelimneq}, \eqref{prelimneq1}, and the fact that the function $g_m(1,\xi)\langle \xi \rangle^{-m+\sum_{r=1}^m l_r}$ is bounded with bounded derivative, we get 
\begin{equation*}
\begin{aligned}
\|J^{\theta_1}_{\xi}\big(g_m(1,\xi)\partial_{\xi}^{l_{m+1}}\widehat{f}\big)\,\|_{L^2}=& \|J^{\theta_1}_{\xi}\Big(\frac{g_m(1,\xi)}{\langle \xi \rangle^{m-\sum_{r=1}^m l_r}} \langle \xi \rangle^{m-\sum_{r=1}^m l_r} \partial_{\xi}^{l_{m+1}} \widehat{f}\Big)\|_{L^2}\\
\lesssim & \|J^{\theta_1+l_{m+1}}_{\xi}\big(\langle \xi \rangle^{m-\sum_{r=1}^m l_r}  \widehat{f}\big)\|_{L^2}\\
\lesssim& \|\langle x \rangle^{\theta}f\|_{L^2}+\|f\|_{H^{\theta,0}},
\end{aligned}    
\end{equation*} 
where we also argued as in \eqref{linearestimeq4.2}, and once again we applied interpolation in the $\xi$-variable. This resolves the estimate of the case $l_j=0$ for some $j=1,\dots,m$.

Consequently, we have completed the estimate of the first term on the right-hand side of \eqref{derivativesimpl}. Next, we study the second term on the right-hand side of \eqref{derivativesimpl}, i.e., we study $e^{it\omega}\partial_{\xi}^k f\in H_{\xi}^{\theta_1,0}(\mathbb{R}\times \mathbb{Z})$.

Since $\frac{1}{2}\leq \theta_1<1$, the hypothesis \eqref{hiphote3} and \eqref{hiphote4} imply that $x^kf$ satisfies the conditions in $(ii)$ and $(iii)$ in Theorem \ref{LinearEst} with $\theta_1=\theta-k$. Thus, when $\frac{1}{2}< \theta_1<1$, we use the results of parts (a), (b), and (c) above to get
\begin{equation*}
\begin{aligned}
 \|D_{\xi}^{\theta_1}\big( e^{it \omega} \partial_{\xi}^{k}\widehat{f}\,\big)\|_{L^2}\lesssim \|\langle x\rangle^{\theta_1}U(t)(x^kf)\|_{L^2}\lesssim& \langle t \rangle^{\theta_1}\big(\|\langle x\rangle^{\theta_1}(x^kf)\|_{L^2}+ \|\|J^{\theta_1}(x^k f)\|_{L^2_{x}}\|_{L^2_y}\big)\\
 \lesssim & \langle t \rangle^{\theta_1}\big(\|\langle x\rangle^{\theta_1+k}f\|_{L^2}+ \|\|J^{\theta_1}(\langle x\rangle^k f)\|_{L^2_{x}}\|_{L^2_y}\big)\\
 \lesssim & \langle t \rangle^{\theta}\big(\|\langle x\rangle^{\theta}f\|_{L^2}+\|f\|_{H^{\theta,0}}\big),
\end{aligned}    
\end{equation*}
where we also used interpolation Proposition \ref{interpo} and Young's inequality (this argument is similar to the interpolation estimate in \eqref{linearestimeq3.1}). Now, when $\theta_1=\frac{1}{2}$, it follows from (b) (which is \eqref{weightineq2}),  \eqref{conmuthilber} and interpolation that
\begin{equation*}
\begin{aligned}
 \|D_{\xi}^{\frac{1}{2}}\big( e^{it \omega} \partial_{\xi}^{k}\widehat{f}\,\big)\|_{L^2}\lesssim \|\langle x\rangle^{\frac{1}{2}}U(t)(x^k f)\|_{L^2}\lesssim& \langle t \rangle^{\frac{1}{2}}\big(\|\langle x\rangle^{\frac{1}{2}}(x^kf)\|_{L^2}+\|\langle x\rangle^{\frac{1}{2}}\mathcal{H}_x(x^kf)\|_{L^2}+ \|\|J^{\frac{1}{2}}(x^k f)\|_{L^2_{x}}\|_{L^2_y}\big)\\
 \lesssim & \langle t \rangle^{\frac{1}{2}}\big(\|\langle x\rangle^{\frac{1}{2}+k}f\|_{L^2}+\|\langle x\rangle^{\frac{1}{2}}\mathcal{H}_x(x^k f)\|_{L^2}+ \|\|J^{\frac{1}{2}}(\langle x\rangle^k f)\|_{L^2_{x}}\|_{L^2_y}\big)\\
 \lesssim & \langle t \rangle^{\frac{1}{2}}\big(\|\langle x\rangle^{\theta}f\|_{L^2}+\|\langle x\rangle^{\frac{1}{2}}\mathcal{H}_x(x^k f)\|_{L^2}+\|f\|_{H^{\theta,0}}\big).
\end{aligned}    
\end{equation*}
This completes the deduction of case (d) and in turn restriction $k+\frac{1}{2}\leq \theta<k+1$ in part $(iv)$ of Theorem \ref{LinearEst}. 
\\ \\
\underline{\bf (e) $k+1\leq \theta<k+\frac{3}{2}$, $k\in \mathbb{Z}^{+}$}. Writing $\theta_1=\theta-k-1$ and following the comments at the beginning of Case (d), to get the current case we must estimate $\partial_{\xi}^{k+1}(e^{it\omega}\widehat{f}\,)\in H^{\theta_1,0}_{\xi}(\mathbb{R}\times \mathbb{Z})$. Proposition \ref{generalderiv} shows
\begin{equation}\label{distriidentity2}
\begin{aligned}
 \partial^{k+1}_{\xi}\big(e^{it\omega}\widehat{f}\,\big)=&\sum_{l=0}^{k}\sum_{m=0}^{k-l}\nu_{l,m}(t)\sin(\eta^2 t)\partial_{\xi}^{l}\widehat{f}(0,\eta)\delta_{\xi}^{(m)} +\sum_{l=0}^{k-2}\sum_{m=0}^{k-2-l} d_{l,m}(t)\cos(\eta^2 t)\partial_{\xi}^l\widehat{f}(0,\eta)\delta_{\xi}^{(m)} \\ 
 &+e^{it \omega }\sum_{m=1}^{k+1} \sum_{\substack{l_1+\dots+l_{m}+l_{m+1}=k+1-m\\ 0\leq l_j\leq 1\\ j=1,\dots,m}}c_{l_1,\dots,l_{m+1}}\partial_{\xi}^{l_1}(2it|\xi|)\dots \partial_{\xi}^{l_m}(2it|\xi|)\partial_{\xi}^{l_{m+1}}\widehat{f}\\
&+ e^{it \omega} \partial_{\xi}^{k+1}\widehat{f}.
\end{aligned}    
\end{equation}
We recall that if $k<2$, the convention above is that $\sum_{l=0}^{k-2}(\cdots)=0$. Now, since \eqref{hiphote2} and \eqref{hiphote4} imply that $\partial_{\xi}^l\widehat{f}(0,\eta)$ for all $\eta\in \mathbb{Z}\setminus\{0\}$, and each $l=1,\dots,k$, and $\sin(\eta^2 t)=0$, when $\eta=0$, we have that the terms with factor $\sin(\eta^2 t)$ in \eqref{distriidentity2} are null. Moreover, when $k\geq 2$, \eqref{hiphote2}, and \eqref{hiphote2.1} show that the terms with $\cos(\eta^2 t)$ in \eqref{distriidentity2} are null. Such cancellation is possible since \eqref{hiphote2.1} states that $\partial^{l}_{\xi}\widehat{f}(0,0)=0$ for all $l=0,1,\dots,k-2$. Summarizing, we have that \eqref{distriidentity2} reduces to 
\begin{equation*}
\begin{aligned}
 \partial^{k+1}_{\xi}\big(e^{it\omega}\widehat{f}\,\big)= &e^{it \omega }\sum_{m=1}^{k+1} \sum_{\substack{l_1+\dots+l_{m}+l_{m+1}=k+1-m\\ 0\leq l_j\leq 1\\ j=1,\dots,m}}c_{l_1,\dots,l_{m+1}}\partial_{\xi}^{l_1}(2it|\xi|)\dots \partial_{\xi}^{l_m}(2it|\xi|)\partial_{\xi}^{l_{m+1}}\widehat{f}\\
&+ e^{it \omega} \partial_{\xi}^{k+1}\widehat{f}.
\end{aligned}    
\end{equation*}
As a consequence, the above identity allows us to follow similar arguments to those used in the previous case (d) to conclude (e). To avoid repetition, we omit further details.
\end{proof}

%%%%%%%%%%%%%%%%%%%%%%%%%%%%%%%%%%%%%%%%%%%%%%%%%%%%%%%%%%%%%%%%%%%%%%%%%%%%%%%%%%%%%%%%%%%%%%%%%%%%%%%%%%%%%%%%%%%%%%%%%%%%%%%%%%%%%%%%%%%%%%%%%%%%%

\subsection{Unique continuation principles linear equation}

In this section, we present the deduction of the unique continuation principles stated in Theorem  \ref{lineaerequniquecont1} and Corollary \ref{lineaerequniquecont2}. More precisely, Theorem  \ref{lineaerequniquecont1} deals with the study of the cases where given $f\in H^{\theta,0}(\mathbb{R}\times\mathbb{T})\cap L^2(|x|^{\theta}\, dx dy)$, there exists a time $t\neq 0$ such that $U(t)f\in H^{\theta,0}(\mathbb{R}\times\mathbb{T})\cap L^2(|x|^{\theta}\, dx dy)$, where $\theta\neq \frac{1}{2}+k$ for all integer $k\geq 0$. In Corollary \ref{lineaerequniquecont2}, we study the cases $\theta=\frac{1}{2}+k$ for some integer $k\geq 0$.

\subsubsection{Proof of Theorem \ref{lineaerequniquecont1}}

\begin{proof}[Proof of Theorem \ref{lineaerequniquecont1}]
We divide the proof into the following parts:
\begin{itemize}
    \item[(a)] Case $\frac{1}{2}<\theta<\frac{3}{2}$.
    \item[(b)] Case $\frac{3}{2}<\theta<\frac{5}{2}$.
    \item[(c)] Case $\frac{5}{2}<\theta<\frac{7}{2}$.
    \item[(d)] Case  $k+\frac{1}{2}<\theta<k+\frac{3}{2}$ with $k\geq 3$.
\end{itemize}
Let us first discuss the ideas and motivations of the previous cases. (a) and (b) correspond to part $(i)$ in Theorem \ref{lineaerequniquecont1}. In these parts, we show that somehow the dispersive effects of $\mathcal{H}_x\partial_y^2$ are more dominant and responsible for condition \eqref{hiphoteunique1}. In part (c), we deduce $k=2$ in $(ii)$ in Theorem \ref{lineaerequniquecont1}. In this part, we show that \eqref{hiphoteunique2} is imposed by the dispersion $\mathcal{H}_x\partial_x^2$, while  \eqref{hiphoteunique3} is  due to $\mathcal{H}_x\partial_y^2$. Finally, we conclude the general case in part (d), which follows from similar considerations as in (c), and using an inductive argument. Additionally, we remark that for technical reasons in the deductions of (c) and (d) below, we get first \eqref{hiphoteunique3} for all $l=0,\dots,k$, and then we show \eqref{hiphoteunique2}.
\\ \\
\underline{\bf (a) Assume $\frac{1}{2}<\theta<\frac{3}{2}$}. 
We consider a function $\phi\in C^{\infty}_{c}(\mathbb{R})$ with $\phi=1$ in a neighborhood of the origin. Let us first assume that $\frac{1}{2}<\theta<1$. We remark that $U(t)f\in L^2(|x|^{2\theta}\, dx dy)$ if and only if  $e^{it\omega}\widehat{f}\in H^{\theta,0}(\mathbb{R}\times \mathbb{Z})$.  Thus, taking the Fourier transform of $U(t)f$, we write
\begin{equation}\label{uniqueq4}
 \begin{aligned}
  e^{it\omega}\widehat{f}(\xi,\eta)=&\big(e^{i t\xi|\xi|}-1\big)e^{i t \sign(\xi)\eta^2}\widehat{f}(\xi,\eta)+e^{i t \sign(\xi)\eta^2}\big(\widehat{f}(\xi,\eta)-\widehat{f}(0,\eta)\phi(\xi)\big)\\
  &+e^{i t \sign(\xi)\eta^2}\widehat{f}(0,\eta)\phi(\xi).   
 \end{aligned}   
\end{equation}
Using the fractional derivative $\mathcal{D}_{\xi}^{\theta}$, its properties, and Lemma \ref{derivexp}, we have that $\big(e^{i t\xi|\xi|}-1\big)e^{i t \sign(\xi)\eta^2}\widehat{f}(\xi,\eta)\in H^{\theta,0}_{\xi}(\mathbb{R}\times \mathbb{Z})$. Setting $g(\xi,\eta)=\big(\widehat{f}(\xi,\eta)-\widehat{f}(0,\eta)\phi(\xi)\big)$, it follows $g(0,\eta)=0$ for all $\eta\in \mathbb{Z}$. Then, we use \eqref{prelimneq}, \eqref{steinderiv2}, and Proposition \ref{decaywight} to deduce
\begin{equation*}
 \begin{aligned}
 \|\mathcal{D}_{\xi}^{\theta}\big(e^{it\sign(\xi)\eta^2}g\big)\|_{L^2}\lesssim & \|g\|_{L^2}+\|\|\mathcal{D}_{\xi}^{\theta}\big(e^{it\sign(\xi)\eta^2}\big)g\|_{L^2_{\xi}}\|_{L^2_{\eta}}\\
 \lesssim & \|g\|_{L^2}+\|\|\min\{1,(|t|\eta^2)\}\frac{1}{|\xi|^{\theta} }g\|_{L^2_{\xi}}\|_{L^2_{\eta}}\\
  \lesssim & \|g\|_{L^2}+\|\|g\|_{H^{\theta}_{\xi}}\|_{L^2_{\eta}}\lesssim \|g\|_{H^{\theta,0}}\lesssim \|\langle x\rangle^{\theta}f\|,
 \end{aligned}   
\end{equation*}
where we used Sobolev embedding to get $|\widehat{f}(0,\eta)|\lesssim \|J^{\theta}_{\xi}\widehat{f}(\cdot,\eta)\|_{L^2_{\xi}}$, for all $\eta\in \mathbb{Z}$. Going back to \eqref{uniqueq4}, when $\frac{1}{2}<\theta<1$, we have
\begin{equation*}
  e^{it\omega}\widehat{f}(\xi,\eta)\in H^{\theta,0}(\mathbb{R}\times \mathbb{Z})\, \, \text{ if and only if }\, \,  e^{i t \sign(\xi)\eta^2}\widehat{f}(0,\eta)\phi(\xi)\in  H^{\theta,0}(\mathbb{R}\times \mathbb{Z}).
\end{equation*}
Since the hypothesis in Lemma \ref{lineaerequniquecont1} imply that the above is true for some $t\neq 0$, it follows that $ e^{i t \sign(\xi)\eta^2}\widehat{f}(0,\eta)\phi(\xi)\in H^{\theta}_{\xi}(\mathbb{R})$ for all $\eta\in \mathbb{Z}$, which is always the case when $|t|\eta^2= 2l'\pi$ for some $l'\in \mathbb{Z}^{+}\cup\{0\}$. Otherwise, we have the following claim

\begin{claim}\label{ClaimUnique1} 
\begin{itemize}
    \item[(i)] Let $\frac{1}{2}<\theta<1$. If $|t|\eta^2\neq 2l'\pi$ for all $l'\in\mathbb{Z}$ with $l'\geq 0$, then 
$\mathcal{D}_{\xi}^{\theta}\big(e^{i t \sign(\xi)\eta^2}\phi(\xi)\big)\notin L^2(\mathbb{R})$.
\item[(ii)]  $\mathcal{D}_{\xi}^{\frac{1}{2}}\big(\sign(\xi)\phi(\xi)\big)\notin L^2(\mathbb{R})$.
\end{itemize}

\end{claim}

\begin{proof}[Proof of Claim \ref{ClaimUnique1}] $(ii)$ is a consequence of \cite[Proposition 3]{FonsecaPonce2011}. Let us deal with $(i)$. Let $\delta>0$ be such that $\phi(\xi)=1 $, whenever $|x|<\delta$. We consider the set $\mathcal{A}_{\delta}=\{x>0:x<\delta\}$. Let $x\in \mathcal{A}_{\delta}$, using that $\phi(x)=1$, and the change of variables $w=x-\xi$, we get
\begin{equation*}
\begin{aligned}
\mathcal{D}_{\xi}^{\theta}\big(e^{i t \sign(\xi)\eta^2}\phi(\xi)\big)^2(x)\gtrsim & \int_{|\xi|\leq \delta}\frac{|e^{it \sign(x)\eta^2}\phi(x)-e^{it \sign(\xi)\eta^2}\phi(\xi)|^2}{|x-\xi|^{1+2\theta}}\, d \xi \\
\sim & \int_{-\delta \leq \xi\leq 0 }\frac{|e^{it \eta^2}-e^{-it\eta^2}|^2}{|x-\xi|^{1+2\theta}}\, d\xi\\
\sim & |\sin(t \eta^2)|^2\int_{x \leq w\leq \delta+x }\frac{1}{w^{1+2\theta}}\, d w\\
\gtrsim & \frac{|\sin(t \eta^2)|^2}{x^{2\theta}},
\end{aligned}
\end{equation*}
where we have used that $\frac{1}{x^{2\theta}}-\frac{1}{(\delta+x)^{2\theta}}\geq \frac{1}{3 x^{2\theta}}$, $x\in \mathcal{A}_{\delta}$. Since $\theta>\frac{1}{2}$, it follows that $\frac{|\sin(t \eta^2)|}{x^{\theta}}\notin L^2_{loc}(\mathbb{R})$, and thus  $\mathcal{D}_{\xi}^{\theta}\big(e^{i t \sign(\xi)\eta^2}\phi(\xi)\big)\notin L^2(\mathbb{R})$. 
\end{proof}
Hence, \eqref{equi1} (see also \eqref{equi2}) and Claim \ref{ClaimUnique1} show
\begin{equation*}
\begin{aligned}
e^{i t \sign(\xi)\eta^2}\widehat{f}(0,\eta)\phi(\xi)\in H^{\theta}_{\xi}(\mathbb{R}) \,\, \text{ if and only if }\,\,  \widehat{f}(0,\eta)=0,
\end{aligned}    
\end{equation*}
for all $\eta\in \mathbb{Z}$ such that $|t|\eta^2\neq 2\pi l'$, for all $l'\in \mathbb{Z}$ with $l'\geq 0$, that is, condition \eqref{hiphoteunique1} must be true. This comment completes the deduction of Theorem \ref{lineaerequniquecont1} $(i)$ whenever $\frac{1}{2}<\theta<1$. Now, when $1\leq \theta<\frac{3}{2}$, it follows that $U(t)f\in L^2(|x|^{2\theta}\, dx dy)\subset L^2(|x|^{2\widetilde{\theta}}\, dx dy)$, for all $\frac{1}{2}<\widetilde{\theta}<1$. Then, the argument above works with $\widetilde{\theta}$; thus, the same conclusion \eqref{lineaerequniquecont1} holds.
\\ \\
\underline{\bf (b) Assume $\frac{3}{2}<\theta<\frac{5}{2}$}. Since $U(t)f\in L^2(|x|^{2\widetilde{\theta}}\,dxdy)$ for all $\frac{1}{2}<\widetilde{\theta}<1$, we have from the previous case that $\widehat{f}(0,\eta)=0$, for all $\eta$ such that $|t|\eta^2\neq 2\pi l'$, for each $l'\in \mathbb{Z}$ with $l'\geq 0$. Thus, since $xU(t)f\in L^2(\mathbb{R}\times \mathbb{T})$, the identity \eqref{distribderiv1} reduces to
\begin{equation}\label{uniqueq5}
   \partial_{\xi}(e^{it\omega}\widehat{f}\,)=e^{it\omega}(2it|\xi|)\widehat{f}+e^{it\omega}\partial_{\xi}\widehat{f}. 
\end{equation}
Consequently, we have
\begin{equation*}
U(t)f\in L^2(|x|^{2\theta}\,dxdy)\,\, \text{ if and only if}\, \, \partial_{\xi}(e^{it}\widehat{f}\,)\in H^{\theta-1,0}(\mathbb{R}\times \mathbb{Z}). 
\end{equation*}
Using similar ideas as those in the proof of Theorem \ref{LinearEst} inequality \eqref{linearestimeq1.1}, we have $e^{it\omega}(2it|\xi|)\widehat{f}\in H^{\theta-1,0}(\mathbb{R}\times \mathbb{Z})$. Hence, from identity \eqref{uniqueq5}, we see that
\begin{equation*}
\begin{aligned}
U(t)f\in L^2(|x|^{2\theta}\,dxdy)\,\, &\text{ if and only if}\, \; e^{it\omega }\partial_{\xi} \widehat{f}\in H^{\theta-1,0}(\mathbb{R}\times \mathbb{Z})\\
&\text{ if and only if}\, \; U(t)(xf)\in L^2(|x|^{2(\theta-1)}\, dx dy).  
\end{aligned}
\end{equation*}

Since $\frac{1}{2}<\theta-1<\frac{3}{2}$,  we can apply the previous case (a) to $xf$ to get $\partial_{\xi}\widehat{f}(0,\eta)=0$ for all $\eta\in \mathbb{Z}$ such that $|t|\eta^2\neq 2\pi l'$ for each $l'\in \mathbb{Z}$ with $l'\geq 0$.  This completes the deduction of the present case.
\\ \\
\underline{\bf (c) Assume $\frac{5}{2}<\theta<\frac{7}{2}$}. Since $U(t)f\in L^2(|x|^{\widetilde{\theta}}\, dx dy)$, for all $0<\widetilde{\theta}<\frac{5}{2}$, previous cases show that $\partial_{\xi}^k\widehat{f}(0,\eta)=0$ for all $\eta\in \mathbb{Z}$ such that $|t|\eta^2\neq 2\pi l'$ for each $l'\in \mathbb{Z}$, $l'\geq 0$, with $k=0,1$ (thus, \eqref{hiphoteunique3} holds when $l=k-1=1$ in this case). Let us consider first the case $\frac{5}{2}<\theta<3$. Using identity \eqref{derivative}, we have that there exist some constants $c_1,c_2$ such that
\begin{equation}\label{uniqueq5.1}
 \begin{aligned} \partial_{\xi}^2(e^{it\omega}\widehat{f}\,)=e^{it\omega}(2it\sign(\xi))\widehat{f}+c_1e^{it\omega}(2it|\xi|)\partial_{\xi}\widehat{f}+c_2e^{it\omega}(2it|\xi|)(2it|\xi|)\widehat{f}+e^{it\omega}\partial_{\xi}^2\widehat{f}.    
 \end{aligned}   
\end{equation}
Similar ideas in \eqref{linearestimeq1.1} show $e^{it\omega}(2it|\xi|)\partial_{\xi}\widehat{f}\in H^{\theta-2,0}(\mathbb{R}\times \mathbb{Z})$. In addition,  arguing as in \eqref{linearestimeq4.1}, we have $e^{it\omega}(2it|\xi|)(2it|\xi|)\widehat{f}\in H^{\theta-2,0}(\mathbb{R}\times \mathbb{Z})$.  On the other hand, since $\frac{1}{2}<\theta-2<1$, Sobolev embedding implies that for all $\eta\in \mathbb{Z}$, the function $\xi\rightarrow \partial_{\xi}^2(e^{it\omega}\widehat{f}\,)(\xi,\eta)$ is continuous. Thus, for all $\eta\in \mathbb{Z}$ such that $|t|\eta^2\neq 2\pi l'$, for each $l'\in \mathbb{Z}$ with $l'\geq 0$, it follows from \eqref{uniqueq5.1} that
\begin{equation}\label{uniqueq6}
 \begin{aligned}
  \lim_{\xi\to 0^{+}}\partial_{\xi}^2(e^{it\omega}\widehat{f}\,)(\xi,\eta)=e^{it\eta^2}\partial_{\xi}^2\widehat{f}(0,\eta)=\lim_{\xi\to 0^{-}}\partial_{\xi}^2(e^{it\omega}\widehat{f}\,)(\xi,\eta)=e^{-it\eta^2}\partial_{\xi}^2\widehat{f}(0,\eta),
 \end{aligned}   
\end{equation}
where we have also used the fact that $f\in L^2(|x|^{\theta}\, dx dy)$ and Sobolev embedding imply $\xi \rightarrow \partial_{\xi}^2\widehat{f}(\xi,\eta)$ is continuous for each $\eta\in \mathbb{Z}$, together with $\sign(\xi)\widehat{f}(0,\eta)=0$ for all $\eta\in \mathbb{Z}$ such that $|t|\eta^2\neq 2\pi l'$ for each $l'\in \mathbb{Z}$ with $l'\geq 0$. Consequently, \eqref{uniqueq6} implies $\partial_{\xi}^2\widehat{f}(0,\eta)$ for all $\eta\in \mathbb{Z}$ such that $|t|\eta^2\neq 2\pi l'$ for each $l'\in \mathbb{Z}$, $l'\geq 0$ (this is \eqref{hiphoteunique3} when $k=2$). Combining this previous fact, and using that $x^2f\in L^2(|x|^{2(\theta-2)}\,dx dy)$, $x^2f\in H^{\theta-2,0}(\mathbb{R}\times \mathbb{T})$ (this is a consequence of interpolation Proposition \ref{interpo}), we can use the argument in the proof of Theorem \ref{LinearEst} $(iii)$ to conclude $U(t)(x^2f)\in L^2(|x|^{2(\theta-2)}\, dxdy)$, that is to say, $e^{it\omega}(\partial_{\xi}^2\widehat{f}\,)\in H^{\theta-2,0}(\mathbb{R}\times\mathbb{Z})$.  Going back to \eqref{uniqueq5.1}, we have
\begin{equation*}
 \begin{aligned} \partial_{\xi}^2(e^{it\omega}\widehat{f}\,)\in H^{\theta-2,0}(\mathbb{R}\times \mathbb{Z}) \, \, &\text{ if and only if }\,\, e^{it\omega}(2it\sign(\xi))\widehat{f}\in H^{\theta-2,0}(\mathbb{R}\times \mathbb{Z}).
 \end{aligned}   
\end{equation*}
Since $\frac{1}{2}<\theta-2<1$, the statement above and continuous Sobolev embedding yields that for each $\eta\in \mathbb{Z}$, the function $\xi \rightarrow e^{it\omega}(2it\sign(\xi))\widehat{f}(\xi,\eta)$ is continuous and it must follow
\begin{equation*}
 \begin{aligned}
  \lim_{\xi\to 0^{+}}  e^{it\omega}(2it\sign(\xi))\widehat{f}(\xi,\eta)=2it e^{it\eta^2}\widehat{f}(0,\eta)=\lim_{\xi\to 0^{-}} e^{it\omega}(2it\sign(\xi))\widehat{f}(\xi,\eta)=-2it e^{-it\eta^2}\widehat{f}(0,\eta),
 \end{aligned}   
\end{equation*}
which implies 
\begin{equation*}
\begin{aligned}
4it \cos(t\eta^2)\widehat{f}(0,\eta)=0.    
\end{aligned}    
\end{equation*}
If $|t|\eta^2=2l'\pi$ for some $l'\in \mathbb{Z}^{+}\cup\{0\}$, the identity above forces us to have $\widehat{f}(0,\eta)=0$, and since we have such cancellation when $\eta$ satisfies $|t|\eta^2\neq 2l'\pi$, for all $l'\in \mathbb{Z}$, $l'\geq 0$, we conclude that $\widehat{f}(0,\eta)=0$ for all $\eta\in \mathbb{Z}$ (this is \eqref{hiphoteunique2} for $k-2=0$). This completes the considerations for the case $\frac{5}{2}\leq \theta <3$. Finally, when $3\leq \theta <\frac{7}{2}$, using that $L^2(|x|^{2\theta}\, dx dy)\subset L^2(|x|^{2\widetilde{\theta}}\, dx dy)$ for all $\frac{5}{2}<\widetilde{\theta}<3$, we have that the previous argument applied to $\widetilde{\theta}$ yields the same restrictions \eqref{hiphoteunique2} and \eqref{hiphoteunique3} on $f$ with $k=2$.  
\\ \\
\underline{\bf (d) Assume $k+\frac{1}{2}<\theta<k+\frac{3}{2}$ with $k\geq 3$}. We will argue by induction on the integer $k\geq 2$. We remark that the deduction of the present case resembles the case $k=2$ above, but for completeness, we will carry on its deduction. We will shows that for each integer $k\geq 2$, if $k+\frac{1}{2}<\theta<k+\frac{3}{2}$, and $f\in H^{\theta,0}(\mathbb{R}\times \mathbb{T})\cap L^2(|x|^{2\theta}\, dx dy)$ is such that $U(t)f\in L^2(|x|^{2\theta}\, dx dy)$ for some $t\neq 0$, then $f$ satisfies \eqref{hiphoteunique2} and \eqref{hiphoteunique3}. Firstly, as a consequence of the previous case, we have that the above holds true when $k=2$. Thus, we proceed to deduce the case $k\geq 3$, and we assume by the inductive hypothesis that the desired result holds for $k-1$. Since $f\in H^{\theta,0}(\mathbb{R}\times \mathbb{T})\cap L^2(|x|^{2\theta}\, dx dy)\subset H^{\widetilde{\theta},0}(\mathbb{R}\times \mathbb{T})\cap L^2(|x|^{2\widetilde{\theta}}\, dx dy)$, and $U(t)f\in L^2(|x|^{2\theta}\, dx dy)\subset L^2(|x|^{2\widetilde{\theta}}\, dx dy)$ for all $(k-1)+\frac{1}{2}<\widetilde{\theta}<(k-1)+\frac{3}{2}$, we have by inductive hypothesis that $f$ satisfies \eqref{hiphoteunique2} and \eqref{hiphoteunique3} with $k-1$, i.e.,  in the frequency domain these conditions transfer to
\begin{equation}\label{uniqueq7}
 \begin{aligned}
  \partial_{\xi}^{l}\widehat{f}(0,\eta)=0,    
 \end{aligned}   
\end{equation}
for all $l=0,\dots, k-1$, all $\eta\in \mathbb{Z}$ such that $|t|\eta^2\neq 2\pi l'$ for each $l'\in \mathbb{Z}$, $l'\geq 0$, and in addition \eqref{uniqueq7} holds for all $\eta \in \mathbb{Z}$, if $l=0,\dots, k-3$. We will show that $U(t)f\in L^2(|x|^{2\theta}\, dx dy)$ for some $t\neq 0$ forces us to have that $f$ also satisfies \eqref{uniqueq7} with $l=k$ and all $\eta\in \mathbb{Z}$ such that $|t|\eta^2\neq 2\pi l'$ for each $l'\in \mathbb{Z}$ with $l'\geq 0$, and \eqref{uniqueq7} is true for all $\eta \in \mathbb{Z}$ when $l=k-2$. Since $L^2(|x|^{2\theta}\, dx dy)\subset L^2(|x|^{2\widetilde{\theta}}\, dx dy)$ for some $k+\frac{1}{2}<\widetilde{\theta}<k+1$, it is enough to only consider the case $\frac{1}{2}<\theta-k<1$.

We first notice that \eqref{uniqueq7} and the identity \eqref{derivative} show that 
\begin{equation}\label{uniqueq9}
\begin{aligned}
\partial^k_{\xi}\big(e^{it\omega}\widehat{f}\,\big)=&e^{it \omega }\sum_{m=1}^k \sum_{\substack{l_1+\dots+l_{m}+l_{m+1}=k-m\\ 0\leq l_j\leq 1\\ j=1,\dots,m}}c_{l_1,\dots,l_{m+1}}\partial_{\xi}^{l_1}(2it|\xi|)\dots \partial_{\xi}^{l_m}(2it|\xi|)\partial^{l_{m+1}}_{\xi}\widehat{f}\\
&+ e^{it \omega} \partial^{k}_{\xi}\widehat{f},
\end{aligned}    
\end{equation}
where, when $m=1$, $l_1=1$, and $l_2=k-2$, $c_{l_1,l_2}=1$. Now, to simplify notation, we write
\begin{equation*}
 \begin{aligned}
\mathcal{A}_m=\{(l_1,\dots, l_m,l_{m+1})\in \big(\mathbb{Z}^{+}\cup\{0\}\big)^{m+1}\,:\, &0\leq l_j\leq 1, \,\text{ for all }\, j=1,\dots,m, \\
&l_1+\dots+l_m+l_{m+1}=k-m\\
& \text{there exists some $0\leq j' \leq m$ such that $l_{j'}=0$}\}.     
 \end{aligned}   
\end{equation*}
Then, splitting the summation in \eqref{uniqueq9} according to the cases where $l_1=\dots=l_m=1$, we have
\begin{equation}\label{uniqueq8}
\begin{aligned}
\partial^k_{\xi}\big(e^{it\omega}\widehat{f}\,\big)=&e^{it \omega }(2it\sign(\xi))\partial_{\xi}^{k-2}\widehat{f}+e^{it \omega }\sum_{m=2}^{\lfloor \frac{k}{2}\rfloor}c_m (2it\sign(\xi))^m\partial_{\xi}^{k-2m}\widehat{f}\\
&+e^{it \omega }\sum_{m=1}^k \sum_{(l_1,\dots,l_m)\in \mathcal{A}_m} c_{l_1,\dots,l_{m+1}}\partial_{\xi}^{l_1}(2it|\xi|)\dots \partial_{\xi}^{l_m}(2it|\xi|)\partial^{l_{m+1}}_{\xi}\widehat{f}\\
&+ e^{it \omega} \partial^{k}_{\xi}\widehat{f}.
\end{aligned}    
\end{equation}
Notice that the assumption $|x|^{\theta}U(t)f\in L^2(\mathbb{R}\times \mathbb{T})$ is equivalent to $\partial^k_{\xi}\big(e^{it\omega}\widehat{f}\,\big)\in H^{\theta-k,0}(\mathbb{R}\times \mathbb{Z})$. Next, we study the identity \eqref{uniqueq8} in the space $H^{\theta-k,0}(\mathbb{R}\times \mathbb{Z})$. The arguments in \eqref{decaycond} and \eqref{linearestimeq4.1} show that
\begin{equation*}
\begin{aligned}
e^{it \omega }\sum_{m=1}^k \sum_{(l_1,\dots,l_m)\in \mathcal{A}_m}c_{l_1,\dots,l_{m+1}}\partial_{\xi}^{l_1}(2it|\xi|)\dots \partial_{\xi}^{l_m}(2it|\xi|)\partial^{l_{m+1}}_{\xi}\widehat{f}\in H^{\theta-k,0}(\mathbb{R}\times \mathbb{Z}).    
\end{aligned}    
\end{equation*}
Now, we show that the second term on the right-hand side of \eqref{uniqueq8} is in $H^{\theta-k,0}(\mathbb{R}\times \mathbb{Z})$. For all $m=2,\dots,\lfloor\frac{k}{2}\rfloor$, we have $x^{k-2m}f\in L^2(|x|^{2(\theta-k)}\, dx dy)$, and by \eqref{uniqueq7},  $\widehat{x^{k-2m}f}(0,\eta)=0$ for all $\eta\in \mathbb{Z}$. Here we used that $k-2m\leq k-4$, and so by the inductive hypothesis, \eqref{uniqueq7} holds for $l=k-2m$ and all $\eta \in \mathbb{Z}$. Thus, we have checked that $x^{k-2m}f$ satisfies the hypothesis of Remark \ref{remarkHilbertL2}, and we get from it
\begin{equation*}
\begin{aligned}
\|D^{\theta-k}_{\xi}\Big(\sum_{m=2}^{\lfloor \frac{k}{2}\rfloor}c_m(2it\sign(\xi))^m\partial_{\xi}^{k-2m}\widehat{f}\,\Big)\|_{L^2}\lesssim &\sum_{m=2}^{\lfloor \frac{k}{2}\rfloor}  \langle t \rangle^{m}\Big(\|\langle x \rangle^{\theta-k}\mathcal{H}_x(x^{k-2m}f)\|_{L^2}+\|\langle x \rangle^{\theta-k}(x^{k-2m}f)\|_{L^2}\Big)\\
\lesssim & \langle t \rangle^{\theta}\|\langle x \rangle^{\theta }f\|_{L^2}.
\end{aligned}    
\end{equation*}
In addition, from the previous estimate and the fact that $\mathcal{F}(\mathcal{H}_x x^{k-2m}f))(0,\eta)=0$ for all $\eta\in \mathbb{Z}$ (this is also the assumption \eqref{uniqueq7} with $l=k-2m\leq k-4$), it also follows  that $\mathcal{H}(x^{k-2m}f)$ satisfies the hypothesis of Theorem \ref{LinearEst}, from which it follows
\begin{equation*}
\begin{aligned}
\|D^{\theta-k}_{\xi}&\Big(e^{it\omega} \sum_{m=2}^{\lfloor \frac{k}{2}\rfloor}c_m(2it\sign(\xi))^m\partial_{\xi}^{k-2m}\widehat{f}\,\Big)\|_{L^2}\\
\lesssim & \sum_{m=2}^{\lfloor \frac{k}{2}\rfloor}\langle t \rangle^{m}\big(\|\langle x \rangle^{\theta-k}U(t)\mathcal{H}_x(x^{k-2m}f)\|_{L^2}+\|\langle x \rangle^{\theta-k}U(t)(x^{k-2m}f) \|_{L^2}\big)\\
\lesssim & \sum_{m=2}^{\lfloor \frac{k}{2}\rfloor}  \langle t\rangle^{\theta-k+m}\Big( \|J^{\theta-k}_{x}(x^{k-2m}f)\|_{L^2}+\|\langle x \rangle^{\theta-k}\mathcal{H}_x(x^{k-2m}f)\|_{L^2}+\|\langle x \rangle^{\theta-k}x^{k-2m}f\|_{L^2}\Big)\\
\lesssim& \langle t \rangle^{\theta}\big(\|J^{\theta}_{x}f\|_{L^2}+ \|\langle x \rangle^{\theta }f\|_{L^2}\big),
\end{aligned}    
\end{equation*}
where we have also used that $\|J^{\theta-k}_x (x^{k-2m}f)\|_{L^2}\lesssim \|J^{\theta-k}_x(\langle x\rangle^{k-2m}f)\|_{L^2}$ and interpolation inequality \eqref{interpineq1} (for a similar argument, see \eqref{linearestimeq4}). Consequently, we have deduced
\begin{equation*}
 \begin{aligned}
e^{it \omega }\sum_{m=2}^{\lfloor \frac{k}{2}\rfloor} c_m(2it\sign(\xi))^m \partial_{\xi}^{k-2m}\widehat{f}\in H^{\theta-k,0}(\mathbb{R}\times \mathbb{Z}).    
 \end{aligned}   
\end{equation*}
Since $\frac{1}{2}<\theta-k<1$, we have from the continuity provided by Sobolev embedding $H^{\theta-2}(\mathbb{R})\hookrightarrow L^{\infty}(\mathbb{R})$, hypothesis \eqref{uniqueq7}, and identity \eqref{uniqueq8} that
\begin{equation}
 \begin{aligned}
  \lim_{\xi\to 0^{+}}\partial_{\xi}^k(e^{it\omega}\widehat{f}\,)(\xi,\eta)=e^{it\eta^2}\partial_{\xi}^k\widehat{f}(0,\eta)=\lim_{\xi\to 0^{-}}\partial_{\xi}^k(e^{it\omega}\widehat{f}\,)(\xi,\eta)=e^{-it\eta^2}\partial_{\xi}^k\widehat{f}(0,\eta),
 \end{aligned}   
\end{equation}
for all $\eta\in \mathbb{Z}$ such that $|t|\eta^2\neq 2\pi l'$ for each $l'\in \mathbb{Z}$ with $l'\geq 0$. It follows,
\begin{equation*}
  2i\sin(t\eta^2)\partial_{\xi}^{k}\widehat{f}(0,\eta)=0,
\end{equation*}
for all $\eta\in \mathbb{Z}$ such that $|t|\eta^2\neq 2\pi l'$ for each $l'\in \mathbb{Z}$ with $l'\geq 0$. Then, $\partial_{\xi}^k\widehat{f}(0,\eta)=0$ for such $\eta$ as above. This fact shows that we can argue as in the proof of Theorem \ref{LinearEst} part $(iii)$ to get $e^{it \omega}\partial_{\xi}^k\widehat{f}\in H^{\theta-k,0}(\mathbb{R}\times \mathbb{T})$, that is to say, $U(t)(x^k f)\in L^{2}(|x|^{2\theta}\, dx dy)$. Moreover, we have that \eqref{uniqueq7} holds true for $l=k$ and all $\eta\in \mathbb{Z}$ such that $|t|\eta^2\neq 2\pi l'$ for each $l'\in \mathbb{Z}$, $l'\geq 0$, which is the first result we require to prove the inductive step. Collecting the results above, we get
\begin{equation*}
 \begin{aligned} \partial_{\xi}^k(e^{it\omega}\widehat{f}\,)\in H^{\theta-k,0}(\mathbb{R}\times \mathbb{Z}) \, \, &\text{ if and only if }\,\, e^{it\omega}(2it\sign(\xi))\partial_{\xi}^{k-2}\widehat{f}\in H^{\theta-2,0}(\mathbb{R}\times \mathbb{Z}).
 \end{aligned}   
\end{equation*}
Since the above is true for some $t\neq 0$, we have $e^{it\omega}(2it\sign(\xi))\partial_{\xi}^{k-2}\widehat{f}\in H^{\theta-2,0}(\mathbb{R}\times \mathbb{Z})$, but this fact and continuity show
\begin{equation*}
 \begin{aligned}
  \lim_{\xi\to 0^{+}}  e^{it\omega}(2it\sign(\xi))\partial_{\xi}^{k-2}\widehat{f}(\xi,\eta)&=2it e^{it\eta^2}\partial_{\xi}^{k-2}\widehat{f}(0,\eta)=\lim_{\xi\to 0^{-}} e^{it\omega}(2it\sign(\xi))\partial_{\xi}^{k-2}\widehat{f}(\xi,\eta)\\
  &=-2it e^{-it\eta^2}\partial_{\xi}^{k-2}\widehat{f}(0,\eta),
 \end{aligned}   
\end{equation*}
which is equivalent to
\begin{equation*}
\begin{aligned}
4it \cos(t\eta^2)\partial_{\xi}^{k-2}\widehat{f}(0,\eta)=0.    
\end{aligned}    
\end{equation*}
Hence, $\partial_{\xi}^{k-2}\widehat{f}(0,\eta)=0$ for all $\eta \in \mathbb{Z}$ such that $t|\eta|^2=2\pi l'$ for some $l'\in \mathbb{Z}$. Since we deduced above that the same cancellation holds for all $\eta \in \mathbb{Z}$ such that $t|\eta|^2\neq 2\pi l'$ for all $l'\in \mathbb{Z}$, it follows $\partial_{\xi}^{k-2}\widehat{f}(0,\eta)=0$ for all $\eta\in \mathbb{Z}$. This completes the deduction of \eqref{uniqueq7}, when $l=k-2$ and $\eta\in \mathbb{Z}$, and in consequence, the inductive step $k\geq 3$. 
\end{proof}

%%%%%%%%%%%%%%%%%%%%%%%%%%%%%%%%%%%%%%%%%%%%%%%%%%%%%%%%%%%%%%%%%%%%%%%%%%%%%%%%%%%%%%%%%%%%%%%%%%%%%%%%%%%%%%%%%%%%%%%%%%%%%%%%%%%%%%%%%%%%%%%%%%%%%%%%%%%%%%%%%%%%%%%%%%%%%%%%%%%%%%%%%%%%%%%%

\subsubsection{Proof of Corollary \ref{lineaerequniquecont2}}

We begin with the deduction of the following lemma, which relates $U(t)\mathcal{H}_x f\in  L^2(|x|\, dxdy)$ and $\mathcal{H}_x f\in  L^2(|x|\, dxdy)$. 

\begin{lemma}\label{uniquecontcritcase}
Let $f\in H^{\frac{1}{2},0}(\mathbb{R}\times \mathbb{T})\cap L^2(|x|\, dxdy)$. Assume that there exists a time $t\neq 0$ such that $U(t)f\in L^2(|x|\, dx dy)$. Then $U(t)\mathcal{H}_x f\in L^2(|x|\, dxdy)$ if and only if $\mathcal{H}_x f\in L^2(|x|\, dxdy)$.
\end{lemma}

\begin{proof}
Recalling $\omega$ as in \eqref{dispersivesymbol}, we take the Fourier transform of $U(t)f$ to write
\begin{equation}\label{uniqeq1}
 \begin{aligned}
  e^{it\omega}\widehat{f}=\big(e^{i t\xi|\xi|}-1\big)e^{i t \sign(\xi)\eta^2}\widehat{f}+e^{i t \sign(\xi)\eta^2}\widehat{f}.   
 \end{aligned}   
\end{equation}
Once again, $U(t)f\in L^2(|x|\, dx dy)$ is equivalent to $e^{it \omega}\widehat{f}\in H^{\frac{1}{2},0}(\mathbb{R}\times \mathbb{Z})$. Now, since $f\in H^{\frac{1}{2},0}(\mathbb{R}\times \mathbb{T})\cap L^2(|x|\, dxdy)$, using the equivalence in \eqref{equi1}, together with \eqref{prelimneq}, \eqref{steinderiv1} and \eqref{steinderiv2}, we can use previous arguments to deduce $\big(e^{i t\xi|\xi|}-1\big)e^{i t \sign(\xi)\eta^2}\widehat{f}\in H^{\frac{1}{2},0}(\mathbb{R}\times \mathbb{Z})$ for all $t$. Consequently, it must follow that
\begin{equation*}
 e^{it \omega}\widehat{f}\in H^{\frac{1}{2},0}(\mathbb{R}\times\mathbb{Z})\, \, \text{ if and only if}\,\, e^{i t \sign(\xi)\eta^2}\widehat{f}\in H^{\frac{1}{2},0}(\mathbb{R}\times\mathbb{Z}), 
\end{equation*}
which by Fourier transform and using the group notation $\{U_1(t)\}$ (see \eqref{Grouponly}) transfers to
\begin{equation}\label{uniqeq2}
 U(t)f\in  L^2(|x|\, dx dy)\, \, \text{ if and only if}\, \, U_1(t)f\in  L^2(|x|\, dx dy).
\end{equation}
Similarly, we have
\begin{equation}\label{uniqeq2.1}
 U(t)\mathcal{H}_x f\in  L^2(|x|\, dx dy)\, \, \text{ if and only if}\, \, U_1(t)\mathcal{H}_x f\in  L^2(|x|\, dx dy).
\end{equation}
Next, using an approximation argument and the fact that $
\langle x \rangle_N\in L^{\infty}(\mathbb{R})$ (see \eqref{trunweightdef}), we can argue as in the deduction of identity \eqref{Afterintegra} in the proof of Proposition \ref{decaysimplyeq} to get
\begin{equation}\label{uniqeq3}
\begin{aligned}
   \|\langle x\rangle_N^{\frac{1}{2}}U(t)f\|_{L^2}^2+\|\langle x\rangle_N^{\frac{1}{2}}U(t)\mathcal{H}_x f\|_{L^2}^2=&\|\langle x\rangle_N^{\frac{1}{2}}f\|_{L^2}^2+\|\langle x\rangle_N^{\frac{1}{2}}\mathcal{H}_x f\|_{L^2}^2.
\end{aligned}
\end{equation}
Since $\langle x\rangle_N^{\frac{1}{2}}\lesssim \langle x\rangle^{\frac{1}{2}}$ (with implicit constant independent of $N$), \eqref{uniqeq2}, the fact that $f\in H^{\frac{1}{2},0}(\mathbb{R}\times \mathbb{T})\cap L^2(|x|\, dxdy)$, and taking $N\to \infty$ in  \eqref{uniqeq3} show
\begin{equation*}
\begin{aligned}
 U_1(t)\mathcal{H}_{x} f\in  L^2(|x|\, dx dy)\, \, \text{ if and only if}\, \, \mathcal{H}_{x}f\in  L^2(|x|\, dx dy),
\end{aligned}
\end{equation*}
$t\neq 0$. The above conclusion and \eqref{uniqeq2.1} yield the desired result.
\end{proof}

\begin{proof}[Proof of Corollary \ref{lineaerequniquecont2}]
Let $f\in H^{k+\frac{1}{2},0}(\mathbb{R}\times \mathbb{T})\cap L^2(|x|^{k+\frac{1}{2}}\, dx dy)$. We notice that $k=0$ is a consequence of Lemma \ref{uniquecontcritcase}. Now, if $k\geq 1$, given that  $f\in L^2(|x|^{2\widetilde{\theta}}\, dx dy)$ for all $(k-1)+\frac{1}{2}<\widetilde{\theta}<(k-1)+\frac{3}{2}$, by Theorem \ref{lineaerequniquecont1}, it follows that $f$ satisfies \eqref{uniqueq7} for all $l=0,\dots, k-1$, and all $\eta\in \mathbb{Z}$ such that $|t|\eta^2\neq 2\pi l'$ for each $l'\in \mathbb{Z}$. Thus, if $k=1$, this previous fact and identity \eqref{derivative} with $k=1$, allow us to use  Lemma \ref{uniquecontcritcase} to get the desired result \eqref{xkuniquecont0}. Consequently, we will consider the case $k\geq2$.

Note that if $k\geq 3$ and $(k-1)+\frac{1}{2}<\widetilde{\theta}<(k-1)+\frac{3}{2}$, by Theorem \ref{lineaerequniquecont1}, we have that \eqref{uniqueq7} holds for all $l=0,\dots,k-3$  and all $\eta\in \mathbb{Z}$. Such results imply that $\partial_{\xi}^k(e^{it\omega}\widehat{f}\,)$ satisfies \eqref{uniqueq9} (or \eqref{uniqueq8}) for all $k\geq 2$. We remark that in the deduction of $(ii)$ in Theorem \ref{lineaerequniquecont1},  we used the fact that Sobolev embedding assures that any distribution in $H^{\frac{1}{2}^{+}}(\mathbb{R})$ can be represented as a continuous function (e.g., see the ideas in \eqref{uniqueq6}). However, in the present case we cannot apply the same ideas as $\partial_{\xi}^k(e^{it\omega}\widehat{f}\,)\in H^{\frac{1}{2},0}(\mathbb{R}\times \mathbb{T})$ and distributions in $H^{\frac{1}{2}}(\mathbb{R})$ are not necessarily represented by a continuous function.

Now, given the same decomposition in \eqref{uniqueq8} , similar arguments in the proof of Theorem \ref{lineaerequniquecont1} $(ii)$ case (d) show 
\begin{equation}
\begin{aligned}
\partial^k_{\xi}\big(e^{it\omega}\widehat{f}\,\big)\in H^{\frac{1}{2},0}(\mathbb{R}\times \mathbb{Z})\qquad \text{ if and only if} \qquad e^{it \omega }(2it\sign(\xi))\partial_{\xi}^{k-2}\widehat{f}\,+ e^{it \omega} \partial^{k}_{\xi}\widehat{f}\,\in H^{\frac{1}{2},0}(\mathbb{R}\times \mathbb{Z}).
\end{aligned}    
\end{equation}
Now, the above statement implies
\begin{equation}
\begin{aligned}
e^{it \omega }(2it\sign(\xi)\partial_{\xi}^{k-2}\widehat{f}\,)\in H^{\frac{1}{2},0}(\mathbb{R}\times \mathbb{Z})\qquad \text{ if and only if} \qquad  e^{it \omega} \partial^{k}_{\xi}\widehat{f}\in H^{\frac{1}{2},0}(\mathbb{R}\times \mathbb{Z}).
\end{aligned}    
\end{equation}
Next, we write
\begin{equation}\label{unique9}
 \begin{aligned}
  e^{it\omega}(\sign(\xi)\partial_{\xi}^{k-2}\widehat{f}(\xi,\eta))=&\big(e^{i t\xi|\xi|}-1\big)e^{i t \sign(\xi)\eta^2}\sign(\xi)\partial_{\xi}^{k-2}\widehat{f}(\xi,\eta)\\
  &+e^{i t \sign(\xi)\eta^2}\sign(\xi)\big(\partial_{\xi}^{k-2}\widehat{f}(\xi,\eta)-\partial_{\xi}^{k-2}\widehat{f}(0,\eta)\phi(\xi)\big)\\
  &+e^{i t \sign(\xi)\eta^2}\sign(\xi)\partial_{\xi}^{k-2}\widehat{f}(0,\eta)\phi(\xi),      
\end{aligned}    
\end{equation}
where $\phi\in C^{\infty}_c(\mathbb{R})$ is equal to $1$ in a neighborhood of the origin. Using similar ideas in the justifications for the estimate in \eqref{uniqueq4}, together with \eqref{decaywight}, it follows that the first two terms on the right-hand side of \eqref{unique9} are in $H^{\frac{1}{2},0}(\mathbb{R}\times \mathbb{Z})$. Thus, we get
\begin{equation*}
\begin{aligned}
e^{it \omega }(2it\sign(\xi)\partial_{\xi}^{k-2}\widehat{f}\,)\in H^{\frac{1}{2},0}(\mathbb{R}\times \mathbb{Z}) \qquad & \text{ if and only if} \\
&  e^{i t \sign(\xi)\eta^2}\sign(\xi)\partial_{\xi}^{k-2}\widehat{f}(0,\eta)\phi(\xi)\in H^{\frac{1}{2},0}(\mathbb{R}\times \mathbb{Z}).
\end{aligned}    
\end{equation*}
Since we know that $\partial_{\xi}^{k-2}\widehat{f}(\xi,\eta)=0$ for all $\eta\in \mathbb{Z}$ such that $|t|\eta^2\neq 2\pi l'$ for each $l'\in \mathbb{Z}$, and by Claim \ref{ClaimUnique1} (ii) we know that $\sign(\xi)\phi(\xi)\notin H^{\frac{1}{2}}(\mathbb{R})$, it must follow
\begin{equation*}
\begin{aligned}
e^{it \omega }(2it\sign(\xi)\partial_{\xi}^{k-2}\widehat{f}\,)\in H^{\frac{1}{2},0}(\mathbb{R}\times \mathbb{Z}) \qquad & \text{ if and only if} \qquad \partial_{\xi}^{k-2}\widehat{f}(0,\eta)=0,  \text{ for each } \eta \in \mathbb{Z}.
\end{aligned}    
\end{equation*}
Gathering the previous results, we have
\begin{equation}
\begin{aligned}
e^{it \omega} \partial^{k}_{\xi}\widehat{f}\in H^{\frac{1}{2},0}(\mathbb{R}\times \mathbb{Z}) \qquad \text{ if and only if} \qquad \partial_{\xi}^{k-2}\widehat{f}(0,\eta)=0,  \text{ for each } \eta \in \mathbb{Z}.
\end{aligned}    
\end{equation}
The previous conclusion is \eqref{xkuniquecont}. On the other hand, assuming either of the conditions above,  we can apply Lemma \ref{uniquecontcritcase} to obtain that $U(t)\mathcal{H}_x(x^k f)\in L^2(|x|\, dx dy)$ if and only if $\mathcal{H}_x(x^k f)\in L^2(|x|\, dx dy)$, that is, \eqref{xkuniquecont1} holds true.
\end{proof}

%%%%%%%%%%%%%%%%%%%%%%%%%%%%%%%%%%%%%%%%%%%%%%%%%%%%%%%%%%%%%%%%%%%%%%%%%%%%%%%%%%%%%%%%%%%%%%%%%%%%%%%%%%%%%%%%%%%%%%%%%%%%%%%%%%%%%%%%%%%%%%%%%%%%%%

\section{Persistence results nonlinear equation}\label{SectionNonlinear}

In this section, we deduce the local existence result stated in Theorem \ref{LWPresultinWS}. We will use standard theory in $H^{s_1,s_2}(\mathbb{R}\times \mathbb{T})$ of the Cauchy problem \eqref{SHeq} to get existence and uniqueness of solutions in the class $C([0,T];H^{s_1,s_2}(\mathbb{R}\times \mathbb{T}))$. Then, we will show that such solutions also persist in the space $L^2(|x|^{2\theta}\, dx dy)$. The strategy is to deduce first a persistence result in weighted spaces for smooth solutions and use it to extend to the general case by energy estimates and approximation.

Standard arguments that do not take into account the dispersion in \eqref{SHeq}, which rely on the embedding $H^{\frac{1}{2}^{+},\frac{1}{2}^{+}}(\mathbb{R}\times \mathbb{T})\hookrightarrow L^{\infty}(\mathbb{R}\times \mathbb{T})$ establish the following local well-posedness result. For similar ideas, we refer to \cite{Iorio1986}. 

\begin{lemma}\label{StandardLWP}
Let $K\geq 1$, $\nu_k \in \mathbb{R}$, $k=1,\dots,K$, not all null, and $s_1>\frac{3}{2}$, $s_2>\frac{1}{2}$. Let $u_0\in H^{s_1,s_2}(\mathbb{R}\times\mathbb{T})$. Then there exist $T>0$ and a unique $u\in C([0,T];H^{s_1,s_2}(\mathbb{R}\times \mathbb{T}))$ solution of the Cauchy problem \eqref{SHeq}. Moreover, the existence time $T=T(\|u_0\|_{H^{s_1,s_2}})$ is a nonincreasing function of $\|u_0\|_{H^{s_1,s
_2}}$ and the flow-map is continuous. 
\end{lemma}

In particular, the result of Lemma \ref{StandardLWP} also establishes the existence of smooth solutions. Indeed, setting $s_1>\frac{3}{2}$, $s_2>\frac{1}{2}$ fixed, when $u_0\in H^{\infty}(\mathbb{R}\times \mathbb{T})$, there exists a time $T=T(\|u_0\|_{H^{s_1,s_2}})>0$ and a unique solution $u\in C([0,T];H^{\infty}(\mathbb{R}\times \mathbb{T}))$ of the Cauchy problem \eqref{SHeq}. Next, we show that smooth solutions of \eqref{SHeq} also persist in weighted spaces.

\begin{lemma}\label{smoothsolutdecay}
Let $K\geq 1$, $\nu_k\in \mathbb{R}$, $k=1,\dots,K$, not all null, $u_0\in H^{\infty}(\mathbb{R}\times \mathbb{T})\cap L^2(|x|^{2\theta}\, dx dy)$. Let $u\in C([0,T];H^{\infty}(\mathbb{R}\times \mathbb{T}))$ be the solution of the Cauchy problem \eqref{SHeq} with initial condition $u_0$.
\begin{itemize}
    \item[(i)] Let $0<\theta\leq \frac{1}{2}$. If $\theta=\frac{1}{2}$, in addition, assume that $\mathcal{H}_xu_{0}\in L^2(|x|\, dx)$. Then \begin{equation}\label{contweights}
 u\in C([0,T];L^2(|x|^{2\theta}\, dx dy)).   
\end{equation} 
\item[(ii)] If, in addition, $\widehat{u}_0(0,\eta)=0$ for all $\eta\in \mathbb{Z}$, then \eqref{contweights} also holds for all $0<\theta<\frac{3}{2}$.

\end{itemize}

\end{lemma}

\begin{proof}
We first deduce $(i)$. Setting $0<\theta\leq \frac{1}{2}$, we multiply the equation in \eqref{SHeq} by $\langle x \rangle_N^{2\theta} u$, and then integrating over $\mathbb{R}\times \mathbb{T}$, we get
\begin{equation*}
\begin{aligned}
\frac{1}{2}\frac{d}{dt}\int (\langle x \rangle_N^{\theta} u)^2\, dx dy -\int \mathcal{H}_x\partial_x^2 u(\langle x \rangle_N^{2\theta} u)\, dx dy&-\int \mathcal{H}_x\partial_y^2u (\langle x \rangle_N^{2\theta} u)\, dx dy\\
&=-\sum_{k=1}^K \nu_k\int  u^k\partial_xu(\langle x \rangle_N^{2\theta} u)\, dx dy.
\end{aligned}    
\end{equation*}
Now, applying $\mathcal{H}_x$ to the equation in \eqref{SHeq}, then multiplying by $\langle x \rangle_N^{2\theta} \mathcal{H}_xu$ and integrating over $\mathbb{R}\times \mathbb{T}$, it follows
\begin{equation*}
\begin{aligned}
\frac{1}{2}\frac{d}{dt}\int (\langle x \rangle_N^{\theta}\mathcal{H}_x u)^2\, dx dy +\int \partial_x^2 u(\langle x \rangle_N^{2\theta} \mathcal{H}_x u)\, dx dy&+\int \partial_y^2u (\langle x \rangle_N^{2\theta} \mathcal{H}_x u)\, dx dy\\
&=-\sum_{k=1}^K\nu_k\int \mathcal{H}_x(u^k\partial_x u )(\langle x \rangle_N^{2\theta} \mathcal{H}_x u)\, dx dy.
\end{aligned}    
\end{equation*}
Combining the previous equation, using integration by parts on the $y$-variable together with the fact that $\langle x \rangle_N^{2\theta}$ is independent of $y$, we deduce
\begin{equation}\label{eqLWP1}
\begin{aligned}
\frac{1}{2}\frac{d}{dt}\Big(\int (\langle x \rangle_N^{\theta} u)^2\, dx dy&+\int (\langle x \rangle_N^{\theta}\mathcal{H}_x u)^2\, dx dy\Big)\\
=&\int \mathcal{H}_x\partial_x^2 u(\langle x \rangle_N^{2\theta} u)\, dx dy-\int \partial_x^2 u(\langle x \rangle_N^{2\theta} \mathcal{H}_x u)\, dx dy\\
&-\sum_{k=1}^K\nu_k\int u^k\partial_x u(\langle x \rangle_N^{2\theta} u)\, dx dy\\
&-\sum_{k=1}^K\nu_k\int \mathcal{H}_x (u^k\partial_x u)(\langle x \rangle_N^{2\theta} \mathcal{H}_x u)\, dx dy.
\end{aligned}    
\end{equation}
Let us estimate the right-hand side of the above identity. We study first the parts involving the dispersion. Integration by parts shows
\begin{equation*}
\begin{aligned}
\int \mathcal{H}_x\partial_x^2 u(\langle x \rangle_N^{2\theta} u)\, dx dy-\int \partial_x^2 u(\langle x \rangle_N^{2\theta} \mathcal{H}_x u)\, dx dy=&\int  \mathcal{H}_x u \partial_x^2(\langle x\rangle_N^{2\theta}) u\, dx dy+2\int \mathcal{H}_x u \partial_x(\langle x\rangle_N^{2\theta}) \partial_x u\, dx dy\\
=:&\,\mathcal{I}+\mathcal{II}.    
\end{aligned}    
\end{equation*}
We proceed with the estimates for $\mathcal{I}$ and $\mathcal{II}$. Given that $0<\theta\leq \frac{1}{2}$, the definition of the weight $\langle x \rangle_N$ yields $|\partial_x^2(\langle x\rangle_N^{2\theta})|\lesssim 1$, where the implicit constant is independent of $\theta$ and $N$. Additionally, since $\mathcal{H}_x$ extends to a bounded operator in $L^2(\mathbb{R}\times \mathbb{T})$, we use H\"older's inequality to deduce \begin{equation*}
\begin{aligned}
 |\mathcal{I}|\leq  \|\partial_x^2(\langle x\rangle_N^{2\theta})\|_{L^{\infty}}\|\mathcal{H}_x u\|_{L^2}\|u\|_{L^2}\lesssim  \| u\|_{L^2}^2.
\end{aligned}
\end{equation*}
Now, since $\mathcal{H}_x$ is a skew-symmetric operator, and $D_x=\mathcal{H}_x\partial_x$, we write
\begin{equation}
 \begin{aligned}
 \frac{1}{2}\mathcal{II}=&-\int u \big([\mathcal{H}_x,\partial_x(\langle x\rangle_N^{2\theta})]\partial_x u\big)\, dx dy-\int u \partial_x(\langle x\rangle_N^{2\theta}) \mathcal{H}_x\partial_x u\, dx dy \\
=&-\int u [\mathcal{H}_x,\partial_x(\langle x\rangle_N^{2\theta})]\partial_x u\, dx dy-\int \big([D_x^{\frac{1}{2}},\partial_x(\langle x\rangle_N^{2\theta})]u\big) D_x^{\frac{1}{2}} u\, dx dy \\
&-\int \partial_x(\langle x\rangle_N^{2\theta}) D_x^{\frac{1}{2}}u D_x^{\frac{1}{2}} u\, dxdy.
 \end{aligned}   
\end{equation}
We use Cauchy-Schwarz inequality and Proposition \ref{CalderonComGU}  to get
\begin{equation*}
\begin{aligned}
\Big|\int u [\mathcal{H}_x,\partial_x(\langle x\rangle_N^{2\theta})]\partial_x u\, dx dy\Big| \leq & \|u\|_{L^2}\|\,\|[\mathcal{H}_x,\partial_x(\langle x\rangle_N^{2\theta})]\partial_x u\|_{L^2_x}\|_{L^2_y}\\
\lesssim & \|\partial_x^2(\langle x\rangle_N^{2\theta})\|_{L^{\infty}}\|u\|_{L^2}^2\\
\lesssim & \|u\|_{L^2}^{2}.
\end{aligned}    
\end{equation*}
Using that $D_x^{\frac{1}{2}}$ is a self-adjoint operator, an application of Proposition \ref{propcomm2} shows 
\begin{equation}
\begin{aligned}
 \Big|\int \big([D_x^{\frac{1}{2}},\partial_x(\langle x\rangle_N^{2\theta})]u\big) D_x^{\frac{1}{2}} u\, dx dy\Big|&=\Big|\int D_x^{\frac{1}{2}}\big([D_x^{\frac{1}{2}},\partial_x(\langle x\rangle_N^{2\theta})]u\big)  u\, dx dy \Big|\\
 &\lesssim \|\,\|D_x^{\frac{1}{2}}\big([D_x^{\frac{1}{2}},\partial_x(\langle x\rangle_N^{2\theta})]u\big)\|_{L^2_x}\|_{L^2_y}\|u\|_{L^2}\\
 &\lesssim \|\partial_x^2(\langle x\rangle_N^{2\theta})\|_{L^{\infty}}\|u\|_{L^2}^2\\
 &\lesssim \| u\|_{L^2}^{2}.
\end{aligned}    
\end{equation}
Using again that $0<\theta\leq \frac{1}{2}$ implies $|\partial_x(\langle x\rangle_N^{2\theta})|\lesssim 1$, where the implicit constant is independent of $\theta$ and $N$, we deduce
\begin{equation*}
\begin{aligned}
\Big|\int \partial_x(\langle x\rangle_N^{2\theta}) D_x^{\frac{1}{2}}u D_x^{\frac{1}{2}} u\, dx dy\Big|\lesssim & \|\partial_x(\langle x\rangle_N^{2\theta})\|_{L^{\infty}}\|D_x^{\frac{1}{2}}u\|_{L^{2}}^2\lesssim \|u\|_{H^{\frac{1}{2},0}}^2.
\end{aligned}   
\end{equation*}
Gathering the previous estimates, we deduce
\begin{equation*}
  |\mathcal{I}|+ |\mathcal{II}|\lesssim  \| u\|_{H^{\frac{1}{2},0}}^{2}.
\end{equation*}
Next, we estimate the nonlinear terms in \eqref{eqLWP1}. To control $\mathcal{H}_x\Big(\sum_{k=1}^K\nu_k  u^k\partial_x u\Big)$, given $k=1,\dots,K$, we write
\begin{equation*}
\begin{aligned}
\langle x \rangle_N^{\theta}\mathcal{H}_x(u^k\partial_x u)= \frac{1}{k+1}[\langle x \rangle_N^{\theta},\mathcal{H}_x]\partial_x(u^{k+1})+\mathcal{H}_x(\langle x \rangle_N^{\theta} u^k\partial_x u),   
\end{aligned}    
\end{equation*}
thus, by H\"older's inequality and Young's inequality, we deduce
\begin{equation*}
\begin{aligned}
\|\langle x \rangle_N^{\theta}\mathcal{H}_x(u^k\partial_x u)\|_{L^2}\lesssim & \|\partial_x(\langle x \rangle_N^{\theta})\|_{L^{\infty}}\|u^{k+1}\|_{L^2}+\|\langle x \rangle_N^{\theta} u^k\partial_x u\|_{L^2}\\
\lesssim & \|u\|_{L^{\infty}}^k\|u\|_{L^{2}}+\|\partial_x u\|_{L^{\infty}}\|u\|_{L^{\infty}}^{k-1}\|\langle x \rangle_N^{\theta} u\|_{L^{2}}\\
\lesssim &\big(\|u\|_{L^{\infty}}^k+\|\partial_x u\|_{L^{\infty}}^k\big)\big(\|u\|_{L^{2}}+\|\langle x \rangle_N^{\theta} u\|_{L^{2}}\big).
\end{aligned}
\end{equation*}
Thus, the previous inequality shows
\begin{equation*}
 \begin{aligned}
\|\langle x \rangle_N^{\theta}\mathcal{H}_x\Big(\sum\limits_{k=1}^K \nu_k u^k\partial_x u\Big)\|_{L^2} \lesssim & \sum_{k=1}^K|\nu_k|\|\langle x \rangle_N^{\theta}\mathcal{H}_x(u^{k}\partial_x u)\|_{L^2} \\
\lesssim & \Big(\sum_{k=1}^K (\| u\|_{L^{\infty}}^k+\|\partial_x u\|_{L^{\infty}}^k)\Big)\big(\|u\|_{L^{2}}+\|\langle x \rangle_N^{\theta} u\|_{L^{2}}\big).
 \end{aligned}   
\end{equation*}
Similarly, we obtain
\begin{equation*}
 \begin{aligned}
\|\langle x \rangle_N^{\theta}\Big(\sum\limits_{k=1}^K \nu_k u^k\partial_x u\Big)\|_{L^2} \lesssim & \sum_{k=1}^K|\nu_k|\|\langle x \rangle_N^{\theta}(u^{k}\partial_x u)\|_{L^2} \\
\lesssim & \Big(\sum_{k=1}^K (\| u\|_{L^{\infty}}^k+\|\partial_x u\|_{L^{\infty}}^k)\Big)\big(\|u\|_{L^{2}}+\|\langle x \rangle_N^{\theta} u\|_{L^{2}}\big).
 \end{aligned}   
\end{equation*}
Consequently, Cauchy-Schwarz inequality yields
\begin{equation*}
\begin{aligned}
\Big|\int \Big(\sum\limits_{k=1}^K \nu_k u^k\partial_x u\Big)&(\langle x \rangle_N^{2\theta} u)\, dx dy +\int \mathcal{H}_x \Big(\sum\limits_{k=1}^K \nu_k u^k\partial_x u\Big)(\langle x \rangle_N^{2\theta} \mathcal{H}_x u)\, dx dy\Big|\\
&\lesssim \|\langle x \rangle_N^{\theta}\Big(\sum\limits_{k=1}^K \nu_k u^k\partial_x u\Big)\|_{L^2}\|\langle x \rangle_N^{\theta} u\|_{L^2}+\|\langle x \rangle_N^{\theta}\mathcal{H}_x\Big(\sum\limits_{k=1}^K \nu_k u^k\partial_x u\Big)\|_{L^2}\|\langle x \rangle_N^{\theta} \mathcal{H}_x u\|_{L^2}\\
&\lesssim  \Big(\sum_{k=1}^K (\| u\|_{L^{\infty}}^k+\|\partial_x u\|_{L^{\infty}}^k)\Big)\big(\|u\|_{L^{2}}+\|\langle x \rangle_N^{\theta} u\|_{L^{2}}\big)\\
&\hspace{3cm}\times \Big(\|\langle x \rangle_N^{\theta} u\|_{L^2}+\|\langle x \rangle_N^{\theta} \mathcal{H}_x u\|_{L^2}\Big).
\end{aligned}    
\end{equation*}
Plugging the previous estimates back to \eqref{eqLWP1}, we infer that there exist positive constants $c_0$, $c_1$ independent of $N$ such that 
\begin{equation*}
\begin{aligned}
\frac{d}{d t}\Big(\|\langle x\rangle_N^{\theta} u(t)\|_{L^2}^2&+\|\langle x\rangle_N^{\theta}\mathcal{H}_x u(t)\|_{L^2}^2\Big)\\
\leq & c_0\Big(1+\sum_{k=1}^K (\| u(t)\|_{L^{\infty}}^k+\|\partial_x u(t)\|_{L^{\infty}}^k)\Big)\|u(t)\|_{H^{\frac{1}{2},0}}^2\\
&+c_1\Big(\sum_{k=1}^K (\| u(t)\|_{L^{\infty}}^k+\|\partial_x u(t)\|_{L^{\infty}}^k)\Big)\Big(\|\langle x \rangle_N^{\theta} u(t)\|_{L^2}^2+\|\langle x \rangle_N^{\theta} \mathcal{H}_x u(t)\|_{L^2}^2\Big).   
\end{aligned}    
\end{equation*}
Setting $s_1>\frac{3}{2}$, $s_2>\frac{1}{2}$ fixed, we can use the fact that $u\in C([0,T];H^{\infty}(\mathbb{R}\times \mathbb{T}))$, and the embedding 
$H^{\frac{1}{2}^{+},\frac{1}{2}^{+}}(\mathbb{R}\times \mathbb{T})\hookrightarrow L^{\infty}(\mathbb{R}\times \mathbb{T})$ to get that there exists a universal constant $c>0$ such that 
\begin{equation*}
 \begin{aligned}
 \sup_{t\in [0,T]}\Big(\|u(t)\|_{L^{\infty}}+\|\partial_x u(t)\|_{L^{\infty}}+\|u(t)\|_{H^{\frac{1}{2}},0} \Big)\leq c \sup_{t\in [0,T]}\|u(t)\|_{H^{s_1,s_2}}=:M.   
 \end{aligned}   
\end{equation*}
Thus, we arrive at 
\begin{equation*}
\begin{aligned}
\frac{d}{dt}\Big(\|\langle x\rangle_N^{\theta} u(t)\|_{L^2}^2&+\|\langle x\rangle_N^{\theta}\mathcal{H}_x u(t)\|_{L^2}^2\Big)\\
\leq & c_0\Big(1+2\sum_{k=1}^K M^k\Big)M^2+2c_1\Big(\sum_{k=1}^K M^k \Big)\Big(\|\langle x \rangle_N^{\theta} u(t)\|_{L^2}^2+\|\langle x \rangle_N^{\theta} \mathcal{H}_x u(t)\|_{L^2}^2\Big),  
\end{aligned}    
\end{equation*}
which after Gronwall's inequality in the case $0<\theta<\frac{1}{2}$ allows us to deduce
\begin{equation}\label{afterGronwa}
\begin{aligned}
 \|\langle x\rangle_N^{\theta} u(t)\|_{L^2}^2+&\|\langle x\rangle_N^{\theta}\mathcal{H}_x u(t)\|_{L^2}^2\\
 \leq & e^{2c_1\big(\sum_{k=1}^K M^k \big)t}\Big(\|\langle x\rangle_N^{\theta} u_0\|_{L^2}^2+\|\langle x\rangle_N^{\theta}\mathcal{H}_x u_0\|_{L^2}^2+c_0\Big(1+2\sum_{k=1}^K M^k\Big)M^2|t|\Big)\\
 \leq & e^{2c_1\big(\sum_{k=1}^K M^k \big)t}\Big(\|\langle x\rangle^{\theta} u_0\|_{L^2}^2+c_0\Big(1+2\sum_{k=1}^K M^k\Big)M^2|t|\Big),
\end{aligned}    
\end{equation}
where we have used $\|\langle x\rangle_N^{\theta}\mathcal{H}_x u_0\|_{L^2}^2\leq \|\langle x\rangle^{\theta} u_0\|_{L^2}^2<\infty$, which follows from Proposition \ref{a2condH} when $0<\theta<\frac{1}{2}$. Now, when $\theta=\frac{1}{2}$, we get inequality \eqref{afterGronwa} but with the addition of the term $\|\langle x\rangle^{\theta} \mathcal{H}_x u_0\|_{L^2}^2$, which is finite by hypothesis. Thus, by the above inequality and taking $N\to \infty$, we see that $u\in L^{\infty}([0,T];L^2(|x|^{2\theta}\, dx dy))$. Continuity follows by similar arguments in the proof of \cite[Theorem 1.3]{Riano2021} and references therein.  This completes the deduction of $(i)$.
\\ \\
Let us now deal with part $(ii)$. We assume that $u_0\in L^{2}(|x|^{2\theta}\, dx dy)$, for some $\frac{1}{2}<\theta<\frac{3}{2}$, and $\widehat{u_0}(0,\eta)=0$ for all $\eta\in \mathbb{Z}$. Now, let $0<\theta_1<\frac{1}{2}$ be fixed. By $(i)$, we know that $u\in C([0,T];L^2(|x|^{2\theta_1}\, dx))$. Thus, we are going to use this fact and the integral formulation of \eqref{SHeq} to deduce that $u\in C([0,T];L^2(|x|^{2\theta_2}\, dx))$ for some $0<\theta_1<\theta_2\leq \theta$. We first recall that since $u$ solves \eqref{SHeq}, it solves the integral equation
\begin{equation}\label{inteque}
\begin{aligned}
u(t)=U(t)u_0-\int_0^t U(t-\tau)\Big(\sum_{k=1}^K \nu_k u^{k}\partial_x u(\tau)\Big)\, d\tau.   
\end{aligned}    
\end{equation}
Thus, since $\widehat{u_0}(0,\eta)=0$ for all $\eta\in \mathbb{Z}$, we can apply Theorem \ref{LinearEst} part $(iii)$ to get $U(t)u_0\in C([0,T];L^{2}(|x|^{2\theta}\, dxdy))$. Consequently, we are reduced to control the integral factor in \eqref{inteque}.

We consider a fixed number $\theta_1<\theta_2\leq \theta$, where
\begin{equation}\label{defotheta}
   \theta_2=\left\{\begin{aligned}
&\frac{(k+2)}{2}\theta_1, \, \, \, \text{ if } \, \, \frac{(k+2)\theta_1}{2}\leq \theta, \\
&\theta, \hspace{1.5cm} \text{ if } \, \, \frac{(k+2)\theta_1}{2}> \theta.
   \end{aligned}\right. 
\end{equation}
We notice that for all $k=1,\dots,K$, $\widehat{u^k \partial_x u}(\xi,\eta,\tau)=\frac{1}{k+1}(i\xi) \widehat{u^{k+1}}(\xi,\eta,\tau)$, it follows $\widehat{u^k \partial_x u}(0,\eta,\tau)=0$, for all $\eta\in \mathbb{Z}$, and $\tau\in (0,T)$. Thus, we apply Theorem \ref{LinearEst} $(iii)$ to get
\begin{equation}\label{integraleqinweights}
 \begin{aligned}
    \|\langle x\rangle^{\theta_2}\int_0^t U(t-\tau)\Big(\sum_{k=1}^K \nu_k u^{k}\partial_x u(\tau)\Big)\, d\tau\|_{L^2}
    \lesssim & \sum_{k=1}^K\int_0^t \Big(\|u^k\partial_x u\|_{H^{\theta_2,0}} +\|\langle x\rangle^{\theta_2}u^k\partial_x u\|_{L^2}\Big)\, d\tau.
 \end{aligned}   
\end{equation}
Next, we will use the fractional Leibniz rule
\begin{equation}\label{fLR}
\begin{aligned}
\|J^s(f g)\|_{L^2}\lesssim \|J^s f\|_{L^2}\|g\|_{L^{\infty}}+\|f\|_{L^{\infty}}\|J^s g\|_{L^2}, 
\end{aligned}
\end{equation}
which for arbitrary $s>0$ was deduced in \cite[Theorem 1]{GrafakosOh2014} (see also \cite{KatoPonce1988}). Then, by multiple applications of the previous inequality in the $x$-variable, it follows
\begin{equation*}
\begin{aligned}
    \|u^k\partial_x u\|_{H^{\theta_2,0}}\lesssim \|\|J^{\theta_2+1}_x(u^{k+1})\|_{L^2_x}\|_{L^2_y}\lesssim & \|\|u\|_{L^{\infty}_x}^k\|J^{\theta_2+1}_x u\|_{L^2_x}\|_{L^2_y}\\
    \lesssim & \|u\|_{L^{\infty}}^k\|u\|_{H^{\theta_2+1,0}}. 
\end{aligned}    
\end{equation*}
On the other hand, writing $u^k\partial_x u=\frac{1}{k+1}\partial_x(u^{k+1})$ and distributing the derivative $\partial_x$, we apply complex interpolation \eqref{interpineq2} to deduce
\begin{equation}\label{finaleq1}
\begin{aligned}
 \|\langle x\rangle^{\theta_2}u^k\partial_x u\|_{L^2}  \lesssim & \|\|J_x(\langle x\rangle^{\theta_2} u^{k+1})\|_{L^2_x}\|_{L^2_y}\\
 \lesssim & \|\|J_x^{\frac{(3k+4)\theta_1}{(3k+4)\theta_1-4\theta_2}}( u^{k+1})\|_{L^2_x}\|_{L^2_y}+\|\|\langle x\rangle^{\frac{(3k+4)}{4}\theta_1} u^{k+1}\|_{L^2_x}\|_{L^2_y}.
\end{aligned}    
\end{equation}
Using \eqref{fLR}, and the fact that $u(t)$ describes a continuous curve in $H^{\infty}(\mathbb{R}\times \mathbb{T})$, we have that the first term on the right-hand side of \eqref{finaleq1} is bounded. Next, we deal with the second term on the right-hand side of \eqref{finaleq1}.

Applying Sobolev embedding $H^{\frac{k}{2(k+1)}}(\mathbb{R})\hookrightarrow L^{2(k+1)}(\mathbb{R})$, $H^{\frac{1}{2}^{+}}(\mathbb{T})\hookrightarrow L^{\infty}(\mathbb{T})$, and the fact that $\mathbb{T}$ defines a finite measure space, we deduce
\begin{equation}\label{finaleq2}
\begin{aligned}
\|\|\langle x\rangle^{\frac{(3k+4)}{4}\theta_1} u^{k+1}\|_{L^2_x}\|_{L^2_y}=\|\|\langle x\rangle^{\frac{(3k+4)}{4(k+1)}\theta_1} u\|_{L^{2(k+1)}_x}^{k+1}\|_{L^{2}_{y}}\lesssim &\|\|J_{x}^{\frac{k}{2(k+1)}}(\langle x\rangle^{\frac{(3k+4)}{4(k+1)}\theta_1} u)\|_{L^2_{x}}^{k+1}\|_{L^2_y}\\
\lesssim &\Big(\sup_{y\in \mathbb{T}}\|J_{x}^{\frac{k}{2(k+1)}}(\langle x\rangle^{\frac{(3k+4)}{4(k+1)}\theta_1} u)\|_{L^2_{x}}\Big)^{k+1}\\
\lesssim &  \|J_{x}^{\frac{k}{2(k+1)}}J_{y}^{1}(\langle x\rangle^{\frac{(3k+4)}{4(k+1)}\theta_1} u)\|_{L^{2}_{x,y}}^{k+1}.
\end{aligned}    
\end{equation}
To complete the estimate of the right-hand side of the above inequality, we apply interpolation Lemma \ref{InterpLemma} to obtain 
\begin{equation}\label{finaleq3}
\begin{aligned}
   \|J_{x}^{\frac{k}{2(k+1)}}J_{y}^{1}(\langle x\rangle^{\frac{(3k+4)}{4(k+1)}\theta_1} u)\|_{L^{2}_{x,y}}^{k+1}\lesssim & \|J_{x,y}^{\frac{k}{2(k+1)}+1}(\langle x\rangle^{\frac{(3k+4)}{4(k+1)}\theta_1} u)\|_{L^{2}_{x,y}}^{k+1}\\ 
    \lesssim & \|\langle x\rangle^{\theta_1} u\|_{L^{2}_{x,y}}^{k+1}+\|J_{x,y}^{10}u\|_{L^2}^{k+1}.
\end{aligned}
\end{equation}
Consequently, the previous estimates, the fact that $u\in  C([0,T];H^{\infty}(\mathbb{R}\times \mathbb{T})\cap L^2(|x|^{2\theta_1}))$ allow us to conclude  
\begin{equation*}
 \int_0^t U(t-\tau)\Big(\sum_{k=1}^K \nu_k u^{k}\partial_x u(\tau)\Big)\, d\tau\in C([0,T];L^2(|x|^{2\theta_2}\, dx dy)).   
\end{equation*}
and in turn $u\in C([0,T];L^2(|x|^{2\theta_2}\,dxdy))$. Consequently, if $\theta_2=\theta$, we are done. In the case $\theta_2<\theta$, replacing $\theta_1$  by $\theta_2$, we can define $\theta_3$ as we did in \eqref{defotheta}, i.e., $\theta_3=\frac{(k+2)}{2}\theta_2$, if $\frac{(k+2)}{2}\theta_2\leq \theta$, and $\theta_3=\theta$, if $\frac{(k+2)}{2}\theta_2>\theta$.  Thus, replacing $\theta_2$ and $\theta_1$ by $\theta_3$ and $\theta_2$, respectively in the arguments in \eqref{finaleq1}-\eqref{finaleq3}, we get $u\in C([0,T];L^2(|x|^{2\theta_3}\,dxdy))$. Finally, since for each $k=1,\dots,K$, $\frac{k+2}{2}>1$, the rest of the proof runs as before, iterating the previous argument a finite number of times to conclude  $u\in C([0,T];L^2(|x|^{2\theta}\,dxdy))$.
\end{proof}

%%%%%%%%%%%%%%%%%%%%%%%%%%%%%%%%%%%%%%%%%%%%%%%%%%%%%%%%%%%%%%%%%%%%%%%%%%%%%%%%%%%%%%%%%%%%%%%%%%%%%%%%%%%%%%%%%%%5%%%%%%%%%%%%%%%%%%%%%%%%%%%%%%%%%%%%%%%%%%%%%%%%%%%%%%%%%%%%%%%%%%%%%%%%%%%%%%%%%%%%%%%%%%%%%%%%%%%%%%%%%%%%%%%%%%%%%%%%%%%%

\subsection{Proof of Theorem \ref{LWPresultinWS}}

The results of Lemma \ref{smoothsolutdecay} establish persistence in weights spaces for smooth solutions of \eqref{SHeq}, where we remark that the proof requires solutions with a higher order of regularity. To optimize such regularity in the general case of Theorem \ref{LWPresultinWS}, we will use energy estimates and approximation with smooth solutions that satisfy the conclusions of Lemma \ref{smoothsolutdecay}.

\begin{proof}[Proof of Theorem \ref{LWPresultinWS}]
Let $0<\theta<\frac{3}{2}$, $s_1>\frac{3}{2}$, $s_2>\frac{1}{2}$ and $u_0\in H^{s_1,s_2}(\mathbb{R}\times \mathbb{T})\cap L^2(|x|^{2\theta}\, dx dy)$ where, if $\theta=\frac{1}{2}$, we further assume $\mathcal{H}_xu_0\in L^2(|x|\, dx dy)$, and if $\frac{1}{2}<\theta<\frac{3}{2}$, we also assume $\widehat{u}_0(0,\eta)=0$ for all $\eta\in \mathbb{Z}$. By Lemma \ref{StandardLWP}, there exist a time $T>0$ and a unique solution $u\in C([0,T];H^{s_1,s_2}(\mathbb{R}\times \mathbb{T}))$ of the Cauchy problem \eqref{SHeq}  with initial condition $u_0$. Moreover, given $N\geq 1$ be integer, we define the sequence $P_N u_0\in H^{\infty}(\mathbb{R}\times \mathbb{T})$, where $(P_N u_0)^{\wedge}(\xi,\eta)=\phi(\xi/N,\eta/N)\widehat{u}_0(\xi,\eta)$, and $\phi\in C^{\infty}_c(\mathbb{R}^2)$ with $\phi(\xi,\eta)=1$, whenever $|(\xi,\eta)|\leq 1$. Thus, by Lemma \ref{StandardLWP} (continuous dependence in it), we can take $N_0$ sufficiently large such that for all $N\geq N_0$, there exists a unique smooth solution $u_N\in C([0,T];H^{\infty}(\mathbb{R}\times \mathbb{T}))$ of \eqref{SHeq}, with initial condition $P_Nu_0$, such that $u_N\to u$ as $N\to \infty$ in the topological sense of $C([0,T];H^{s_1,s_2}(\mathbb{R}\times \mathbb{T}))$. Also, observe that $P_Nu_0\to u_0$ in $L^2(|x|^{2\theta}\, dx dy)$.

In what follows, we will use approximation with the sequence of smooth solutions $u_N$ and Lemma \ref{smoothsolutdecay} to show that $u\in C([0,T];L^2(|x|^{2\theta}\, dx dy))$. Once this has been established, the continuous dependence in $L^2(|x|^{2\theta}\, dx dy)$ follows by our arguments below and approximation. This, in turn, leads to the well-posedness statement in Theorem \ref{LWPresultinWS}. 

Let us show that $u\in C([0,T];L^2(|x|^{2\theta}\, dx dy))$. We will consider two cases: $0<\theta\leq \frac{1}{2}$ and $\frac{1}{2}<\theta<\frac{3}{2}$. But, we first notice that $u_0\in L^2(|x|^{2\theta}\, dx dy)$ implies that $P_Nu_0\in L^2(|x|^{2\theta}\, dx dy)$, which can be deduced by the properties of the fractional derivative \eqref{prelimneq0}-\eqref{prelimneq1}. For example, one can use some of the ideas in the proof of Theorem \ref{LinearEst}. Additionally, since the operator $P_N$ commutes with the Hilbert transform,  we have that $\mathcal{H}_x u_0\in L^2(|x|^{2\theta}\, dx dy)$ implies $\mathcal{H}_x P_N u_0\in L^2(|x|^{2\theta}\, dx dy)$. Besides that, we have that if $\widehat{u}_0(0,\eta)=0$ for all $\eta\in \mathbb{Z}$, it follows $\widehat{P_Nu_0}(0,\eta)=0$. Summarizing, we have that the conditions over $u_0$ in the statements of Theorem \ref{LWPresultinWS} transfer to $P_Nu_0$.
\\ \\
\underline{Proof of Theorem \ref{LWPresultinWS} $(i)$ and $(ii)$; weights $0<\theta\leq \frac{1}{2}$}. By the previous discussion and Lemma \ref{smoothsolutdecay}, we have that $u_N\in C([0,T];L^2(|x|^{2\theta}\, dx dy))$. Thus, the same energy estimate in the proof of $(i)$ in  Lemma \ref{smoothsolutdecay} (see \eqref{afterGronwa}) applied to the sequence $u_N$, allows us to take $N\to \infty$ to deduce  $u\in C([0,T];L^2(|x|^{2\theta}\, dx dy))$. We remark that such an argument is valid provided that $u_N\to u$ as $N\to \infty$ in $C([0,T];H^{s_1,s_2}(\mathbb{R}\times \mathbb{T}))$.
\\ \\
\underline{Proof of Theorem \ref{LWPresultinWS} $(iii)$; weights $\frac{1}{2}<\theta< \frac{3}{2}$}. Once again, the fact that $\widehat{P_Nu_0}(0,\eta)=0$ and Lemma \ref{smoothsolutdecay} show that $u_N\in C([0,T];L^2(|x|^{2\theta}\, dx dy))$. We write $\theta=\theta_1+\theta_2$, where $\theta_1\in \{0,1\}$, $\theta_2\in [0,1)$. Arguing as in the deduction of \eqref{eqLWP1}, it is not hard to get
\begin{equation}\label{integraleqN}
\begin{aligned}
\frac{1}{2}\frac{d}{dt}\Big(\int (x^{\theta_1}\langle x \rangle^{\theta_2} u_N)^2\, dx dy&+\int (x^{2\theta_1}\langle x \rangle^{2\theta_2}\mathcal{H}_x u_N)^2\, dx dy\Big)\\
=&\underbrace{\int \mathcal{H}_x\partial_x^2 u_N(x^{2\theta_1}\langle x \rangle^{2\theta_2} u_N)\, dx dy-\int \partial_x^2 u_N(x^{2\theta_1}\langle x \rangle^{2\theta_2} \mathcal{H}_x u_N)\, dx dy}_{DT}\\
&-\underbrace{\sum_{k=1}^K\nu_k\int u_N^k\partial_x u_N(x^{2\theta_1}\langle x \rangle^{2\theta_2} u_N)\, dx dy}_{NL1}\\
&-\underbrace{\sum_{k=1}^K\nu_k\int \mathcal{H}_x (u^k_N\partial_x u_N)(x^{2\theta_1}\langle x \rangle^{2\theta_2} \mathcal{H}_x u_N)\, dx dy}_{NL2}.
\end{aligned}    
\end{equation}
We remark that above we used directly the weight $x^{\theta_1}\langle x \rangle^{\theta_2}$ not an approximation of it, which is justified by the fact that $u_N\in C([0,T];H^{\infty}(\mathbb{R}\times \mathbb{T}))\cap C([0,T];L^2(|x|^{2\theta}\, dx dy))$. We first estimate $NL1$ and $NL2$ in \eqref{integraleqN}. Integration by parts and Young's inequality yield
\begin{equation*}
\begin{aligned}
| NL1|=&\Big|\sum_{k=1}^K \frac{\nu_k}{2}\int \partial_x(x^{2\theta_2}\langle x\rangle^{2\theta_1} u_{N}^k)u_{N}^2\, dx dy  \Big|\\
\lesssim & \sum_{k=1}^{K}\big(\|u_{N}\|_{L^{\infty}}^k+\|u_{N}\|_{L^{\infty}}^{k-1}\|\partial_x u_{N}\|_{L^{\infty}}\big)\|x^{\theta_1}\langle x \rangle^{\theta_2}u_{N}\|_{L^2}^2\\
\lesssim & \sum_{k=1}^K (\|u_N\|_{L^{\infty}}^{k}+\|\partial_x u_{N}\|_{L^{\infty}}^{k})\|x^{\theta_1}\langle x \rangle^{\theta_2}u_{N}\|_{L^2}^2.    
\end{aligned}    
\end{equation*}
To estimate $NL2$, we take $k=1,\dots,K$ fixed. Given that $\partial_x(u_N^k)$ is such that $\widehat{\partial_x(u_N^k)}(0,\eta)=0$ for all $\eta \in \mathbb{Z}$, we use \eqref{conmm} and the fact that $\theta_1\in \{0,1\}$ to get
\begin{equation*}
\begin{aligned}
x^{\theta_1}\langle x\rangle^{\theta_2} \mathcal{H}_x (u^k_N\partial_x u_N)=\frac{1}{k+1}\langle x\rangle^{\theta_2} \mathcal{H}_x (x^{\theta_1} \partial_x (u^{k+1}_N)).  
\end{aligned}    
\end{equation*}
Now, since $\theta=\theta_1+\theta_2\in (\frac{1}{2},\frac{3}{2})$, when $\theta_1=0$, it follows that $\theta_2\in (\frac{1}{2},1)$, then we can apply Remark \ref{remarkHilbertL2} to the previous identity to deduce
\begin{equation*}
 \begin{aligned}
\|x^{\theta_1}\langle x\rangle^{\theta_2} \mathcal{H}_x (u^k_N\partial_x u_N)\|\lesssim\|x^{\theta_1}\langle x\rangle^{\theta_2}  \partial_x (u^{k+1}_N)\|_{L^2} \lesssim \|u_N\|^{k-1}_{L^{\infty}}\|\partial_x u_N\|_{L^{\infty}}\|x^{\theta_1}\langle x\rangle^{\theta_2}u_N\|_{L^2}.
 \end{aligned}   
\end{equation*}
The exact same estimate above holds when $\theta_1=1$, since this implies that $\theta_1\in (0,\frac{1}{2})$, and thus the Hilbert transform $\mathcal{H}_x$ defines a bounded operators from $L^2(|x|^{2\theta_2}\, dx dy)$ into itself. Thus, we deduce
\begin{equation*}
\begin{aligned}
 |NL2|\leq & \sum_{k=1}^K\|x^{\theta_1}\langle x\rangle^{\theta_2} \mathcal{H}_x (u^k_N\partial_x u_N)\|\|x^{\theta_1}\langle x\rangle^{\theta_2}u_N\|_{L^2}\\
 \lesssim& \sum_{k=1}^K\|u_N\|^{k-1}_{L^{\infty}}\|\partial_x u_N\|_{L^{\infty}}\|x^{\theta_1}\langle x\rangle^{\theta_2}u_N\|_{L^2}^2\\
 \lesssim & \sum_{k=1}^K (\|u_N\|_{L^{\infty}}^{k}+\|\partial_x u_{N}\|_{L^{\infty}}^{k})\|x^{\theta_1}\langle x \rangle^{\theta_2}u_{N}\|_{L^2}^2.  
\end{aligned}    
\end{equation*}
Next, we deal with the dispersive term $DT$ in \eqref{integraleqN}. Using integration by parts twice, we have
\begin{equation*}
\begin{aligned}
DT=&\int u_n\big(\partial_x^2(x^{2\theta_1}\langle x \rangle^{2\theta_2})\mathcal{H}_x u_N\big)\, dx dy+   2\int u_n\big(\partial_x(x^{2\theta_1}\langle x \rangle^{2\theta_2})\partial_x\mathcal{H}_x u_N\big)\, dx dy \\
=&: DT1+DT2.
\end{aligned}    
\end{equation*}
Using continuity of the Hilbert transform, the fact that $|\partial_x^2 (x^{2\theta_1}\langle x\rangle^{2\theta_2})|\lesssim \langle x\rangle^{\theta_1+\theta_2}$, it is seen that
\begin{equation*}
\begin{aligned}
|DT1|\lesssim & \|\partial_x^2(x^{\theta_1}\langle x \rangle^{\theta_2})u_N\|_{L^{2}}\|u_N\|_{L^2}\leq \|\langle x \rangle^{\theta_1+\theta_2}u_N\|_{L^2}\|u_N\|_{L^2}\\
\lesssim &\|u_N\|_{L^2}^2+\|x^{\theta_1}\langle x\rangle^{\theta_2}u_N\|_{L^2}^2.
\end{aligned}    
\end{equation*}
On the other hand, when $\theta_1+\theta_2>1$, using that $|\partial_x (x^{2\theta_1}\langle x\rangle^{2\theta_2})|\lesssim \langle x\rangle^{2\theta_1+2\theta_2-1}$ and interpolation, we infer
\begin{equation*}
\begin{aligned}
|DT2|\lesssim & \|\langle x \rangle^{\theta_1+\theta_2} u_N\|_{L^2}\|\langle x \rangle^{\theta_1+\theta_2-1}\partial_x \mathcal{H}_xu_N\|_{L^2}\\
\lesssim &  \|\langle x \rangle^{\theta_1+\theta_2} u_N\|_{L^2}\big(\|\|J_x(\langle x \rangle^{\theta_1+\theta_2-1}\mathcal{H}_x u_N)\|_{L^2_x}\|_{L^2_y}\big) \\
\lesssim &  \|\langle x \rangle^{\theta_1+\theta_2} u_N\|_{L^2}\big(\|J_x^{\theta_1+\theta_2}u_N\|_{L^2}+\|x^{\theta_1}\langle x\rangle^{\theta_2}\mathcal{H}_xu_N\|_{L^2}\big).
\end{aligned}    
\end{equation*}
We remark that the same conclusion of the inequality above holds when $\theta_1+\theta_2\leq 1$ as in such a case, we do not need interpolation and the estimate follows from $\langle x\rangle^{\theta_1+\theta_2-1}\lesssim 1$. Hence, Young's inequality yields
\begin{equation*}
\begin{aligned}
|DT2|\lesssim  \|J_x^{\theta_1+\theta_2}u_N\|_{L^2}^2+\|x^{\theta_1}\langle x\rangle^{\theta_2}u_N\|_{L^2}^2+\|x^{\theta_1}\langle x\rangle^{\theta_2}\mathcal{H}_xu_N\|_{L^2}^2.
\end{aligned}    
\end{equation*}
Since $u_N \to u$ in $C([0,T];H^{s_1,s_2}(\mathbb{R}\times \mathbb{T}))$, with $s_1\geq \max\{\theta_1+\theta_2,\frac{3}{2}^{+}\}$, there exists $N_1$ large enough such that for all $N\geq N_1$,
\begin{equation*}
 \sup_{t\in[0,T]}\big(\|u_N(t)\|_{L^{\infty}}+\|\partial_x u_N(t)\|_{L^{\infty}}+\|u_N(t)\|_{H^{s_1,s_2}}\big)\leq 2\Big(\sup_{t\in [0,T]}\|u(t)\|_{H^{s_1,s_2}}\Big)=:M_1,   
\end{equation*}
where we controlled the $L^{\infty}$-norms above by Sobolev embedding $H^{\frac{1}{2}^{+},\frac{1}{2}^{+}}(\mathbb{R}\times\mathbb{T})\subset L^{\infty}(\mathbb{R}\times \mathbb{R})$. Thus, collecting the estimates for $DT$, $NL1$ and $NL2$, developed above, we deduce that there exist some positive constants $c_1,c_2$ independent of $N$ and $M_1$ such that 
\begin{equation*}
\begin{aligned}
\frac{d}{dt}\Big(&\|x^{\theta_1}\langle x\rangle^{\theta_2}u_N\|_{L^2}^2+\|x^{\theta_1}\langle x\rangle^{\theta_2}\mathcal{H}_xu_N\|_{L^2}^2\Big)\\
&\leq c_1M_1+c_2\langle M_1 \rangle^K \Big(\|x^{\theta_1}\langle x\rangle^{\theta_2}u_N\|_{L^2}^2+\|x^{\theta_1}\langle x\rangle^{\theta_2}\mathcal{H}_xu_N\|_{L^2}^2\Big).   
\end{aligned}    
\end{equation*}
Hence, Gronwall's inequality establishes that for all $t\in (0,T]$
\begin{equation*}
\begin{aligned}
\|x^{\theta_1}\langle x\rangle^{\theta_2}u_N(t)\|_{L^2}^2+\|x^{\theta_1}\langle x\rangle^{\theta_2}\mathcal{H}_xu_N(t)\|_{L^2}^2\leq & e^{c_2 \langle M_1 \rangle^K  t}\big(c_1 M_1 t+\|x^{\theta_1}\langle x\rangle^{\theta_2}P_Nu_0\|_{L^2}^2+\|x^{\theta_1}\langle x\rangle^{\theta_2}\mathcal{H}_xP_Nu_{0}\|_{L^2}^2\big)\\
\leq & e^{c_2 \langle M_1 \rangle^K  T}\big(c_1 M_1 T+\|x^{\theta_1}\langle x\rangle^{\theta_2}P_Nu_0\|_{L^2}^2+\|x^{\theta_1}\langle x\rangle^{\theta_2}\mathcal{H}_xP_Nu_{0}\|_{L^2}^2\big) .  
\end{aligned}
\end{equation*}
Since $P_Nu_0\to u_0$ in $L^2(|x|^{2\theta}\, dx dy)$ and $u_N(x,t)\to u(x,t)$ uniformly in space for each time $t\in [0,T]$ as $N\to \infty$, we can apply Fatou's lemma and the previous inequality to conclude $u\in L^{\infty}([0,T];L^2(|x|^{2\theta}\, dx dy))$. Continuity follows by the energy estimate above and similar arguments in \cite[Theorem 1.3]{Riano2021}.
\end{proof}

%%%%%%%%%%%%%%%%%%%%%%%%%%%%%%%%%%%%%%%%%%%%%%%%%%%%%%%%%%%%%%%%%%%%%%%%%%%%%%%%%%%%%%%%%%%%%%%%%%%%%%%%%%%%%%%%%%%%%%%%%%%%%%%%%%%%%%%%%%%%%%%%%%%%%%%%%%%%%%%%%%%%%%%%%%%%%%%%%%%%%%%%%%%%%%%%%%%%%%%%%%%%%%%%%%%%

\subsection{Proof of Theorem \ref{Uniquecontprinc} part $(i)$}

Let $K\geq 1$ be integer, $\frac{1}{4}<\theta<\frac{1}{2}$, $s_1\geq \frac{8\theta+2}{4\theta-1}$, $s_2> \frac{6\theta}{4\theta-1}$ and $u\in C([0,T];H^{s_1,s_2}(\mathbb{R}\times \mathbb{T}))\cap L^{\infty}([0,T];L^2(|x|^{2\theta}dxdy))$ be a solution of \eqref{SHeq} such that there exist two different times $t_2>t_1$, for which $u(t_1)\in H^{s_1,s_2}(\mathbb{R}\times\mathbb{T})\cap L^2(|x|^{1^{+}}\, dx dy)$ and $u(t_2)\in H^{s_1,s_2}(\mathbb{R}\times\mathbb{T})\cap L^2(|x|^{1^{+}}\, dx dy)$. We need the following result.
\begin{claim}\label{decayintegraleq}
Let $\nu_k\in \mathbb{R}$, $k=1,\dots,K$. Consider $\frac{1}{4}<\theta<\frac{1}{2}$, $s_1\geq \frac{8\theta+2}{4\theta-1}$, $s_2>\frac{6\theta}{4\theta-1}$ and $u\in C([0,T];H^{s_1,s_2}(\mathbb{R}\times \mathbb{T}))\cap C([0,T];L^2(|x|^{2\theta} dxdy))$ be a solution of \eqref{SHeq}. Then it follows
\begin{equation*}
 \int_0^t U(t-\tau)\Big(\sum_{k=1}^K \nu_k u^{k}\partial_x u(\tau)\Big)\, d\tau  \in C([0,T];L^2(|x|^{1^{+}} dxdy)).
\end{equation*}
\end{claim}

\begin{proof}[Proof of Claim \ref{decayintegraleq}]
Given $k=1,\dots, K$ and $\frac{1}{4}<\theta<\frac{1}{2}$, let $\theta_1=\frac{4\theta+3}{8}$. We first observe
\begin{equation}\label{eqinclaim1}
u^k\partial_x u \in  C([0,T];H^{\theta_1,0}(\mathbb{R}\times \mathbb{T}))\cap C([0,T];L^2(|x|^{2\theta_1}\, dx)).  
\end{equation}
Since $\theta_1+1\leq s_1$, the fact that $u^k\partial_x u $ is in the space $C([0,T];H^{\theta_1,0}(\mathbb{R}\times \mathbb{T}))$ follows from \eqref{fLR} and the fact that $u(t)$ describes a continuous curve in $H^{s_1,s_2}(\mathbb{R}\times \mathbb{T})$. To show persistence in weighted spaces, setting $\theta_2=\frac{4\theta+1}{4}$, interpolation yields  
\begin{equation*}
\begin{aligned}
\|\langle x\rangle^{\theta_1}u^{k}\partial_x u\|_{L^2}\lesssim & \|\|J_x(\langle x\rangle^{\theta_1}u^{k+1})\|_{L^2_x}\|_{L^2_y}\\
\lesssim& \|J^{\frac{2(4\theta+1)}{4\theta-1}}_x u\|_{L^2}+\|\|\langle x \rangle^{\frac{\theta_2}{(k+1)}}u\|_{L^{2(k+1)}_x}^{k+1}\|_{L^2_y}.
\end{aligned}    
\end{equation*}
Note that the first term on the inequality above is controlled in the space $H^{s_1,0}(\mathbb{R}\times \mathbb{T})$ with $s_1\geq \frac{8\theta+2}{4\theta-1}$. Arguing as in \eqref{finaleq2} and \eqref{finaleq3}, we apply Lemma \ref{InterpLemma} to deduce
\begin{equation*}
\begin{aligned}
   \|\|\langle x \rangle^{\frac{\theta_2}{(k+1)}}u\|_{L^{2(k+1)}_x}^{k+1}\|_{L^2_y}\lesssim& \|J_{x,y}^{\frac{k}{2(k+1)}+\frac{1}{2}^{+}}\big(\langle x \rangle^{\frac{\theta_2}{(k+1)}}u\big)\|_{L^{2}}^{k+1}\\
    \lesssim& \|J^{\frac{4(k+1)\theta}{4k\theta-1}\big(\frac{k}{2(k+1)}+\frac{1}{2}^{+}\big)}_{x,y} u\|_{L^2}^{k+1}+\|\langle x \rangle^{\theta}u\|_{L^{2}}^{k+1}.
\end{aligned}    
\end{equation*}
We have that the definition of $\frac{1}{2}^{+}$, and the fact that $k\geq 1$ imply $\frac{4(k+1)\theta}{4k\theta-1}\big(\frac{k}{2(k+1)}+\frac{1}{2}^{+}\big)\leq \frac{6\theta}{4\theta-1}+\epsilon$, for some $0<\epsilon\ll 1$. This is the regularity condition $s_2>\frac{6\theta}{4\theta-1}$.  Using the estimates above, we have $u^k\partial_x u$ is in the space $C([0,T];L^2(|x|^{\theta_1}\, dx))$. This completes the deduction of \eqref{eqinclaim1}. Now, the argument in \eqref{integraleqinweights}, which depends on Theorem \ref{LinearEst} $(iii)$, allows us to use \eqref{eqinclaim1} to get the desired result.
\end{proof}

We are in a position to show the first part of Theorem \ref{Uniquecontprinc}.

\begin{proof}[Proof of Theorem \ref{Uniquecontprinc} (i)]
    By assumptions on $u(t_1)$, $u(t_2)$, and Claim \ref{decayintegraleq}, the integral formulation of \eqref{inteque}  implies $U(t_1)u_0\in H^{s_1,s_2}(\mathbb{R}\times\mathbb{T})\cap L^2(|x|^{1^{+}}dx dy)$, and $U(t_2)u_0\in H^{s_1,s_2}(\mathbb{R}\times\mathbb{T})\cap L^2(|x|^{1^{+}} dx dy)$. Thus, setting $f=U(t_1)u_0$, we can apply Theorem \ref{lineaerequniquecont1} $(i)$ to get
\begin{equation*}
    \widehat{U(t_1)u_0}(0,\eta)=0, \,\, \text{ for all $\eta\in \mathbb{Z}$ be such that $|t_2-t_1|\eta^2\neq 2\pi l'$, for all $l'\in \mathbb{Z}$ with $l'\geq 0$}.
\end{equation*}
Since $U(t_1)u_0\in L^2(|x|^{1^{+}} dx dy)$, it follows from Sobolev embedding that the function $\xi\rightarrow\widehat{U(t_1)u_0}(\xi,\eta)$ is continuous for all $\eta\in \mathbb{Z}$ and the same is true for the Fourier transform of $ \int_0^t U(t-\tau)\Big(\sum_{k=1}^K \nu_k u^{k}\partial_x u(\tau)\Big)\, d\tau $ provided Claim \ref{decayintegraleq}. Hence, gathering the previous results, we deduce that for each $t\geq t_1$,
\begin{equation}
\begin{aligned}
\lim_{\xi\to 0} \widehat{u}(\xi,\eta,t)=&\lim_{\xi\to 0} \bigg(e^{i\omega(\xi,\eta)(t-t_1)}\widehat{U(t_1)u_0}(\xi,\eta)-\mathcal{F}\Big( \int_0^t U(t-\tau)\big(\sum_{k=1}^K \nu_k u^{k}\partial_x u(\tau)\big)\, d\tau  \Big)(\xi,\eta)\bigg)  \\
=&0,
\end{aligned}    
\end{equation}
for all $\eta\in \mathbb{Z}$ be such that $|t_2-t_1|\eta^2\neq 2\pi l'$, for all $l'\in \mathbb{Z}$, $l'\geq 0$. This completes the proof of Theorem \ref{Uniquecontprinc}. 
\end{proof}

%%%%%%%%%%%%%%%%%%%%%%%%%%%%%%%%%%%%%%%%%%%%%%%%%%%%%%%%%%%%%%%%%%%%%%%%%%%%%%%%%%%%%%%%%%%%%%%%%%%%%%%%%%%%%%%%%%%%%%%%%%%%%%%%%%%%%%%%%%%%%%%%%%%%%

\subsection{Proof of Theorem \ref{Uniquecontprinc} part $(ii)$}

Transferring decay to regularity in the frequency domain, the hypothesis in Theorem \ref{Uniquecontprinc} $(ii)$ implies that $J^{\frac{3}{2}^{+}}_{\xi}\widehat{u}(t_j)\in L^2(\mathbb{R}\times \mathbb{Z})$, for $j=1,2$. Thus, taking the Fourier transform to the integral formulation of \eqref{SHeq} at times $t_2>t_1$, we write
\begin{equation}\label{decomps1}
\begin{aligned}
\widehat{u}(t_2)=&e^{i(t_2-t_1)\omega(\xi,\eta)}\widehat{u}(\xi,\eta,t_1)-\int_{0}^{t_2-t_1}e^{i(t_2-t_1-\tau)\omega(\xi,\eta)}\Big(\sum_{k=1}^K \frac{\nu_k}{k+1} i\xi \widehat{u^{k+1}}(\xi,\eta,\tau)\Big)\, d \tau\\
=:&\,\mathcal{G}_1(u,t_1,t_2)+\mathcal{G}_2(u,t_1,t_2)+\mathcal{G}_3(u,t_1,t_2),
\end{aligned}    
\end{equation}
where  
\begin{equation*}
\begin{aligned}
\mathcal{G}_1(u,t_1,t_2)=e^{i(t_2-t_1)\omega(\xi,\eta)}\frac{\partial}{\partial \xi}\widehat{u}(0,\eta,t_1)\xi\phi(\xi)-\int_{0}^{t_2-t_1}e^{i(t_2-t_1-\tau)\omega(\xi,\eta)}\Big(\sum_{k=1}^K \frac{i\nu_k}{k+1} \widehat{u^{k+1}}(0,\eta,\tau)\xi\phi(\xi)\Big)\, d\tau,
\end{aligned}    
\end{equation*}
with $\phi(\xi)\in C^{\infty}_c(\mathbb{R})$ be such that $\phi(\xi)=1$ for all $|\xi|\leq 1$, we also define
\begin{equation*}
\begin{aligned}
\mathcal{G}_2(u,t_1,t_2)=&e^{i(t_2-t_1)\omega(\xi,\eta)}\Big(\widehat{u}(\xi,\eta,t_1)-\frac{\partial}{\partial \xi}\widehat{u}(0,\eta,t_1)\xi \phi(\xi)\Big)
\end{aligned}
\end{equation*}
and
\begin{equation*}
\begin{aligned}
\mathcal{G}_{3}(u,t_1,t_2)=&-\int_{0}^{t_2-t_1}e^{i(t_2-t_1-\tau)\omega(\xi,\eta)}\sum_{k=1}^K\frac{\nu_k}{k+1}\Big(i\xi \widehat{u^{k+1}}(\xi,\eta,\tau)-i\widehat{u^{k+1}}(0,\eta,\tau)\xi\phi(\xi)\Big)\, d\tau.
\end{aligned}
\end{equation*}

We first present some regularity conditions of $\mathcal{G}_2$ and $\mathcal{G}_3$.
\begin{claim}\label{Zerosolunique}
Let $\nu_k\in \mathbb{R}$, $k=1,\dots,K$. Consider $1<\theta<\frac{3}{2}$, $s_1\geq \frac{8\theta+6}{4\theta-3}$, $s_2> \frac{6\theta}{4\theta-3}$, and $u\in C([0,T];H^{s_1,s_2}(\mathbb{R}\times \mathbb{T}))\cap C([0,T];L^2(|x|^{2\theta}\,dx dy))$ be a solution of \eqref{SHeq} such that $\widehat{u}(0,\eta,t)=0$ for all $\eta\in \mathbb{Z}$, $t\in [0,T]$ and $u(t_1),u(t_2)\in L^2(|x|^{3^{+}}dx dy)$. Then it follows
\begin{equation*}
J_{\xi}^{\frac{3}{2}^{+}}\big(\mathcal{G}_j(u,t_1,t_2)\big)\in L^2(\mathbb{R}\times \mathbb{Z}),
\end{equation*}
for each $j=2,3$.
\end{claim}
To keep continuity to the main argument leading to the proof of Theorem \ref{Uniquecontprinc} $(ii)$, we will deduce Claim \ref{Zerosolunique} below.  By assumption at times $t_2,t_1$, Claim \ref{Zerosolunique} and \eqref{decomps1}, it follows that $J^{\frac{3}{2}^{+}}(\mathcal{G}_1(u,t_1,t_2))\in L^2(\mathbb{R}\times \mathbb{Z})$. Now, using \eqref{distribderiv1}, we compute the partial derivative of $\mathcal{G}_1(u,t_1,t_2)$ with respect to $\xi$ to get
\begin{equation}\label{decomps2}
 \begin{aligned}
\partial_{\xi}\mathcal{G}_1(u,t_1,t_2)=:\mathcal{G}_{1,1}(u,t_1,t_2)+\mathcal{G}_{1,2}(u,t_1,t_2)+\mathcal{G}_{1,3}(u,t_1,t_2),
 \end{aligned}   
\end{equation}
where 
\begin{equation*}
\begin{aligned}
\mathcal{G}_{1,1}(u,t_1,t_2)=e^{i(t_2-t_1)\omega(\xi,\eta)}\frac{\partial}{\partial \xi}\widehat{u}(0,\eta,t_1)\phi(\xi)-\int_0^{t_2-t_1}e^{i(t_2-t_1-\tau)\omega(\xi,\eta)}\Big(\sum_{k=1}^K 
    \frac{i \nu_k}{k+1}\widehat{u^{k+1}}(0,\eta,\tau)\phi(\xi)\Big)\, d\tau,
\end{aligned}    
\end{equation*}
\begin{equation*}
\begin{aligned}
\mathcal{G}_{1,2}(u,t_1,t_2)= e^{i(t_2-t_1)\omega(\xi,\eta)}\frac{\partial}{\partial \xi}\widehat{u}(0,\eta,t_1)(2i(t_2-t_1)|\xi|)\xi\phi(\xi)+e^{i(t_2-t_1)\omega(\xi,\eta)}\frac{\partial}{\partial \xi}\widehat{u}(0,\eta,t_1)\xi \frac{d}{d\xi}\phi(\xi),
\end{aligned}    
\end{equation*}
and
\begin{equation*}
\begin{aligned}
\mathcal{G}_{1,3}(u,t_1,t_2)= \int_{0}^{t_2-t_1}e^{i(t_2-t_1-\tau)\omega(\xi,\eta)}\Big(\sum_{k=1}^K \frac{i\nu_k}{k+1}\widehat{u^{k+1}}(0,\eta,\tau)\Big)\Big((2i(t_2-t_1-\tau)|\xi|)\xi\phi(\xi)+\xi \frac{d}{d\xi}\phi(\xi)\Big)\, d\tau.
\end{aligned}    
\end{equation*}
We have the following claim.
\begin{claim}\label{Zerosolunique2} Under the same hypothesis of Claim \ref{Zerosolunique}, we have
\begin{equation*}
J_{\xi}^{\frac{1}{2}^{+}}\big(\mathcal{G}_{1,j}(u,t_1,t_2)\big)\in L^2(\mathbb{R}\times \mathbb{Z}),
\end{equation*}
for each $j=2,3$.
\end{claim}
We deduce Claim \ref{Zerosolunique2} below. Consequently, given that $J^{\frac{3}{2}^{+}}(\mathcal{G}_1(u,t_1,t_2))\in L^2(\mathbb{R}\times \mathbb{Z})$, Sobolev embedding establishes that $\partial_{\xi}\mathcal{G}_1(u,t_1,t_2)$ is a continuous function in $\xi$ for each $\eta$, this fact combined with Claim \ref{Zerosolunique2} and decomposition \eqref{decomps2} give us $\mathcal{G}_{1,1}(u,t_1,t_2)$ is a continuous function of $\xi$. In particular, for each $\eta\in \mathbb{Z}$,
\begin{equation*}
\begin{aligned}
    \lim_{\xi\to 0^{+}}\mathcal{G}_{1,1}(u,t_1,t_2)= \lim_{\xi\to 0^{-}}\mathcal{G}_{1,1}(u,t_1,t_2).
\end{aligned}
\end{equation*}
Now, using that  $\lim_{\xi \to 0^{+}}\omega(\xi,\eta)=\eta$, $\lim_{\xi \to 0^{-}}\omega(\xi,\eta)=-\eta$ and $\phi(\xi)=1$ in a neighborhood of the origin, the continuity statement above shows that
\begin{equation*}
\begin{aligned}
    \sin((t_2-t_1)\eta^2)\frac{\partial}{\partial \xi}\widehat{u}(0,\eta,t_1)=\int_0^{t_2-t_1}\sin((t_2-t_1-\tau)\eta^2)\Big(\sum_{k=1}^K 
    \frac{i \nu_k}{k+1}\widehat{u^{k+1}}(0,\eta,\tau)\Big)\, d\tau,
\end{aligned}    
\end{equation*}
for all $\eta\in \mathbb{Z}$. This completes the deduction of \eqref{Uniqueident2} in Theorem \ref{Uniquecontprinc} $(ii)$. Now, assuming that $u\in L^{\infty}([0,T_1];L^2(|x|^{3^+} dx dy))$ with $u(0)=u_0\in L^2(|x|^{3+}dxdy)$ for some $0<T_1\leq T$, we have from the previous identity that
\begin{equation}\label{identityzerotime}
\begin{aligned}
    \sin(t\eta^2)\frac{\partial}{\partial \xi}\widehat{u_0}(0,\eta)=\int_0^{t}\sin((t-\tau)\eta^2)\Big(\sum_{k=1}^K 
    \frac{i \nu_k}{k+1}\widehat{u^{k+1}}(0,\eta,\tau)\Big)\, d\tau,
\end{aligned}    
\end{equation}
for all $\eta \in \mathbb{Z}\setminus\{0\}$ and $t\in [0,T_1]$ such that $u(t)\in L^2(|x|^{3^{+}}dxdy)$. Now using that $u\in C([0,T];H^{s_1,s_2}(\mathbb{R}\times \mathbb{T}))\cap C([0,T];L^2(|x|^{2\theta}\,dx dy))$ for some $1<\theta<\frac{3}{2}$, we can apply Sobolev embedding $H^{\frac{1}{2}^{+}}(\mathbb{R})\hookrightarrow L^{\infty}(\mathbb{R})$ in the $\xi$-variable, together with the embedding $L^{2}(\mathbb{Z})\hookrightarrow L^{\infty}(\mathbb{Z})$ in the $\eta$-variable to deduce that the map $\tau\mapsto \sum_{k=1}^K \frac{\nu_k}{k+1}\widehat{u^{k+1}}(0,\eta,\tau)$ is continuous. It turns out that we can multiply \eqref{identityzerotime} by $t^{-1}$, then taking $t\to 0$, it follows
\begin{equation*}
 \eta^2\frac{\partial}{\partial \xi}\widehat{u_0}(0,\eta) =0,  
\end{equation*}
for all $\eta\in \mathbb{Z}\setminus\{0\}$. This completes the deduction of Theorem \ref{Uniquecontprinc} $(ii)$. It remains to deduce Claims \ref{Zerosolunique} and \ref{Zerosolunique2}.

\begin{proof}[Proof of Claim \ref{Zerosolunique}] Roughly, to get the regularity condition stated in Claim \ref{Zerosolunique}, the idea is to apply Theorem \ref{LinearEst} part $(iv)$. We consider the function
\begin{equation*}
F_{1}(\xi,\eta):=\widehat{u}(\xi,\eta,t_1)-\frac{\partial}{\partial \xi}\widehat{u}(0,\eta,t_1)\xi \phi(\xi).    
\end{equation*}
Since $u(t_1)\in L^2(|x|^{3^{+}}dx dy)$, we have $J^{\frac{3}{2}^{+}}_{\xi}\widehat{u}(t_1)\in L^2(\mathbb{R}\times \mathbb{Z})$. Using this fact, we get $\mathcal{F}^{-1}(F_{1})(x,y)\in H^{\frac{3}{2}^{+},0}(\mathbb{R}\times \mathbb{T})\cap L^2(|x|^{3^{+}} dx dy)$. Additionally, since  $F_{1}(0,\eta)=0$ and $\partial_{\xi}F_{1}(0,\eta)=0$ for all $\eta\in \mathbb{Z}$, the previous results show that $\mathcal{F}^{-1}(F_{1})(x,y)$ satisfies the hypothesis of Theorem \ref{LinearEst} $(iv)$. Hence, since
\begin{equation*}
\mathcal{F}^{-1}(\mathcal{G}_2(u,t_1,t_2))=U(t_2-t_1)\mathcal{F}^{-1}(F_1),    
\end{equation*}
Theorem \ref{LinearEst} $(iv)$ yields $\mathcal{F}^{-1}(\mathcal{G}_2(u,t_1,t_2))\in H^{\frac{3}{2}^{+},0}(\mathbb{R}\times \mathbb{T})\cap L^2(|x|^{3^{+}} dx dy)$, from which $J_{\xi}^{\frac{3}{2}^{+}}\big(\mathcal{G}_2(u,t_1,t_2)\big)\in L^2(\mathbb{R}\times \mathbb{Z})$.

On the other hand, given $k=1,\dots,K$ and $1<\theta<\frac{3}{2}$, we let $\theta_1=\frac{4\theta+9}{8}$. It follows that $\theta_1>\frac{3}{2}$ and 
\begin{equation}\label{eqfinalclaim1}
  u^k\partial_x u\in C([0,T];H^{\theta_1,0}(\mathbb{R}\times \mathbb{T}))\cap C([0,T];L^2(|x|^{2\theta_1}\, dx dy)).  
\end{equation}
Setting $\theta_2=\frac{4\theta+3}{4}$, the proof of the previous statement follows by similar arguments in the inference of \eqref{eqinclaim1}, thus we omit its deduction. We consider the function
\begin{equation*}
F_{2,k}(\xi,\eta,\tau):=i\xi \widehat{u^{k+1}}(\xi,\eta,\tau)-i\widehat{u^{k+1}}(0,\eta,\tau)\xi\phi(\xi),    
\end{equation*}
$\xi\in \mathbb{R}$, $\eta\in \mathbb{Z}$, $\tau\in (0,t_2-t_1)$. Given that $\phi\in C^{\infty}_c(\mathbb{R})$ and $\phi$ is equal to $1$ around the origin, we have that $F_{2,k}(0,\eta,\tau)=0$ and $\partial_{\xi}F_{2,k}(0,\eta,\tau)=0$. Moreover, by \eqref{eqfinalclaim1}, it follows $\mathcal{F}^{-1}(F_{2,k})(x,y,\tau)$ satisfies the hypothesis of Theorem \ref{LinearEst} $(iv)$. Consequently, we can argue as in \eqref{integraleqinweights} to conclude $J_{\xi}^{\theta_1}\big(\mathcal{G}_3(u,t_1,t_2)\big)\in L^2(\mathbb{R}\times \mathbb{Z})$.
\end{proof}

\begin{proof}[Proof of Claim \ref{Zerosolunique2}]

We first notice that Sobolev embedding in the $\xi$-variable implies 
\begin{equation*}
\|\frac{\partial}{\partial \xi}\widehat{u}(0,\eta,t_1)\|_{L^2_{\eta}}\lesssim  \|J_{\xi}^{\frac{3}{2}^{+}}\widehat{u}(\xi,\eta,t_1)\|_{L^2_{\xi,\eta}},
\end{equation*}
and together with Plancherel's identity, it is seen that
\begin{equation*}
\begin{aligned}
  \|\widehat{u^{k+1}}(0,\eta,\tau)\|_{L^2_{\eta}}\lesssim  \|J_{\xi}^{\frac{1}{2}^{+}}\widehat{u^{k+1}}(\tau)\|_{L^2_{\xi,\eta}}=&\|\langle x\rangle^{\frac{1}{2}^{+}}u^{k+1}(\tau)\|_{L^2}\\
\lesssim & \|u(\tau)\|_{L^{\infty}}^k\|\langle x\rangle^{\frac{1}{2}^{+}} u(\tau) \|_{L^2}.  
\end{aligned}
\end{equation*}
Now, we define
\begin{equation*}
\begin{aligned}
F_1(\xi,\eta)=\frac{\partial}{\partial \xi}\widehat{u}(0,\eta,t_1)(2i(t_2-t_1)|\xi|)\xi\phi(\xi)+\frac{\partial}{\partial \xi}\widehat{u}(0,\eta,t_1)\xi \frac{d}{d\xi}\phi(\xi)
\end{aligned}    
\end{equation*}
and
\begin{equation*}
\begin{aligned}
F_{2}(\xi,\eta,\tau)=\Big(\sum_{k=1}^K \frac{i\nu_k}{k+1}\widehat{u^{k+1}}(0,\eta,\tau)\Big)\Big((2i(t_2-t_1-\tau)|\xi|)\xi\phi(\xi)+\xi \frac{d}{d\xi}\phi(\xi)\Big),
\end{aligned}    
\end{equation*}
$\xi\in \mathbb{R}$, $\eta\in \mathbb{Z}$, $\tau\in (0,t_2-t_1)$. Using the regularity and decay properties of $\widehat{u}(t_2)$, $\widehat{u}$, $\phi$, and the inequalities deduced at the beginning of the proof, we get that $\mathcal{F}^{-1}(F_1)$ and $\mathcal{F}^{-1}(F_2)$ satisfy the hypothesis of Theorem \ref{LinearEst} $(iii)$. Thus, this theorem and familiar arguments to those in Claim \ref{Zerosolunique} show $J_{\xi}^{\frac{1}{2}^{+}}\big(\mathcal{G}_{1,j}(u,t_1,t_2)\big)\in L^2(\mathbb{R}\times \mathbb{Z})$, for each $j=2,3$.
\end{proof}

%%%%%%%%%%%%%%%%%%%%%%%%%%%%%%%%%%%%%%%%%%%%%%%%%%%%%%%%%%%%%%%%%%%%%%%%%%%%%%%%%%%%%%%%%%%%%%%%%%%%%%%%%%%%%%%%%%%%%%%%%%%%%%%%%%%%%%%%%%%%%%%%%%%%%%%%%%%%

\bibliographystyle{acm}
	%	\nocite{*}
{\small  \bibliography{bibli}}

\end{document}